\documentclass[12pt]{article}

\usepackage[all]{xy}

\usepackage[latin1]{inputenc}
\usepackage{amsfonts}
\usepackage{amsmath}
\usepackage{amssymb}
\usepackage{amscd}
\usepackage{amsthm}
\usepackage{indentfirst}
\usepackage[hmargin=2.83cm
,vmargin=4.0cm
]{geometry}

\usepackage{faktor}

\usepackage{epsfig}

\newtheorem{thm}{Theorem}[section]
\newtheorem{lemma}[thm]{Lemma}
\newtheorem{conj}[thm]{Conjecture}
\newtheorem{prop}[thm]{Proposition}
\newtheorem{coroll}[thm]{Corollary}
\newtheorem{claim}[thm]{Claim}

\theoremstyle{definition}

\newtheorem{defin}[thm]{Definition}
\newtheorem{rem}[thm]{Remark}
\newtheorem{exam}[thm]{Example}

\newtheorem{notation}[thm]{Notation}

\newtheorem*{convention}{Convention}

\newcommand{\R}{{\mathbb{R}}}
\newcommand{\F}{{\mathbb{F}}}

\newcommand{\T}{{\mathbb{T}}}
\newcommand{\Z}{{\mathbb{Z}}}
\newcommand{\N}{{\mathbb{N}}}
\newcommand{\C}{{\mathbb{C}}}

\newcommand{\cA}{{\mathcal{A}}}
\newcommand{\cB}{{\mathcal{B}}}
\newcommand{\cC}{{\mathcal{C}}}
\newcommand{\cD}{{\mathcal{D}}}

\newcommand{\cF}{{\mathcal{F}}}
\newcommand{\cG}{{\mathcal{G}}}
\newcommand{\cH}{{\mathcal{H}}}
\newcommand{\cI}{{\mathcal{I}}}

\newcommand{\cK}{{\mathcal{K}}}
\newcommand{\cL}{{\mathcal{L}}}
\newcommand{\cM}{{\mathcal{M}}}

\newcommand{\cP}{{\mathcal{P}}}
\newcommand{\cQ}{{\mathcal{Q}}}
\newcommand{\cR}{{\mathcal{R}}}
\newcommand{\cS}{{\mathcal{S}}}
\newcommand{\cT}{{\mathcal{T}}}
\newcommand{\cU}{{\mathcal{U}}}

\newcommand{\cX}{{\mathcal{X}}}

\newcommand{\fc}{{:\ }}

\newcommand{\ol}{\overline}
\newcommand{\wt}{\widetilde}
\newcommand{\wh}{\widehat}

\newcommand{\tb}{\textbf}

\DeclareMathOperator{\sgn}{sgn}

\DeclareMathOperator{\Crit}{Crit}

\DeclareMathOperator{\im}{im}
\DeclareMathOperator{\id}{id}
\DeclareMathOperator{\area}{area}
\DeclareMathOperator{\Vol}{Vol}
\DeclareMathOperator{\PSS}{PSS}

\DeclareMathOperator{\cl}{cl}

\DeclareMathOperator{\rk}{rk}

\DeclareMathOperator{\pr}{pr}
\DeclareMathOperator{\uw}{uw}

\DeclareMathOperator{\pt}{pt}

\DeclareMathOperator{\ini}{ini}
\DeclareMathOperator{\ter}{ter}
\DeclareMathOperator{\tel}{tel}
\DeclareMathOperator{\cone}{cone}
\DeclareMathOperator{\co}{co}

\DeclareMathOperator{\Tor}{Tor}

\DeclareMathOperator{\Ham}{Ham}
\DeclareMathOperator{\Symp}{Symp}

\DeclareMathOperator{\Int}{Int}

\DeclareMathOperator{\diag}{diag}

\DeclareMathOperator{\res}{res}

\begin{document}

\title{Symplectic topology and ideal-valued measures}

\author{Adi Dickstein$^{1,3}$, Yaniv Ganor$^2$, Leonid Polterovich$^3$, and Frol Zapolsky}

\footnotetext[1]{Partially supported by the Milner Foundation}
\footnotetext[2]{Partially supported by the Israel Science Foundation grant 1715/18, partially supported by
Technion scholarship funds, and supported in part at the Technion by a fellowship from the Lady Davis Foundation}
\footnotetext[3]{Partially supported by the Israel Science Foundation grant 1102/20}

\date{}

\setcounter{tocdepth}{2}

\renewcommand{\labelenumi}{(\roman{enumi})}

\maketitle

\begin{abstract}
We adapt Gromov's notion of ideal-valued measures to symplectic topology, and use it for
proving new results on symplectic rigidity and symplectic intersections. Furthermore, it allows us to discuss three ``big fiber theorems''---the Centerpoint Theorem in combinatorial geometry, the Maximal Fiber Inequality in topology, and the Non-displaceable Fiber Theorem in symplectic topology---from a unified viewpoint. Our main technical tool is an enhancement of the  symplectic cohomology theory recently developed by Varolg\"une\c s.
\end{abstract}

\tableofcontents

\section{Introduction and main results}

\subsection{Three big fiber theorems}

In various fields of mathematics there exist ``big fiber'' theorems, which are of the following type:

\begin{center}
  \emph{For any map $f \fc X \to Y$ in a suitable class there is $y_0 \in Y$\\ such that the fiber $f^{-1}(y_0)$ is big.}
\end{center}

The ``suitable class'' and ``big'' have different interpretations in different fields. Here we will present three results which exemplify this principle.

\begin{thm} [{{\bf Maximal fiber inequality}, Gromov \cite[p.758]{gromov2009singularities}, \cite[p.425]{gromov2010singularities}}] \label{thm:Gromov_torus}
Let $Y$ be a metric space of covering dimension $d$, and let $p,n$ be positive integers such that $n \geq p(d+1)$. Then for any continuous map $f \fc \T^n \to Y$ there is $y_0 \in Y$ such that
$$\rk \big(\check H^*(\T^n) \to \check H^*(f^{-1}(y_0))\big) \geq 2^p\,.$$
\end{thm}
\noindent Here $\check H^*$ stands for \v Cech cohomology with field coefficients. The suitable class consists of continuous maps into metric spaces of a given covering dimension and the result is that there is a ``big'' fiber, namely the restriction map from the cohomology of the ambient space $\T^n$ to that of the fiber has large rank.

\begin{thm} [{{\bf Topological centerpoint theorem}, Karasev, \cite[Theorem 5.1]{karasev2014covering}}] \label{thm:top_Tverberg} Let $Y$ be a metric space of covering dimension $d$, let $p$ be a positive integer, and put $n = p(d+1)$. Then for any continuous map $f \fc \Delta^n \to Y$, where $\Delta^n$ is the $n$-simplex, there is $y_0 \in Y$ such that $f^{-1}(y_0)$ intersects all the $pd$-dimensional faces of $\Delta^n$.
\end{thm}
\noindent Here we have continuous maps from simplexes into metric spaces of a given covering dimension, while a fiber is ``big'' if it intersects all the high-dimensional faces of the simplex. The affine
version of this theorem can be found in \cite{edelsbrunner2012algorithms}, a classical result
proved in a slightly different language by Rado, 1946 \cite{rado1946theorem}.

Our final sample result belongs to the field of symplectic topology. Recall that a \tb{symplectic manifold} is a pair $(M,\omega)$ where $M$ is a manifold and $\omega$ is a closed $2$-form on $M$ which is also nondegenerate, meaning that $\omega^{\wedge \frac 1 2 \dim M}$ is a volume form. Given $f \in C^\infty(M)$ its \tb{Hamiltonian vector field} $X_f$ is uniquely determined by the equation $\iota_{X_f}\omega = -df$. The \tb{Poisson bracket} of $f,g \in C^\infty(M)$ is the function $\{f,g\} = - \omega(X_f,X_g) = df(X_g)$; $f,g$ \tb{Poisson commute} if $\{f,g\}=0$. If $f \in C^\infty(M\times[0,1])$ is a time-dependent function, then integrating the Hamiltonian vector field $X^t_f$ of $f_t\equiv f(\cdot,t)$ yields a \tb{Hamiltonian isotopy} $\phi_f^t$. The set of $\phi_f^1$ for all such $f$ is the \tb{Hamiltonian group} $\Ham(M,\omega)$ of $(M,\omega)$. A set $S \subset M$ is \tb{displaceable} if there is $\phi \in \Ham(M,\omega)$ such that $\phi(S) \cap \ol S = \varnothing$, and \tb{non-displaceable} otherwise.

Let us describe the suitable class of maps defined on symplectic manifolds.
\begin{defin}\label{def:involutive}
  Let $(M,\omega)$ be a symplectic manifold and let $B$ be a smooth manifold. We call a smooth map $\pi \fc M \to B$ \tb{involutive} if for all $f,g \in C^\infty(B)$ we have $\{\pi^*f,\pi^*g\} \equiv 0$.
\end{defin}

\begin{thm} [{{\bf Non-displaceable fiber theorem},  Entov--Polterovich, \cite{entov2006quasi}}]\label{thm:EP_nondisplaceable_fiber} Let $(M,\omega)$ be a closed symplectic manifold. Then for any involutive map $\pi \fc M \to B$ there is $b_0 \in B$ such that $\pi^{-1}(b_0)$ is non-displaceable.
\end{thm}
\noindent Non-displaceable sets in a symplectic manifold are ``big,'' so this can be interpreted as a ``big fiber'' theorem in symplectic topology.

\paragraph*{Concepts used in the proofs of the above theorems.} Both Gromov's maximal fiber inequality and Karasev's topological centerpoint theorem can be proved using Gromov's notion of \emph{ideal-valued measures} coming from \emph{\v Cech and singular cohomology}. Ideal-valued measures are the subject of Definition \ref{def:IVM} (Section \ref{ss:IVMs} below), and they are used to prove the aforementioned results in Section \ref{s:abstr_centerpt_thm_IVMs}. By contrast, Entov--Polterovich's result is proved using \emph{partial symplectic quasi-states} defined via \emph{Floer homology}, which are seemingly unrelated concepts.

Our idea is to unify the two approaches, using a generalization of ideal-valued measures to what we call \emph{symplectic ideal-valued quasi-measures} (Definition \ref{def:IVQM}), defined on symplectic manifolds. Our main result, Theorem \ref{thm:QH_IVQM_is_IVQM}, states that any closed symplectic manifold carries such an object, which also crucially satisfies some additional properties, and which arises from U.\ Varolg\"une\c s's \emph{relative symplectic cohomology} \cite{varolgunes2018mayer}. Armed with these, we
\begin{itemize}
  \item Refine the non-displaceable fiber theorem, see Theorems \ref{thm:quantitative_nondisplaceable_fiber}, \ref{cor:quant_nondisplaceable_fiber_QH_IVQM};
  \item Prove a symplectic version of the centerpoint theorem, see Theorem \ref{thm:symplectic_Tverberg}, and use it to produce a new example of rigid symplectic intersections, Theorem \ref{thm:torus_cross};
  \item Define \emph{SH-heavy subsets} of a symplectic manifold, which are a variant of Entov--Polterovich's notion of heavy subsets \cite{entov2009rigid}, and provide a simple algebraic criterion which guarantees that two SH-heavy sets intersect, see Definition \ref{def:SH_heavy} and Proposition \ref{prop:alg_crit_SH_heavy_nondispl}. In Section \ref{ss:SH_heavy_sets} we address a question about the connection between SH-heavy and heavy sets and prove that under certain assumptions heavy sets are necessarily SH-heavy.
\end{itemize}
\begin{rem}\label{rem:meridians_in_torus}
  In \cite{entov2009rigid} it is shown that a heavy set is non-displaceable, but beyond that it was unclear when two heavy sets should intersect. For instance, if $L,L' \subset \T^2$ are meridians, then both of them are heavy, but they may or may not be displaceable from one another. The two situations are when $L,L'$ lie in distinct homology classes versus when they represent the same homology class and have zero geometric intersection number. Using our notion of SH-heavy sets and the intersection criterion, we are able to distinguish between the two situations, see Example \ref{ex:transverse_Lag_tori_nondispl}. \end{rem}

\subsection{Gromov's ideal-valued measures: a review}\label{ss:IVMs}

In this section we review Gromov's notion of ideal-valued measures, and introduce the main example, the so-called cohomology ideal-valued measure. This will be used in Section \ref{s:abstr_centerpt_thm_IVMs} to prove Theorems \ref{thm:Gromov_torus} and \ref{thm:top_Tverberg}.

An algebra $(A,*)$ over a field is called \tb{$\Z_{2k}$-graded} if it decomposes as a direct sum $A = \bigoplus_{i\in\Z_{2k}} A^i$ of graded components, where $k$ is a nonnegative integer, such that $A^i*A^j \subset A^{i+j}$ for all $i,j$. We say that $A$ is \tb{skew-commutative} if for homogeneous $a,b \in A$ we have $ab = (-1)^{|a||b|}ba$, where $|\cdot|$ denotes the degree. In the rest of the paper by an \emph{algebra} we mean a graded skew-commutative associative unital algebra. Note that if $k=0$, we obtain a $\Z$-graded algebra. A typical example is the cohomology ring of a space.

For future use note that if $A$ is a $\Z$-graded algebra, then we can produce a $\Z_{2k}$-graded algebra $A^{*\bmod 2k}$, called its \tb{mod $2k$ regrading}, as follows:
$$(A^{*\bmod 2k})^{[i]} = \bigoplus_{j\equiv i \bmod 2k}A^j\quad\text{for }[i] \in \Z_{2k}\,.$$

An ideal $I\subset A$ is \tb{graded} if it decomposes as the direct sum of its graded components, that is $I = \bigoplus_i (I \cap A^i)$. Equivalently, $I$ is the kernel of a graded algebra morphism $A \to B$, that is an algebra morphism mapping $A^i\to B^i$ for all $i$. Note that in skew-commutative algebras left, right, and two-sided ideals are equivalent notions. We let $\cI(A)$ be the collection of graded ideals of $A$. We say that $A$ is \tb{graded Noetherian} if every ascending sequence of graded ideals stabilizes. This is the case, for instance, if $A$ is Noetherian, and in particular if it is finite-dimensional.

\begin{defin} (Gromov, \cite[Section 4.1]{gromov2010singularities})\label{def:IVM}
  Let $X$ be a topological space and let $A$ be an algebra. An \tb{$A$-ideal-valued measure (an $A$-IVM)} is an assignment $U \mapsto \mu(U) \in \cI(A)$, where $U\subset X$ runs over open sets, such that the following properties hold:
  \begin{enumerate}
    \item \tb{(normalization):} $\mu(\varnothing) = 0$ and $\mu(X) = A$;
    \item \tb{(monotonicity):} $\mu(U) \subset \mu(U')$ if $U \subset U'$;
    \item \tb{(continuity):} if $U_1\subset U_2\subset\dots$ and $U = \bigcup_i U_i$, then $\mu(U) = \bigcup_i\mu(U_i)$;
    \item \tb{(additivity):} $\mu(U\cup U') = \mu(U) + \mu(U')$ for disjoint $U,U'$;
    \item \tb{(multiplicativity):} $\mu(U)*\mu(U') \subset \mu(U \cap U')$;
    \item \tb{(intersection):} if $U,U'$ cover $X$, then $\mu(U\cap U') = \mu(U) \cap \mu(U')$.
  \end{enumerate}
\end{defin}

\begin{rem}
  The condition $\mu(X) = A$ is called fullness in \cite{gromov2010singularities}, however we opted to include it as part of the normalization condition because we will only use IVMs which satisfy it. Also, in Gromov's definition additivity and intersection are stated for countably many sets.
\end{rem}

\begin{rem}\label{rem:intersection_axiom_multiple_sets}
  Note that the intersection property only holds for covers of $X$ by \emph{two} subsets. The generalization to multiple sets is as follows: if $U_1,\dots, U_k\subset X$ cover $X$ pairwise, that is for each $i\neq j$, $U_i\cup U_j=X$, then $\bigcap_i\mu(U_i)=\mu\big(\bigcap_i U_i\big)$.
\end{rem}

\begin{rem}\label{rem:regularization_IVM}
  Given an $A$-IVM $\mu$ on a compact Hausdorff space we will use its natural extension to compact sets defined by
  $$\mu(K) = \bigcap_{U\text{ open}\supset K}\mu(U)\,.$$
  This extended function then satisfies the analogs of the monotonicity, multiplicativity, and intersection properties. Note that if $X$ is in addition a $G_\delta$-space\footnote{A set in a topological space is $G_\delta$ if it is a countable intersection of open sets. A space is $G_\delta$ if each closed set is $G_\delta$. In our case, since $X$ is compact and Hausdorff, being $G_\delta$ implies, by passing to complements, that each open set is a countable union of compacts, whence the conclusion.}, then the values of $\mu$ on open sets are recoverable from those on compact sets via $\mu(U) = \bigcup_{K\text{ compact}\subset U}\mu(K)$.
\end{rem}

\begin{exam}\label{ex:trivial_IVM}
  Given any algebra $A$ and a compact connected Hausdorff space $X$, the \tb{trivial $A$-IVM} on $X$ is given by $\mu(U) = 0$ for all $U \subsetneq X$ and $\mu(X) = A$.
\end{exam}

\begin{exam}\label{ex:coh_IVM}
  Here we describe two fundamental examples of IVMs: the \emph{\v Cech cohomology IVM} and the \emph{singular cohomology IVM}.
  \begin{itemize}
    \item Let $\check H^*$ denote \v Cech cohomology with coefficients in a fixed field. Letting $X$ be a compact Hausdorff space and $A = \check H^*(X)$ and putting
    $$\mu(U) = \ker \big(\check H^*(X) \to \check H^*(X\setminus U)\big)$$
    for open $U\subset X$ yields an IVM, called the \tb{\v Cech cohomology IVM} (Gromov, \cite[Section 4.1]{gromov2010singularities}). Note that $A$ is $\Z$-graded.
    \item In this example $H^*$ stands for singular cohomology with fixed field coefficients. We will describe the \tb{singular cohomology IVM} on a compact Hausdorff space $X$. Let $A = H^*(X)$ and note that $A$ is likewise $\Z$-graded. The idea is to take the construction of the previous example and \emph{regularize} it: for compact $K \subset X$ we put
$$\mu(K) = \bigcap_{U\text{ open}\supset K}\ker\big(H^*(X) \to H^*(X \setminus  U)\big)\,,$$
and for open $U \subset X$ we let
$$\mu(U) = \bigcup_{K\text{ compact}\subset U}\mu(K)\,.$$
  \end{itemize}
\end{exam}

\begin{rem}
  \begin{itemize}
    \item Regularization refers to the standard approximation of compact sets by open sets and of open sets by compacts. It is needed for the continuity property of the singular cohomology IVM.
    \item The reason we first define the singular cohomology IVM for compact sets is to make it similar to our definition of ideal-valued quasi-measures based on Varolg\"une\c s's relative symplectic cohomology, where we must use restriction maps to compact sets, see Definition \ref{def:QH_IVQM}.
    \item If $X$ is a closed manifold, then the singular cohomology and the \v Cech cohomology IVMs coincide, because singular cohomology and \v Cech cohomology coincide on codimension zero compact submanifolds of $X$ with boundary, and any compact set $K$ can be approximated by such submanifolds containing $K$ in their interior. Henceforth we will refer to either IVM on $X$ as the \tb{cohomology IVM}.
  \end{itemize}
\end{rem}

IVMs provide a conceptual framework in which to prove Karasev's topological centerpoint theorem \ref{thm:top_Tverberg}.
Similarly, Gromov's Theorem \ref{thm:Gromov_torus} can be proved using a general result about IVMs, Theorem \ref{thm:Gromov_central_point}, also due to Gromov \cite[Section 4.2]{gromov2010singularities}.

\subsection{Symplectic ideal-valued quasi-measures}\label{ss:sympl_IVQMs}

Here we present a new notion, \emph{symplectic ideal-va\-lu\-ed quasi-measures}. These are a suitable generalization of IVMs in the setting of symplectic manifolds, and are central to the results of the present paper. Throughout this section, $(M,\omega)$ is a fixed closed symplectic manifold.

\begin{defin}\label{def:commuting_sets}
  We say that two compact subsets $K,L \subset M$ \tb{commute} if there are Poisson-commuting $f,g \in C^\infty(M,[0,1])$ such that $K = f^{-1}(0), L = g^{-1}(0)$. Two open sets commute if their complements commute.
\end{defin}

\begin{rem}\label{rem:preimages_involutive_maps}
  It is not hard to see that two closed (respectively, open) sets commute if and only if they are the preimages of two closed (respectively, open) subsets by an involutive map $M\to B$.
\end{rem}

\begin{exam}\label{ex:disjoint_sets_commute}We leave it as a nice exercise for the reader to show that any two disjoint open sets commute, as do any two disjoint compact sets.
\end{exam}

\begin{defin}\label{def:IVQM}Fix an algebra $A$.
  A \tb{symplectic $A$-ideal-valued quasi-measure (symplectic $A$-IVQM)} on $(M,\omega)$ is an assignment $U \mapsto \tau(U) \in \cI(A)$, where $U$ ranges over open subsets of $M$, such that $\tau$ satisfies all of the properties of an $A$-IVM, with the exception of the multiplicativity axiom, which is replaced by the weaker
  \begin{itemize}
    \item \tb{(quasi-multiplicativity):} $\tau(U) * \tau(U') \subset \tau(U \cap U')$ whenever $U,U'$ commute.
  \end{itemize}
\end{defin}

\begin{rem}\label{rem:remarks_after_def_IVQM}
  \begin{itemize}
    \item Any $A$-IVM on $M$ is also a symplectic $A$-IVQM.
    \item For brevity, from this point on we will usually refer to symplectic IVQMs simply as IVQMs.
    \item Analogously to IVMs, we extend an $A$-IVQM $\tau$ to compact sets by
         $$\tau(K) = \bigcap_{U\text{ open}\supset K}\tau(U)\,.$$
         This extension satisfies the analogs of the monotonicity, quasi-mul\-ti\-pli\-cativity, and intersection properties.
    \item In the concrete examples of IVQMs we will construct the quasi-multiplicativity actually holds for more general pairs of subsets, see Sections \ref{ss:axioms_rel_SH_pairs}, \ref{sss:accel_data_descent}, but the stated property is easier to formulate, and it suffices for applications.
  \end{itemize}

\end{rem}

In what follows, $\Symp_0(M,\omega)$ stands for the identity component of the group of symplectomorphisms of $(M,\omega)$.
\begin{defin}For an $A$-IVQM $\tau$ on $(M,\omega)$ we define two further properties:
  \begin{itemize}
    \item \tb{(invariance):} if $\phi\in \Symp_0(M,\omega)$, then $\tau(\phi(U)) = \tau(U)$ for all $U$.
    \item \tb{(vanishing):} if $K\subset M$ is a displaceable compact set, then $\tau(M \setminus K) = A$, and there is an open $U$ such that $K\subset U$ and $\tau(U) = 0$.
  \end{itemize}
\end{defin}

\begin{rem}
  \begin{itemize}
    \item If $\tau$ is an IVQM which satisfies the vanishing property, then its extension to compacts, as in Remark \ref{rem:remarks_after_def_IVQM}, satisfies $\tau(K) = 0$ if $K$ is a displaceable compact set.
    \item Let now $\tau$ be an arbitrary $A$-IVQM, where $A$ is finite-dimensional.\footnote{Or, more generally, graded Artinian, that is each descending chain of graded ideals stabilizes.} If $K\subset M$ is compact and $\tau(K) = 0$, then there is an open $U\supset K$ with $\tau(U)=0$. In particular for such algebras $A$ the second half of the vanishing property is equivalent to requiring $\tau(K)=0$ for all displaceable compact $K$.
  \end{itemize}
\end{rem}

\begin{rem}\label{rem:IVMs_are_IVQMs}
 The trivial $A$-IVM (see Example \ref{ex:trivial_IVM} above)  satisfies the invariance property. If $A \neq 0$ and $M$ has positive dimension, then the trivial $A$-IVM does not satisfy the vanishing property.
\end{rem}

In order to formulate our main result, we need to recall the notion of quantum cohomology of $M$, as well as the relative symplectic cohomology for compacts in $M$, recently introduced by U.\ Varolg\"une\c s \cite{varolgunes2018mayer}, \cite{Varolgunes_2018_PhD}.

Fix a base field $\F$, and recall that the corresponding \tb{Novikov field}\footnote{Sometimes referred to as ``the universal Novikov field.''} is
$$\Lambda = \left\{\sum_{i=0}^{\infty}c_iT^{\lambda_i}\,\Big|\, c_i \in \F\,,\R\ni\lambda_i \xrightarrow[i\to\infty]{} \infty\right\}\,.$$
The \tb{quantum cohomology of $M$ with coefficients in $\Lambda$} is additively the singular cohomology
$$QH^*(M) = H^{*\operatorname{mod} 2N_M}(M;\Lambda)$$
regraded modulo $2N_M$, where $N_M$ is the minimal Chern number of $M$. The product operation on $QH^*(M)$ is given by the quantum product \cite{mcduff2012j}. More specifically, we use its $T$-weighted version, given for $a,b \in H^*(M;\Lambda)$ of pure degree by
$$a*b=\sum_{\alpha\in H_2(M;\Z)}\text{PD}\big(\text{GW}_3^\alpha(a,b,\cdot)\big)\,T^{\langle [\omega],\alpha\rangle}\,,$$
where $\text{GW}_3^\alpha\,(\cdot,\cdot,\cdot)$ is the genus zero $3$-point Gromov--Witten invariant in class $\alpha$, $\text{PD}$ is the Poincar\'e duality, which transforms functionals on cohomology into cohomology classes, and $[\omega]$ is the de Rham cohomology class of $\omega$.

In \cite{varolgunes2018mayer} Varolg\"une\c s defined, for each compact $K\subset M$, its relative symplectic cohomology $SH^*(K;\Lambda)$. This invariant has the following properties, among others:
\begin{itemize}
  \item Each $SH^*(K;\Lambda)$ is a $\Z_{2N_M}$-graded unital associative skew-commutative algebra over $\Lambda$, and $SH^*(M;\Lambda) = QH^*(M)$ as an algebra \cite{tonkonog2020super};
  \item For $K\subset K'$ there is a \tb{restriction map} $\res^{K'}_K \fc SH^*(K';\Lambda) \to SH^*(K;\Lambda)$, which is a graded unital algebra morphism, such that $\res^K_K = \id$, and if $K \subset K'\subset K''$, then $\res^{K'}_K\circ \res^{K''}_{K'} = \res^{K''}_K$;
  \item The \tb{Mayer--Vietoris property}: if $K,K'$ commute, the restriction maps fit into a long exact sequence
  \begin{multline*}
  \dots \to SH^*(K\cup K';\Lambda) \xrightarrow{\big(\res^{K\cup K'}_K,\,\res^{K\cup K'}_{K'}\big)} SH^*(K;\Lambda) \oplus SH^*(K';\Lambda)\\ \xrightarrow{\res^K_{K\cap K'} - \res^{K'}_{K\cap K'}} SH^*(K\cap K';\Lambda) \xrightarrow{+1}\dots
  \end{multline*}

  \item If $K$ is displaceable, then $SH^*(K;\Lambda) = 0$.
\end{itemize}
\noindent See \cite{varolgunes2018mayer}, and Sections \ref{ss:quantum_coh_IVQM}, \ref{sss:accel_data_descent} (Definition \ref{def:rel_SH_subset}) for a precise definition of $SH^*(K;\Lambda)$.
\begin{rem}
  Note that $SH^*$ has a $\Z_{2N_M}$-grading, rather than merely a $\Z_2$-grading. The reason for this is that in our version of $SH^*$ we only use $1$-periodic Hamiltonian orbits which are contractible in $M$. See also Remark \ref{rem:only_contractible_orbits_SH}.
\end{rem}

We define a set function $\tau$ with values in the graded ideals of $QH^*(M)$, as follows:
$$\tau(K) = \bigcap_{U\text{ open}\supset K}\ker \big(\res^M_{M\setminus U} \fc QH^*(M)=SH^*(M;\Lambda) \to SH^*(M\setminus U;\Lambda)\big)$$
for compact $K \subset M$ and
$$\tau(U) = \bigcup_{K\text{ compact}\subset U} \tau(K)$$
for open $U \subset M$.

We can now formulate our main result.
\begin{thm}[{\textbf{Main theorem}}]\label{thm:QH_IVQM_is_IVQM}
  The function $\tau$ is a symplectic $QH^*(M)$-IVQM on $M$ satisfying the invariance and vanishing properties.
\end{thm}

\begin{defin}\label{def:QH_IVQM}
  The $QH^*(M)$-IVQM $\tau$ is called the \tb{quantum cohomology IVQM}.
\end{defin}

\noindent By far the most nontrivial property which needs to be proved for $\tau$ is quasi-multiplica\-ti\-vity. To this end we develop the tool of a relative symplectic cohomology of pairs, and the whole of Section \ref{s:proofs_properties_SH_pairs} is dedicated to the proof of its existence.

\begin{rem}
  Note that if $K$ is compact, then $\tau(K) = \bigcap_{U\text{ open}\supset K}\tau(U)$. In other words, had we used the values of $\tau$ on open sets and extended it to compact sets as in Remark \ref{rem:regularization_IVM}, we would have recovered the values of $\tau$ on compacts as defined above.
\end{rem}

\begin{rem}
  One of the key features of IVQMs is as follows, see Proposition \ref{prop:involutive_maps_push_IVQM_to_IVM}:
  \begin{center}
    \emph{The pushforward of an IVQM under an involutive map is an IVM}.
  \end{center}
  \noindent Pushforwards are the subject of Definition \ref{def:pushforward}. This proposition allows us to take the quantum cohomology IVQM on a given symplectic manifold, push it forward by an involutive map, obtain an IVM, and then apply to it the corresponding results relating to IVMs. See Section \ref{s:abstr_centerpt_thm_IVMs} for details.
\end{rem}

\noindent The proof of the main theorem uses Floer theory. However, on $S^2$, we can describe an IVQM in elementary terms, as follows:
\begin{prop}\label{prop:unique_A_IVQM_sphere}
  Fix a nonzero algebra $A$, and equip $S^2$ with an area form of total area $1$. Then there exists a unique $A$-IVQM $\tau$ on $S^2$, with the following property: if $D\subset S^2$ is an open disk with smooth boundary, then $\tau(D) = 0$ provided that $D$ has area $\leq \frac 1 2$, and $\tau(D) = A$ otherwise. It satisfies the invariance and vanishing properties.
\end{prop}
\begin{proof}[Sketch of proof.]
  First, we will show that the intersection, additivity, and continuity properties determine $\tau$ uniquely, given its values on disks with smooth boundary. If $U$ is a connected open set with smooth boundary, let $S^2\setminus U = D_1\cup\dots\cup D_k$ be the decomposition into connected components, which are closed disks with smooth boundary. Put $U_i=S^2\setminus D_i$. Then the sets $U_1,\dots,U_k$ are open disks with smooth boundary, and they form an open pairwise covering of $S^2$, that is for each $i\neq j$, $U_i\cup U_j=S^2$. It follows from the intersection axiom that $\tau(U) = \tau(U_1)\cap\dots\cap\tau(U_k)$, and since the values of $\tau(U_i)$ are given, this recovers the value $\tau(U)$. The additivity axiom recovers the values of $\tau$ on any open subset which has smooth boundary and finitely many connected components.

  Finally, if $U\subset S^2$ is any open set, it can be represented as an increasing union $U=\bigcup_{i\in\N}U_i$ of open sets with smooth boundary and finitely many components, and thus $\tau(U)$ is recovered from the continuity property: $\tau(U) = \bigcup_i \tau(U_i)$.

  To prove that this object is indeed an $A$-IVQM requires technical work using tools of $2$-dimensional geometry and topology. We emphasize that $\tau$ is only an $A$-IVQM and not an $A$-IVM. Indeed, if $K,L\subset S^2$ are two transversely intersecting equators, then $\tau(K)=\tau(L)=A$ but $\tau(K\cap L) = 0\not\supset A=\tau(K)*\tau(L)$, which shows that the multiplicativity property fails.
\end{proof}

\subsection{Quantitative non-displaceable fiber theorem}\label{ss:quantit_nd_fiber_thm}

In this section we formulate and prove a refinement of Theorem \ref{thm:EP_nondisplaceable_fiber}.
\begin{defin}\label{def:continuous_involutive_map}
  Let $(M,\omega)$ be a symplectic manifold and let $Y$ be a topological space. We call a continuous map $f \fc M \to Y$ \tb{involutive} if there is a (smooth) involutive map $\pi \fc M \to B$ and a continuous map $\ol f \fc B \to Y$ such that $f = \ol f\circ \pi$.
\end{defin}
\noindent Note that smooth involutive maps are involutive in this generalized sense.

Next, we need to define multiplicative ranks of algebras, see \cite{gromov2010singularities}.
\begin{defin}\label{def:ranks_algebras}
  Let $(A,*)$ be a finite-dimensional algebra. For $r\geq 1$ put
  $$A^{/r} = \bigcap\{I \in \cI(A)\,|\, \dim A/I < r\}\,.$$
  For $d \geq 1$, \tb{the $d$-rank of $A$}, denoted $\rk_dA$, is defined as the maximal $r$ for which $(A^{/r})^{*d} \neq 0$.
\end{defin}
\begin{exam}\label{ex:rank_at_least_1}
  Note that $A^{/1} = A$. If $A \neq 0$, this implies that $\rk_dA \geq 1$ for all $d$. Also note that $A^{/(\dim A+1)} = 0$, therefore $\rk_dA \leq \dim A$ for all $d$.
\end{exam}

\begin{thm}\label{thm:quantitative_nondisplaceable_fiber}Let $A$ be a finite-dimensional algebra, and let $(M,\omega)$ be a closed symplectic manifold equipped with an $A$-IVQM $\tau$. If $Y$ is a metric space of covering dimension $d$, then for any involutive map $f \fc M \to Y$ there is $y_0 \in Y$ such that
$$\dim A/\tau \big(M \setminus  f^{-1}(y_0)\big) \geq \rk_{d+1}A\,.$$
If furthermore $A\neq0$ and $\tau$ satisfies vanishing, then $f^{-1}(y_0)$ is non-displaceable.
\end{thm}

\noindent Theorem \ref{thm:quantitative_nondisplaceable_fiber} is proved in Section \ref{s:abstr_centerpt_thm_IVMs}. As a particular case of this result, we obtain the following generalization of the Non-displaceable fiber theorem \ref{thm:EP_nondisplaceable_fiber}.
\begin{coroll}[\tb{Quantitative non-displaceable fiber theorem}]\label{cor:quant_nondisplaceable_fiber_QH_IVQM}
  Let $(M,\omega)$ be a closed symplectic manifold, $Y$ a metric space of covering dimension $d$, and let $f \fc M \to Y$ be a continuous involutive map. If $\tau$ is the quantum cohomology IVQM on $M$, then there is $y_0 \in Y$ such that
  $$\dim_\Lambda QH^*(M) / \tau\big(M \setminus  f^{-1}(y_0)\big) \geq \rk_{d+1}QH^*(M)\,.$$
  In particular $f^{-1}(y_0)$ is non-displaceable. \qed
\end{coroll}

\subsection{Rigid fibers of involutive maps and symplectic rigidity}\label{ss:cent_fibers_invol_maps_sympl_rigidity}
Let us formulate the existence of rigid fibers in the context of involutive maps.
\begin{thm}\label{thm:symplectic_Tverberg}
Let $(M,\omega)$ be a closed symplectic manifold equipped with an $A$-IVQM $\tau$, where $A$ is an algebra. Let $I \in \cI(A)$ be a graded ideal such that $I^{*(d+1)} \neq 0$ for some $d \geq 1$. If $Y$ is a metric space of covering dimension $d$, then any continuous involutive map $f \fc M \to Y$ has a fiber which intersects each closet subset $Z\subset M$ such that $I\subset \tau(Z)$.
\end{thm}
\noindent This is proved in Section \ref{s:abstr_centerpt_thm_IVMs}, as a consequence of the corresponding topological result, Corollary \ref{cor:big_fiber_cont_maps_into_dim_d_metric}.

We will now formulate a new example of rigid symplectic intersections, whose proof, which appears in Section \ref{ss:pf_thm_torus_cross}, is based on a concrete example of Theorem \ref{thm:symplectic_Tverberg}. Consider the standard symplectic $6$-torus $\T^6$ with coordinates $p_i,q_i$, $i=1,2,3$ and symplectic form $\omega = dp\wedge dq$. For $a,b,c \in \T^2$ consider the following coisotropic subtori:
\begin{align*}
  T_1(a) &= \{(p,q) \in \T^6\,|\,(q_1,q_2) = a\}\,,\\
  T_2(b) &= \{(p,q) \in \T^6\,|\,(p_1,p_3) = b\}\,,\\
  T_3(c) &= \{(p,q) \in \T^6\,|\,(p_2,q_3) = c\}\,.
\end{align*}
Let
$$T(a,b,c) = T_1(a) \cup T_2(b) \cup T_3(c)\,.$$
In the next theorem an \tb{equator} is any smoothly embedded circle in $S^2$ dividing it into two disks of equal area.
\begin{thm}\label{thm:torus_cross}
  Let $B$ be a surface. Then any involutive map $\T^6\times S^2 \to B$ has a fiber which intersects every set of the form
  $$T(a,b,c) \times \text{equator}\,.$$
\end{thm}
\begin{rem}
  Let us comment on the sharpness of the various assumptions in this theorem:
  \begin{itemize}
    \item The dimension of $B$ cannot be increased. Indeed, consider the involutive map $f\fc\T^6 \times S^2 \to \T^3$, $(p,q;z) \mapsto (q_1,p_2,p_3)$. Let $t=(q_1',p_2',p_3') \in \T^3$. If $(a,b,c) \in \T^6$ is such that $a_1\neq q_1'$, $b_2\neq p_3'$, $c_1\neq p_2'$, then the fiber $f^{-1}(t)$ is disjoint from $T(a,b,c) \times S^2$.
    \item The involutivity assumption is essential. Consider the non-involutive projection $\pi\fc \T^6 \times S^2 \to S^2$, and let $w \in S^2$. Let $L\subset S^2$ be an equator such that $w \notin L$. Then any $T(a,b,c)\times L$ is disjoint from $\pi^{-1}(w) = \T^6\times\{w\}$.
    \item The union of just two coisotropic tori does not work: consider the involutive map $f\fc \T^6 \times S^2 \to \T^2$, $(p,q;z) \mapsto (q_1,p_3)$, and let $t = (q_1',p_3') \in \T^2$. For any $a,b\in\T^2$ such that $a_1\neq q_1'$ and $b_2\neq p_3'$ the fiber $f^{-1}(t)$ is disjoint from $(T_1(a) \cup T_2(b)) \times S^2$.
  \end{itemize}
\end{rem}

\subsection{SH-heavy sets}\label{ss:SH_heavy_sets}

In \cite{entov2009rigid} Entov--Polterovich defined a special class of compact subsets of a closed symplectic manifold, the so-called \emph{heavy sets} (see Section \ref{ss:heavy_implies_SH_heavy} for a reminder). They proved that a heavy subset is non-displaceable, however it remained unclear in general when two heavy sets should intersect. Here we address this problem to a degree.

Throughout this subsection, $\tau$ stands for the quantum cohomology IVQM on the appropriate symplectic manifold. First, let us introduce the suitable class of subsets.
\begin{defin}\label{def:SH_heavy}
  Let $(M,\omega)$ be a closed symplectic manifold. We call a compact set $K \subset M$ \tb{SH-heavy} if $\tau(K) \neq 0$.
\end{defin}
\begin{rem}
     A different hierarchy of rigid subsets of symplectic manifolds based on Varolg\"une\c s's relative symplectic cohomology was introduced in \cite{tonkonog2020super}, \cite{Borman2021quantum}. It would be interesting to explore its  relation to SH-heaviness.\footnote{\tb{Added in revision:} See \cite{maksunvarolgunes2023characterizationheaviness} for advances in this direction.} Here we will only point out the fact that for a compact set $K\subset M$ in a closed symplectic manifold $(M,\omega)$ we have $\tau(K) = QH^*(M)$ if and only if $K$ is SH-full in the terminology of \cite{tonkonog2020super}, that is $SH^*(K') = 0$ for each compact $K'\subset M$ disjoint from $K$. Indeed, if $K$ is SH-full and $U\supset K$ is open, then $SH^*(M\setminus U;\Lambda) = 0$, therefore $\tau(K) = \bigcap_{U\text{open}\supset K}\ker\big(SH^*(M;\Lambda) \to SH^*(M\setminus U;\Lambda)\big) = SH^*(M;\Lambda) = QH^*(M)$. Conversely, if $\tau(K) = QH^*(M)$ and $K'$ is compact and disjoint from $K$, then from the definition of $\tau$ it follows that $\ker \big(SH^*(M;\Lambda) \to SH^*(K';\Lambda)\big)= SH^*(M;\Lambda)$, or equivalently that the unit of $QH^*(M)$ is killed by the restriction $\res^M_{K'}$. Since restriction maps are unital (\cite{tonkonog2020super}), it follows that the unit of the algebra $SH^*(K';\Lambda)$ vanishes, therefore $SH^*(K';\Lambda) = 0$.
\end{rem}

\noindent The following is an immediate consequence of the vanishing property of $\tau$:
\begin{prop}\label{prop:SH_heavy_nondispl}
  SH-heavy sets are non-displaceable. \qed
\end{prop}

Next we formulate an algebraic criterion for nondisplaceability.
\begin{prop}\label{prop:alg_crit_SH_heavy_nondispl}
  Let $K,K'\subset M$ be compact sets in a closed symplectic manifold $(M,\omega)$. If $\tau(K)*\tau(K')\neq 0$, then $K,K'$ are SH-heavy and cannot be displaced from one another by a symplectic isotopy.
\end{prop}
\begin{proof}
  The assumption clearly implies that $\tau(K),\tau(K') \neq 0$, whence the first assertion. If $\phi \in \Symp_0(M,\omega)$ displaces $K'$ from $K$, then by the invariance and multiplicativity properties we have:
  $$\tau(K)*\tau(K') = \tau(K)*\tau(\phi(K')) \subset \tau(K \cap \phi(K')) = \tau(\varnothing) = 0\,,$$
  contradicting the assumption.
\end{proof}

The following theorem provides examples of SH-heavy subsets in standard symplectic tori. In its formulation we identify
$$QH^*(\T^{2n}) = H^*(\T^{2n};\Lambda) = H^*(\T^n;\Lambda)\otimes H^*(\T^n;\Lambda)$$
using the K\"unneth formula, while for a space $X$, $\mu_X$ denotes the cohomology IVM on $X$ (see Example \ref{ex:coh_IVM}).

\begin{thm}\label{thm:IVQM_annuli_in_tori}
  Let $M=\T^{2n} = \T^n(p)\times \T^n(q)$ be endowed with the symplectic form $\omega=dp\wedge dq$ and let $S\subset \T^n$ be a closed subset. Then
  \begin{multline*}
    \tau(S\times \T^n) = \mu_M(S\times \T^n) = \mu_{\T^n}(S)\otimes H^*(\T^n;\Lambda)\subset H^*(\T^n;\Lambda)\otimes H^*(\T^n;\Lambda) = H^*(M;\Lambda)\,.
  \end{multline*}
  In particular $S\times \T^n$ is SH-heavy if $S\neq \varnothing$. The same results hold if $S$ is open.
\end{thm}

\noindent This theorem is proved in Section \ref{ss:SH_special_domains} as a consequence of Theorem \ref{thm:domain_no_contr_orbits_bdry_SH_equal_H}, see Section \ref{ss:add_computations}, where we also present additional computations of symplectic cohomology.

We will now present a nontrivial instance of the use of Proposition \ref{prop:alg_crit_SH_heavy_nondispl}, based on Theorem \ref{thm:IVQM_annuli_in_tori}. For this we will need on the following K\"unneth-type lemma, proved in Section \ref{ss:pf_thm_torus_cross}, where $\tau$ stands for the quantum cohomology IVQM on $M,N,M\times N$:
\begin{lemma}\label{lem:IVQM_products}
	Let $M,N$ be closed symplectic manifolds and let $K\subset M$, $L\subset N$ be compact sets. If $N \setminus  L$ decomposes into a finite number of pairwise disjoint displaceable subsets, then $\tau(L) = SH^*(N;\Lambda)$, and moreover the K\"unneth isomorphism
	\[
		\psi\colon SH^*(M;\Lambda)\otimes SH^*(N;\Lambda) \to SH^*(M\times N;\Lambda)
	\]
	maps $\tau(K)\otimes \tau(L) = \tau(K) \otimes SH^*(N;\Lambda)$ into $\tau(K\times L)$.
\end{lemma}
\noindent We refer the reader to \cite{Varolgunes_2018_PhD} for the definition of the K\"unneth morphism for compact subsets of $M,N$. In general it is neither injective nor surjective, however when the sets in question are $M$ and $N$ themselves, it can be shown to be an isomorphism.

\begin{exam}\label{ex:transverse_Lag_tori_nondispl}
  Let $L = \{\pt\} \times \T^n(q),\, L'=\T^n(p)\times\{\pt\} \subset \T^{2n}$ be linear Lagrangian tori. Since $\ker\big(H^*(\T^n;\Lambda)\to H^*(\T^n\setminus\pt;\Lambda)\big)$ is spanned by the volume class, we have, according to Theorem \ref{thm:IVQM_annuli_in_tori}:
  $$\tau(L) =\Lambda\cdot [dp_1\wedge\dots\wedge dp_n]\otimes H^*(\T^n;\Lambda) = H^*(\T^{2n};\Lambda)\cdot\langle [dp_1\wedge\dots\wedge dp_n]\rangle\,,$$
  $$\text{and similarly}\quad\tau(L') = H^*(\T^{2n};\Lambda)\cdot\langle [dq_1\wedge\dots\wedge dq_n]\rangle\,.$$
  In particular we see that $\tau(L)*\tau(L')$ is spanned by $[\omega^n]$, therefore nonzero. By Proposition \ref{prop:alg_crit_SH_heavy_nondispl}, $L,L'$ cannot be displaced from one another by a symplectic isotopy. Of course, since the homological intersection number of $L,L'$ is nonzero, they in fact cannot be displaced from one another even by a smooth isotopy. To obtain a nontrivial example, we take the product with an equator $E \subset S^2$, namely let us identify $SH^*(\T^{2n};\Lambda)\otimes SH^*(S^2;\Lambda)$ with $SH^*(\T^{2n}\times S^2;\Lambda)$ by means of the K\"unneth isomorphism $\psi$ from Lemma \ref{lem:IVQM_products}, which then implies
  $$\tau(L\times E) \supset \tau(L)\otimes QH^*(S^2)\,,\qquad \tau(L'\times E) \supset \tau(L')\otimes QH^*(S^2)\,,$$
  whence
  \begin{multline*}
    \tau(L\times E)*\tau(L'\times E) \supset (\tau(L)\otimes QH^*(S^2))*(\tau(L')\otimes QH^*(S^2))\\=
      (\tau(L)*\tau(L'))\otimes QH^*(S^2) = H^*(\T^{2n};\Lambda)\cdot \langle [\omega^n]\rangle\otimes QH^*(S^2) \neq 0\,,
  \end{multline*}
  which by Proposition \ref{prop:alg_crit_SH_heavy_nondispl} implies that $L\times E$ and $L'\times E$ cannot be displaced from one another by a symplectic isotopy. That these subsets cannot be displaced from one another by a \emph{Hamiltonian} isotopy was proved in \cite{kawasaki2018function} by different techniques. Note that the intersection number argument no longer applies, and indeed $L\times E,L'\times E$ can be displaced from one another by a smooth isotopy. Note as well that $L\times E,L'\times E$ are heavy, but the technology of \cite{entov2009rigid} cannot guarantee a rigid intersection for them.
\end{exam}

It would be interesting to understand the connection between heavy and SH-heavy sets. In fact, we propose the following
\begin{conj}
  A compact subset of a closed symplectic manifold is heavy if and only if it is SH-heavy.\footnote{\tb{Added in revision:} The ``if'' direction was proved in full generality in \cite{maksunvarolgunes2023characterizationheaviness}. See also \cite{sun2021index_bdd} for an earlier partial result in the same direction.}
\end{conj}
\noindent Here we would like to prove that under certain assumptions, heavy sets are SH-heavy. First, recall that a symplectic manifold $(M,\omega)$ is \tb{symplectically aspherical} if $\omega$ and $c_1(M)$ vanish on $\pi_2(M)$. Next, a cooriented hypersurface $\Sigma\subset M$ is of \tb{contact type} if there exists a vector field $Y$ defined on a neighborhood of $\Sigma$ satisfying $\cL_Y\omega = \omega$, and such that along $\Sigma$, $Y$ points everywhere in the positive direction. Note that in this case $\alpha:=(\iota_Y\omega)|_\Sigma$ is a contact form on $\Sigma$. Lastly, $\Sigma$ is \tb{incompressible} if the map $\pi_1(\Sigma) \to \pi_1(M)$, induced by the inclusion, is injective.

If $\Sigma$ is an incompressible cooriented hypersurface of contact type in a symplectically aspherical manifold $(M,\omega)$ and all the contractible\footnote{Note that thanks to the incompressibility of $\Sigma$, contractibility may equivalently mean either in $\Sigma$ or in $M$.} Reeb orbits of $\alpha$ on $\Sigma$ are nondegenerate, then we can unambiguously assign a Conley--Zehnder index to each such orbit $\gamma$, as follows: $\gamma$ admits a contracting disk $u$ in $\Sigma$, and if $\xi=\ker\alpha$ is the contact structure on $\Sigma$, then $u^*\xi$ is trivializable and a trivialization allows us to assign an index to $\gamma$. Since $c_1(\xi)=c_1(M)|_\Sigma$ and $c_1(M)|_{\pi_2(M)}=0$ by the symplectic asphericity, the index is independent of the choice of the contracting disk.
\begin{defin}[See \cite{tonkonog2020super}]\label{def:idx_bdd}
  Let $(M,\omega)$ be symplectically aspherical and let $\Sigma\subset M$ be an incompressible cooriented hypersurface of contact type. We say that $\Sigma$ is \tb{index-bounded} if there exists a vector field $Y$ near $\Sigma$ as above, such that all the contractible Reeb orbits of $(\iota_Y\omega)|_\Sigma$ are nondegenerate and such that for each $k \in \Z$ the set of periods of the contractible Reeb orbits of Conley--Zehnder index $k$ is bounded.
\end{defin}

\begin{convention}For the rest of the paper, by a \tb{region} in a closed manifold we mean a compact codimension zero submanifold with (possibly empty) boundary.
\end{convention}

If $W \subset M$ is a region, we say that $W$ has \tb{contact-type boundary} or that it is a \tb{contact-type region} if $\partial W$ is of contact type relative to the outward coorientation. We then have the following result, where $\tau$ is the quantum cohomology IVQM on $M$:
\begin{thm}\label{thm:heavy_implies_SH_heavy}
Let $(M,\omega)$ be closed and symplectically aspherical. Let $K\subset M$ be a compact set such that there is a sequence $W_i$ of contact-type regions with incompressible index-bounded boundary such that $K\subset \Int W_i$ for each $i$ and such that $K = \bigcap_i W_i$. If $K$ is heavy, then for each $i$ we have
$$[\Vol] \in \ker \res^M_{\ol{W_i^c}}\,,$$
which implies that $[\Vol] \in \tau(K) = \bigcap_i\ker \res^M_{\ol{W_i^c}}$, and in particular that $K$ is SH-heavy.
\end{thm}
\noindent Here $[\Vol] \in H^{2n}(M;\Lambda)$ is the volume class. Theorem \ref{thm:heavy_implies_SH_heavy} is proved in Section \ref{ss:heavy_implies_SH_heavy}. We have the following immediate consequence.
\begin{coroll}
  If $(M,\omega)$ is closed and symplectically aspherical and $K$ is a heavy contact-type region with incompressible index-bounded boundary, then $K$ is SH-heavy.
\end{coroll}
\begin{proof}
  Let $\Sigma = \partial K$. Then there is a neighborhood of $\Sigma$ which is diffeomorphic to $\Sigma\times (1-\epsilon,1+\epsilon)_r$, such that the vector field $r\partial_r$ points outwards along $\Sigma = \Sigma\times\{0\}$, and such that $\cL_{r\partial_r}\omega = \omega$. In particular $W_i = K\cup(\Sigma\times [0,\epsilon/{2i}])$ are as in the theorem and the assertion follows.
\end{proof}

\subsection{Quantum IVQM versus cohomology IVM}\label{ss:add_computations}
In this section $(M,\omega)$ stands for a closed symplectically aspherical symplectic manifold, and $\tau$ is the quantum cohomology IVQM on $M$. Theorem \ref{thm:IVQM_annuli_in_tori} is a consequence of the following result, proved in Section \ref{ss:pf_SH_no_contr_orbits}:
\begin{thm}\label{thm:domain_no_contr_orbits_bdry_SH_equal_H}
  Let $K\subset M$ be a region with $\Sigma = \partial K$, such that $\Sigma \hookrightarrow M$ extends to a smooth embedding $(-\epsilon,\epsilon)\times \Sigma \hookrightarrow M$ such that no $\{\rho\}\times\Sigma$ carries closed characteristics which are contractible in $M$, for $|\rho|<\epsilon$. Then $SH^*(K;\Lambda) = H^*(K;\Lambda)$ and the restriction map $SH^*(M;\Lambda) \to SH^*(K;\Lambda)$ coincides with $H^*(M;\Lambda) \to H^*(K;\Lambda)$.
\end{thm}
For the remainder of this subsection, $\mu$ is the cohomology IVM on $M$, see Example \ref{ex:coh_IVM} above.
\begin{coroll}\label{cor:no_contr_chars_bdry_IVQM_equals_IVM}
  Under the assumptions of Theorem \ref{thm:domain_no_contr_orbits_bdry_SH_equal_H}, we have $\tau(K) = \mu(K) = \ker \big(H^*(M;\Lambda)\to H^*(M\setminus K;\Lambda)\big)$.
\end{coroll}
\begin{proof}
  Let us identify $(-\epsilon,\epsilon)\times\Sigma$ with its image under the embedding appearing in the formulation of the theorem. The sets $Q_\delta:=M\setminus(K\cup([0,\delta)\times\Sigma))$ for $\delta \in (0,\frac{\epsilon}2)$ are cofinal in the collection of compact sets disjoint from $K$, whence
  $$\tau(K) = \bigcap_{\delta\in(0,\frac{\epsilon}{2})} \ker\big(QH^*(M)\to SH^*(Q_\delta;\Lambda)\big)\,.$$
  Since each $Q_\delta$ also satisfies the assumptions of the theorem, it follows that
  $$\ker\big(QH^*(M)\to SH^*(Q_\delta;\Lambda)\big) = \ker\big(H^*(M;\Lambda)\to H^*(Q_\delta;\Lambda)\big)\,,$$
  which equals $\ker\big(H^*(M;\Lambda)\to H^*(M\setminus K;\Lambda)\big)$ since each $Q_\delta$ is a deformation retract of $M\setminus K$. Thus
  $$\tau(K) = \ker\big(H^*(M;\Lambda)\to H^*(M\setminus K;\Lambda)\big)\,,$$
  which also equals $\mu(K)$ by the same arguments.
\end{proof}

We also include results regarding regions with contact-type incompressible index-bounded boundary. Given a contact-type region $K$, we let $\wh K$ be the completion of $K$, obtained by attaching to it the positive end of the symplectization of $\partial K$. In Section \ref{sss:SH_Liouv_domains} we will define $SH_{\cl}^*(\wh K;\Lambda)$, the classical symplectic cohomology of $\wh K$, as in \cite{viterbo1999functorsI}. In Section \ref{ss:pf_thm_SH_Liouv_domains} we will prove the following result:
\begin{thm}\label{thm:Liouville_dom_idx_bdd_SH} Let $K\subset M$ be a contact-type region with incompressible index-bounded boundary. Then:
  \begin{enumerate}
    \item $SH^*(K;\Lambda)$ is canonically isomorphic to $SH_{\cl}^*(\wh K;\Lambda)$;
    \item $\ker\big(H^*(M;\Lambda)\to H^*(K;\Lambda)\big) \subset \ker \big(\res^M_K \fc SH^*(M;\Lambda) \to SH^*(K;\Lambda)\big)$.
  \end{enumerate}
\end{thm}

\begin{rem}
  This result can be interpreted as a kind of excision property for relative symplectic cohomology under these assumptions. Another result in this direction appears in \cite{groman2023floertheory}. It would be interesting to understand the relation between the two.
\end{rem}

Based on this result, and using the arguments of Corollary \ref{cor:no_contr_chars_bdry_IVQM_equals_IVM}, we can prove:
\begin{coroll}
  Under the assumptions of Theorem \ref{thm:Liouville_dom_idx_bdd_SH}, we have
  $$\tau(M \setminus K)   \supset \mu(M \setminus K)\,. \qquad \qed$$
\end{coroll}

\begin{rem}
  Note that this corollary, as well as Corollary \ref{cor:no_contr_chars_bdry_IVQM_equals_IVM}, fail without the assumptions. Indeed, in the sphere $S^2$ let $K$ be a closed disk with smooth boundary and area $\geq\frac 12$. Then $S^2\setminus K$ is an open disk with smooth boundary of area $\leq \frac12$, and thus $\tau(S^2\setminus K) = 0$ (see Example \ref{ex:QH_IVQM_sphere}), whereas $\mu(S^2\setminus K)$ is spanned by the area class, and in particular is nonzero.
\end{rem}

For additional computations of relative symplectic cohomology under related assumptions see \cite{sun2021index_bdd}.

\subsection{Discussion}

\subsubsection{Persistence modules and symplectic cohomology}

Varolg\"une\c s's symplectic cohomology is constructed as a module over the Novikov ring $\Lambda_{\geq0}$. Here we only use the invariant obtained by tensoring with the Novikov field $\Lambda$, which forgets torsion. This operation is unnecessary for many constructions in this paper. It would be interesting to understand and apply the additional information carried by the torsion part.

Novikov modules, that is modules over $\Lambda_{\geq0}$, can carry an additional structure, namely a \emph{valuation}, which in turn gives rise to \emph{persistence modules} (see for instance \cite{polterovich2020topological} for preliminaries on persistence modules in the symplectic context). All the symplectic cohomology modules come with natural valuations. Some elementary examples show that the cokernel of the restriction map $SH^*(M) \to SH^*(M \setminus U)$ is a Novikov module which can have non-trivial torsion. The latter can be interpreted as a persistence module whose structure is encoded by combinatorial invariants, the \emph{barcodes} (see \emph{ibid.}). Thus, to every subset of $M$ there correspond a persistence module and its barcodes. Understanding the algebraic structure behind this correspondence should undoubtedly lead to new applications of symplectic cohomology.

\subsubsection{Rigidity of open covers and multi-intersections}\label{subsubsec-multi}

In \cite{entov2006quasi}, Entov and Polterovich proved the following theorem.
\begin{thm}[Entov-Polterovich]\label{thm: EP no commute par. of 1 of displ cover}
	Let $(M,\omega)$ be a closed symplectic manifold, let $\cU=\{U_1,\ldots,U_N\}$ be a finite open cover of $M$ and let $\{\varphi_1,\ldots,\varphi_N\}$ be a partition of unity subordinated to $\cU$. If for every  $1\leq i,j\leq N$ we have $\{\varphi_i,\varphi_j\}=0$ then there exists an element of $\cU$ which is non-displaceable.
\end{thm}

Using symplectic IVQMs, we can generalize this result. To this end, we need the following generalization of Poisson commutation to several sets.
\begin{defin}
Let $(M,\omega)$ be a symplectic manifold. Let $\{A_1,\dots,A_N\}$ be a collection of open (respectively, closed) subsets of $M$. We say that $\{A_1,\dots,A_N\}$ \textbf{Poisson commutes} if there is a smooth involutive map $\pi\fc M\to B$ and open (respectively, closed) subsets $B_1,\ldots,B_N\subset B$ such that $A_i=\pi^{-1}(B_i)$ for every $1\leq i\leq N$.
\end{defin}
\begin{rem}\label{rem:commutation_several_sets}
  \begin{itemize}
    \item Note that for two subsets, this definition is equivalent to Definition \ref{def:commuting_sets}.
    \item If a collection of sets $A_1,\dots,A_N$ Poisson commutes, where all the sets are simultaneously open or closed, then so does the closure of $\{A_i\}_i$ under unions and intersections.
  \end{itemize}
\end{rem}

Now we can formulate our generalization.

\begin{prop} \label{thm-ep-revisited}
	Let $(M,\omega)$ be a closed symplectic manifold, let $A$ be an algebra, and fix an $A$-IVQM $\sigma$ on $M$. If $K_1,\ldots,K_N \subset M$ is a Poisson commuting collection of compacts with $\bigcap_iK_i=\varnothing$, then
$$\prod_i \sigma(K_i) = 0\,.$$
\end{prop}
\begin{proof}Thanks to Remark \ref{rem:commutation_several_sets}, each $K_i$ commutes with $K_1\cap\dots \cap K_{i-1}$, therefore
$$\prod_i\sigma(K_i) \subset \sigma\big(\textstyle\bigcap_i K_i\big) = \sigma(\varnothing)=0$$
by induction and the quasi-multiplicativity property of $\sigma$.
\end{proof}

\begin{proof}[Proof of Theorem \ref{thm: EP no commute par. of 1 of displ cover}]The map $\Phi = (\varphi_1,\dots,\varphi_N) \fc M \to \R^N$
is involutive, since the $\varphi_i$ pairwise Poisson commute. The sets $K_i:=\varphi_i^{-1}(0)$ therefore form a commuting collection of compacts. Moreover $\bigcap_iK_i = \varnothing$ since $\{\varphi_1,\dots,\varphi_N\}$ is a partition of unity.

Let $\tau$ be the quantum cohomology IVQM on $M$. Thanks to Theorem \ref{thm-ep-revisited}, we have
$$\prod_i\tau(K_i) = 0\,,$$
which implies that there exists $i=1,\dots,N$ such that $\tau(K_i) \neq QH^*(M)$. Since the collection $(\{\varphi_i<\alpha\})_{\alpha>0}$ is cofinal in the family of all open sets containing $K_i$, we obtain
$$\tau(K_i) = \bigcap_{\alpha>0}\tau\big(\{\varphi_i<\alpha\}\big)\,,$$
which implies the existence of $\alpha > 0$ with $\tau\big(\{\varphi_i<\alpha\}\big) \neq QH^*(M)$. Thanks to the vanishing property of $\tau$, it follows that $\{\varphi_i\geq\alpha\}=M\setminus\{\varphi_i<\alpha\}$ is non-displaceable, and therefore so is the larger set $U_i$.
\end{proof}

\medskip\noindent

Note that Theorem \ref{thm-ep-revisited} admits the following equivalent formulation: \emph{Let $U_1,\dots,U_N$ be a Poisson commuting open cover of $M$. Then
	\begin{equation}
		\label{eq-prod-vsp}
		\prod_i \sigma(M\setminus U_i) = 0\,.
	\end{equation}}

\begin{exam}\label{exam-threesets} Let $M$ be the two-dimensional torus $\mathbb{R}^2/\mathbb{Z}^2$ equipped with the form $dp \wedge dq$ and let $\tau$ be the quantum cohomology IVQM on $M$. Consider annuli $Q = \{0 < q <a\}$ and $P= \{0 < p <b\}$ with $a, b \in (0,1)$. Consider the rectangle $M \setminus (P \cup Q)$, and denote by $R$ a slightly bigger open rectangle. By Theorem \ref{thm:IVQM_annuli_in_tori} we have $\tau(P^c) = A\cdot [dp]$ and $\tau(Q^c) = A\cdot [dq]$, where $A = QH^*(M)$. Additionally, \cite[Corollary 1.15]{tonkonog2020super} implies that $\tau(R^c) = A$. Indeed, \cite[Corollary 1.15]{tonkonog2020super} implies that for every closed disk $D$ in $M$ we have $SH(D;\Lambda)=0$, in particular $\ker \res^M_D=SH(M;\Lambda)=A$. Thus, since $R$ is an open disk in $M$, we see that
$$\tau(R^c)=\bigcap_{U\,\text{open} \supset R^c } \ker \res^M_{M\setminus U} = \bigcap_{D\,\text{closed disk} \subset R }\ker\res^M_D = A.$$
It follows that
$$\tau(P^c)\cdot\tau(Q^c)\cdot\tau(R^c) = A\cdot[dp\wedge dq] \neq 0\,,$$
that is \eqref{eq-prod-vsp} does not hold for the open cover $P,Q,R$, which means that it does not Poisson commute. Furthermore, by the invariance property of $\tau$ and Theorem \ref{thm-ep-revisited}, for every triple of symplectomorphisms $f,g,h \in \Symp_0(M)$ the images $f(P), g(Q), h(R)$ cannot form a Poisson commuting cover of $M$.
\end{exam}

\medskip
Given a closed symplectic manifold $M$, Proposition \ref{thm-ep-revisited} implies that if $K_1,\ldots,K_N$ are closed subsets with $\prod_i \tau(K_i)\neq0$, where $\tau$ is the quantum cohomology IVQM of the manifold, then there are no $\varphi_1,\ldots,\varphi_N \in \Symp_0(M,\omega)$ such that $\varphi(K_1),\ldots,\varphi(K_N)$ Poisson commute and $\bigcap_i \varphi_i(K_i)=\varnothing$.

Note as well that in \cite{varolgunes2018mayer} Varolg\"une\c s defined an invariant $SH^*(K_1,\dots,K_r)$ for arbitrary compact $K_1,\dots,K_r\subset M$, which vanishes if the $K_i$ commute. This invariant then measures, in a sense, the non-commutativity of the given sets. It is interesting to understand how this invariant is related to discussion in this section.

\subsubsection{Lagrangian IVQMs}\label{subsubsec-LagIVQM} Fix a closed Lagrangian submanifold $L \subset M$.
In \cite{tonkonog2020super} Tonkonog and Varolg\"une\c s presented an extension of Varolg\"une\c s's relative symplectic cohomology to the Lagrangian setting, where the Hamiltonian Floer cohomology $HF^*(H)$ of a Hamiltonian $H$ is replaced by the Lagrangian Floer cohomology $HF^*(L,H)$.
It is natural to expect that, arguing along the lines of the present paper, one can arrive at an IVQM on $M$ based on the Lagrangian quantum ring $QH^*(L)$ (see the paper \cite{biran2007quantum} by Biran and Cornea for a detailed introduction to this ring for
monotone Lagrangian submanifolds). Note that this ring is graded, but in general not skew-commutative, so one has to work out how to extend the notion of an IVQM to non-commutative algebras.

``Big fiber" theorems for such a $QH^*(L)$-IVQM should yield further symplectic intersection results. For instance, an analogue of Theorem \ref{thm:quantitative_nondisplaceable_fiber} above should yield the following result due to Varolg\"une\c s: \emph{Every involutive map $M\to B$ admits a fiber which cannot be displaced from $L$ by a Hamiltonian isotopy} (see \cite[Theorem 1.4.1]{varolgunes2018mayer} and \cite[Corollary 1.11]{kawasaki2018function}).

\medskip
\noindent
{\bf Organization of the paper:} In Section \ref{ss:quantum_coh_IVQM} we briefly recall the definition of Varolg\"une\c s's symplectic cohomology for compact subsets of symplectic manifolds. In Section \ref{ss:axioms_rel_SH_pairs} we formulate a list of axioms for a relative symplectic cohomology for pairs of compact subsets. We then state Theorem \ref{thm:existence_relative_SH_pairs}, which is an existence result for such an object. We then use it to prove our Main Theorem \ref{thm:QH_IVQM_is_IVQM} in Section \ref{s:proof_main_thm}, thus establishing the existence of the quantum cohomology IVQM. Section \ref{s:proofs_properties_SH_pairs}, which is the technical heart of the paper, is dedicated to the proof of the existence of a relative symplectic cohomology of pairs, Theorem \ref{thm:existence_relative_SH_pairs}, using a combination of Floer theory with tools of homotopical algebra. Section \ref{s:examples_computations} contains the proofs of the results formulated in Section \ref{ss:quantit_nd_fiber_thm} (quantitative non-displaceable fiber theorem), Section \ref{ss:cent_fibers_invol_maps_sympl_rigidity} (a new example of symplectic rigidity), Section \ref{ss:SH_heavy_sets} (SH-heavy sets), and in Section \ref{ss:add_computations} (additional computations of SH). Last, but not least, Section \ref{s:abstr_centerpt_thm_IVMs} is dedicated to a streamlined exposition of the use of IVMs in the proofs of big fiber Theorems \ref{thm:Gromov_torus}, \ref{thm:top_Tverberg}, as well as the relation between IVMs, IVQMs, and involutive maps, and how these imply the quantitative non-displaceable fiber Theorem \ref{thm:quantitative_nondisplaceable_fiber}, the existence of rigid fibers of involutive maps, Theorem \ref{thm:symplectic_Tverberg}, and our new symplectic intersection result, Theorem \ref{thm:torus_cross}.

\section{Relative symplectic cohomology of pairs}\label{s:relative_SH_pairs}

In this section we briefly review a definition of Varolg\"une\c s's relative symplectic cohomology \cite{Varolgunes_2018_PhD}, \cite{varolgunes2018mayer}. The proof of our Main Theorem \ref{thm:QH_IVQM_is_IVQM}, which appears in Section \ref{s:proof_main_thm}, is based on a new invariant---a relative symplectic cohomology of pairs---which we introduce in Section \ref{ss:axioms_rel_SH_pairs} using an axiomatic approach. Throughout this section $(M,\omega)$ is a fixed closed symplectic manifold.

Let us begin by detailing our \tb{sign conventions} regarding Floer theory.
\begin{itemize}
                                \item An almost complex structure $J$ on $M$ is compatible with $\omega$ if $\omega(\cdot,J\cdot)$ is a Riemannian metric.
                                \item The symplectic action of a loop $x$ capped by a disk $u$ relative to a Hamiltonian $H$ is $\cA_H(x,u) = \int_0^1H_t(x(t))\,dt - \int u^*\omega$.
                                \item The Floer equation we use corresponds to the positive gradient flow of the action functional: $\partial_su-J(u)(\partial_tu-X_H(u))=0$. Similarly for continuations maps, homotopies, and so on.
                              \end{itemize}
\begin{rem}
Our sign conventions differ from those of \cite{varolgunes2018mayer}. \emph{The two are related by time reversal.} In detail, if $H$ is a time-dependent Hamiltonian, $x$ is a loop in $M$, $u \fc \R\times S^1 \to M$ is a smooth map, and $(J_t)_t$ is a loop of almost complex structures on $M$, we define the corresponding time-reversed objects $\ol H$, $\ol x$, $\ol u$, and $\ol J$ by $\ol H_t = H_{-t}$, $\ol x(t) = x(-t)$, $\ol u(s,t) = u(s,-t)$, and $\ol J_t= J_{-t}$. Then $x\mapsto \ol x$ is a bijection between the $1$-periodic orbits of $H$ according to the conventions of $\cite{varolgunes2018mayer}$ and those of $\ol H$ according to our conventions. The induced map on Floer complexes is an isomorphism, provided we use $J$ in Varolg\"une\c s's Floer equation and $\ol J$ in ours, since $u$ solves the former if and only if $\ol u$ solves the latter. Note in particular that this means that appropriately defined (cohomological) spectral invariants for autonomous Hamiltonians are identical in the two conventions, since time reversal has no effect.
\end{rem}

\subsection{Relative symplectic cohomology}\label{ss:quantum_coh_IVQM}

We will now briefly review the definition of the relative symplectic cohomology. The \tb{Novikov ring} is the subring
$$\Lambda_{\geq0} = \left\{\sum_{i=0}^{\infty}c_iT^{\lambda_i} \in \Lambda\,\Big|\, \forall i\,:\lambda_i\geq 0\right\}$$
of the Novikov field $\Lambda$. Given a compact $K \subset M$ and a commutative $\Lambda_{\geq0}$-algebra $R$, the \tb{relative symplectic cohomology of $K$ in $M$ with coefficients in $R$} is
\begin{equation}\label{eqn:def_SH_Lambda_coeffs}
  SH^*(K;R) = H^*\Big(\textstyle\wh{\varinjlim_{i}}\, CF^*(H_i) \otimes_{\Lambda_{\geq0}}R\Big)\,,
\end{equation}
where $H_i$ is a pointwise increasing sequence of non-degenerate time-periodic Hamiltonians on $M$ such that $H_i|_K<0$, and such that
$$\lim_{i\to\infty}H_i(x) = \left\{\begin{array}{ll}0\,, & x \in K \\ \infty\,, & x \notin K\end{array}\right.\,;$$
$CF^*(H_i)$ stands for the Floer complex
\begin{equation}\label{eqn:defin_Floer_cx}
  CF^*(H_i) = \bigoplus_{x \in \cP^\circ(H_i)}\Lambda_{\geq 0}\cdot x\,,
\end{equation}
where $\cP^\circ(H_i)$ is the set of $1$-periodic orbits of $H_i$, which are contractible in $M$. The complexes $CF^*(H_i)$ are connected by Floer continuation maps. Finally, the hat stands for the completion of $\Lambda_{\geq 0}$-modules (see Section \ref{ss:completion}).

\begin{rem}
  \begin{itemize}
    \item We will abbreviate $SH^*(K;\Lambda_{\geq0}$) to $SH^*(K)$ throughout.
    \item Varolg\"une\c s uses the notation $SH_M^*$ to explicitly point out the symplectic manifold. We drop the subscript $_M$, trusting that the context will resolve the ambiguity.
    \item The cohomology $SH^*(K;R)$ is independent of the specific choice of Hamiltonians $H_i$ and other auxiliary data. To prove this, and in general to be able to work with $SH^*$, it is advantageous to use an alternative definition based on homotopy theory, see Section \ref{s:proofs_properties_SH_pairs} and \cite{varolgunes2018mayer} for more details.
    \item \emph{Throughout the paper we use the properties of $SH^*$ listed in Section \ref{ss:sympl_IVQMs}}.
  \end{itemize}
\end{rem}

\begin{rem}\label{rem:only_contractible_orbits_SH}
  Unlike the original definition by Varolg\"une\c s, we only consider contractible periodic orbits. The reason for this is that we are only interested in kernels of restriction maps defined on $SH^*(M;\Lambda)$, which on the homology level is generated by contractible orbits, and chain level restriction maps  preserve the free homotopy class of periodic orbits and thus preserve the contractible component. Also, since we only consider contractible orbits, the grading can be taken mod $2N_M$ rather than merely mod $2$.
\end{rem}

\begin{exam}\label{ex:QH_IVQM_sphere}
  Let $M = S^2$ with an area form $\omega$ which is normalized to have area $1$. The quantum cohomology algebra of $M$ is $QH^*(M) = \Lambda\langle 1, h \rangle$, where $1$ is the unit class while $h = [\omega]$ is the area class. The grading is modulo $2N_{S^2} = 4$, $|1| = 0$, $|h|=2$. The multiplication is completely determined by $h^2 = T\cdot 1$.

  If $D\subset M$ is a smooth closed disk of area $<\frac{1}{2}$, then $D$ is displaceable, and thus $SH^*(D;\Lambda) = 0$. If $D$ has area $>\frac{1}{2}$, then $SH^*(D;\Lambda) = SH^*(M;\Lambda) = QH^*(S^2)$, and moreover $\res^M_D$ is the identity, which can be inferred from the Mayer--Vietoris property of $SH^*$.\footnote{The case $\area(D)=\frac 1 2$ requires an additional calculation. See for instance \cite{Varolgunes_2018_PhD}.} It follows that the quantum cohomology IVQM $\tau$ on $S^2$ satisfies
  $$\tau(D) = \left\{\begin{array}{ll}0\,,&\text{if }\area(D)<\frac 1 2\,,\\ QH^*(S^2)\,,&\text{if }\area(D) \geq \frac 1 2\,,\end{array}\right.$$
  for a closed disk $D\subset S^2$. It follows that $\tau$ is precisely the unique $QH^*(S^2)$-IVQM described in Proposition \ref{prop:unique_A_IVQM_sphere}. From the definition of $\tau$ it also easily follows that if $U\subset S^2$ is an open disk, then $\tau(U)=0$ if $\area(U)\leq\frac12$ and $\tau(U)=QH^*(S^2)$ otherwise.
\end{exam}

\subsection{Axioms for relative symplectic cohomology of pairs}\label{ss:axioms_rel_SH_pairs}

Here we formulate the axioms for a relative symplectic cohomology of pairs of compact subsets of $M$. It is an extension of Varolg\"une\c s's relative symplectic cohomology. The axioms are reminiscent of those of Eilenberg--Steenrod for cohomology, see the paper \cite{cieliebak2018symplectic} by Cieliebak and Oancea. This is not the most extensive list, but it is reasonably complete, and it contains all the properties we need to prove our results. Section \ref{s:proof_main_thm} contains the proof of our Main Theorem \ref{thm:QH_IVQM_is_IVQM} based on the axioms.

As we will see in Section \ref{s:proof_main_thm}, in order to prove our main theorem, we only use these axioms, regardless of the details of a particular choice of a relative symplectic cohomology of pairs. However, since we do perform a concrete construction in order to prove the existence of such an object, and to somewhat demystify the axioms, we will comment here on our construction. It is instructive to recall that the singular cohomology $H^*(X,A)$ of a pair $(X,A)$ is the cohomology of the complex $C^*(X,A)$, which by definition is the kernel of the restriction map $C^*(X)\to C^*(A)$, where $C^*(\cdot)$ is the singular cochain functor. In our story for each compact $K\subset M$ and a $\Lambda_{\geq0}$-algebra $R$ there is a complex $SC^*(K;R)$ computing $SH^*(K;R)$. The difference between this and the singular complex is that $SC^*(K;R)$ depends on various choices, such as Hamiltonians, almost complex structures, and so on, and is really well-defined only up to homotopy equivalence. This forces us to use the \emph{homotopy kernel of the chain level restriction map} $\res^K_{K'} \fc SC^*(K;R) \to SC^*(K';R)$ for $K'\subset K$, as a complex underlying the symplectic cohomology $SH^*(K,K';R)$ of the pair $(K,K')$. We use a very convenient model for homotopy kernels, given by the \emph{cocone construction}. See Section \ref{s:proofs_properties_SH_pairs}, in particular Definition \ref{def:rel_complex_pair}, for details. Also note that the cocone construction is standard when defining relative cohomology invariants on the chain level, see for instance \cite[p.78]{bott1982differential}.

Let $\cC$ denote the category of compact subsets of $M$, where the morphisms are inclusions, and let $\cC\cP$ denote the category of compact pairs of subsets of $M$, that is the objects are pairs $(K,K')$ of compact subsets of $M$ such that $K' \subset K$, and there is exactly one morphism $(K,K') \to (L,L')$ if $K\subset L$ and $K' \subset L'$.

Recall that if $P,Q$ are graded modules with graded components $P^i,Q^i$, respectively, then a module map $f \fc P \to Q$ is \tb{graded of degree $d$} if $f(P^i)\subset Q^{i+d}$ for all $i$.

Fix a commutative $\Lambda_{\geq0}$-algebra $R$; in applications $R$ is either $\Lambda_{\geq0}$ or $\Lambda$. Let $\cM$ be the category of $\Z_{2N_M}$-graded $R$-modules and degree zero module maps, and let $\cA_{\text{nu}}\subset \cM$ be the subcategory consisting of associative skew-commutative non-unital $R$-algebras and degree zero algebra morphisms. We let $[1]$ be the shift functor on $\cM$. Varolg\"une\c s's symplectic cohomology is a contravariant functor $SH^*(\cdot;R)\fc \cC \to \cA_{\text{nu}}$. We refer the reader to the discussion in Section \ref{ss:rel_SH_discussion} regarding the product structures on $SH^*(\cdot;R)$.

For the definition we need the so-called \emph{descent} property for pairs of compact sets. It is defined in \cite{varolgunes2018mayer}, and we recall the definition in Section \ref{sss:accel_data_descent}. As Varolg\"une\c s proves in \cite{varolgunes2018mayer}, if two compact sets commute, they are in particular in descent; therefore descent can be interpreted as a kind of ``algebraic commutation'' condition.

We define a \tb{relative symplectic cohomology of pairs with coefficients in $R$} as any contravariant functor $SH^*(\cdot,\cdot;R) \fc \cC\cP \to \cA_{\text{nu}}$ satisfying the properties appearing below. Note that the functor property means that for each compact pair $(K,K')$ we have the corresponding graded algebra $SH^*(K,K';R)$, and that for each inclusion of pairs $(K,K')\subset (L,L')$ there is a restriction map
$$\res^{(L,L')}_{(K,K')} \fc SH^*(L,L';R) \to SH^*(K,K';R)\,,$$
which is a degree zero algebra morphism, such that if $(P,P')$ is another pair containing $(L,L')$, then
$$\res^{(L,L')}_{(K,K')}\circ \res^{(P,P')}_{(L,L')} = \res^{(P,P')}_{(K,K')}\,,$$
and such that $\res^{(K,K')}_{(K,K')} = \id_{SH^*(K,K';R)}$. The properties are as follows:
\begin{itemize}
  \item \tb{(Normalization):} Let $\iota\fc\cC \to \cC\cP$ be the inclusion functor $\iota \fc K \mapsto (K,\varnothing)$. There is a natural isomorphism between the composition $SH^*(\cdot,\cdot;R)\circ\iota$ and $SH^*(\cdot;R)$, meaning that for each $K \in \cC$ there is an isomorphism $SH^*(K,\varnothing;R) = SH^*(K;R)$ in $\cA_{\text{nu}}$ commuting with restrictions.

  \begin{notation}
    For $K',K''\subset K$ we denote by $\res^{(K,K')}_{K''}$ the composition of $\res^{(K,K')}_{(K'',\varnothing)}$ and the isomorphism $SH^*(K'',\varnothing;R)=SH^*(K'';R)$.
  \end{notation}

  \begin{notation}
    We let $j \fc \cA_{\text{nu}}\hookrightarrow\cM$ denote the inclusion functor.
  \end{notation}

  \item \tb{(Triangle):} Let $\Pi' \fc \cC\cP \to \cC$ be the functor projecting onto the second factor, that is $(K,K')\mapsto K'$. Then there is a natural transformation
      $$\delta^* \fc j\circ SH^*(\cdot;R)\circ \Pi'\to j \circ SH^*(\cdot,\cdot;R)[1]\quad\text{as functors }\cC\cP\to\cM\,,$$
      that is for each $(K,K') \in \cC\cP$ we have a module map $\delta_{(K,K')}^* \fc SH^*(K';R) \to SH^*(K,K';R)$ of degree $1$, compatible with restrictions. Moreover, it fits into an exact triangle
      $$\xymatrix{SH^*(K,K';R) \ar[rr]^{\res^{(K,K')}_K} & & SH^*(K;R)\ar[ld]^{\res^K_{K'}} \\ & SH^*(K';R)\ar[lu]_(0.3){\delta_{(K,K')}^*}}$$

  \item \tb{(Mayer--Vietoris):} Let us define the following subcategory of $\cC^3$:\footnote{`CTD' stands for ``compact triples in descent.''}
      $$\mathcal{CTD} = \{(K,K',K'') \in \cC^3\,|\,K',K''\subset K\text{ and }K',K''\text{ are in descent}\}\,.$$
      Let us define the intersection and union functors $I, U \fc \mathcal{CTD} \to \cC\cP$ by
      $$I(K,K',K'') = (K,K'\cap K'') \,, \qquad U(K,K',K'') = (K,K'\cup K'')\,.$$
      There is then a natural transformation
      $$\Delta^* \fc j\circ SH(\cdot,\cdot;R) \circ I \to j\circ SH(\cdot,\cdot;R)[1] \circ U\quad \text{as functors }\mathcal{CTD}\to\cM\,,$$
      meaning whenever $(K,K',K'') \in \mathcal{CTD}$, then we have a degree $1$ linear map
      $$\Delta^*_{(K,K',K'')} \fc SH^*(K,K'\cap K'';R) \to SH^*(K,K'\cup K'';R)$$
      compatible with restrictions. It moreover fits into the exact \tb{Mayer--Vietoris triangle}
      $$\small\xymatrix{SH^*(K,K'\cup K'';R) \ar[rr]^-{\big(\res^{(K,K'\cup K'')}_{(K,K')},\,\res^{(K,K'\cup K'')}_{(K,K'')}\big)} & & SH^*(K,K';R)\oplus SH^*(K,K'';R)\ar[ld]|{\res^{(K,K')}_{(K,K'\cap K'')}-\res^{(K,K'')}_{(K,K'\cap K'')}} \\ & SH^*(K,K'\cap K'';R)\ar[lu]|{\Delta^*_{(K,K',K'')}}}$$

  \item \tb{(Product):}
      Given $(K,K',K'')\in\mathcal{CTD}$, there exists a map $\wt{*}$ fitting in the following commutative diagram:
      $$\xymatrix{SH^*(K,K';R)\otimes SH^*(K,K'';R) \ar[r]^-{\wt*} \ar[d]_{\res^{(K,K')}_K \otimes \res^{(K,K'')}_K} & SH^*(K,K'\cup K'';R) \ar[d]^{\res^{(K,K'\cup K'')}_K} \\ SH^*(K;R) \otimes SH^*(K;R)\ar[r]^-{*} & SH^*(K;R)}$$
      where $*$ stands for the Tonkonog--Varolg\"une\c s product, see Section \ref{ss:rel_SH_discussion}. Specializing to the case $K'=K''$, we see that this defines a product on $SH^*(K,K';R)$. \emph{We require that it coincide with the given product on $SH^*(K,K';R)$ as an algebra.}
\end{itemize}

\noindent In Section 4, we prove
\begin{thm}\label{thm:existence_relative_SH_pairs}There exists a relative symplectic cohomology of pairs with coefficients in $\Lambda_{\geq0}$.
\end{thm}
\noindent The definition of a relative symplectic cohomology of pairs is given in Section \ref{sss:rel_SH_pairs}, see Definition \ref{def:rel_complex_pair}. In Section \ref{ss:constructing_product} we construct the product and prove that it indeed satisfies the Product axiom. We prove the Normalization and the Triangle property in full detail in Section \ref{ss:SH_pair_well_defined}. There also the Mayer--Vietoris property is elaborated upon, with the exception of the compatibility of the natural transformation with restrictions, but this compatibility can be proved using techniques appearing in Section \ref{ss:SH_pair_well_defined}.

\begin{rem}\label{rem:restricted_products}
It is very likely that, using the techniques of Section \ref{s:proofs_properties_SH_pairs}, it is possible to prove that the product we construct satisfies the following stronger version of the Product axiom: Let $\Pi',\Pi'' \fc \mathcal{CTD} \to \cC\cP$ be the projection functors $\Pi'(K,K',K'') = (K,K')$, $\Pi''(K,K',K'') = (K,K'')$; then there is a natural transformation
$$\wt * \fc SH(\cdot,\cdot;R)\circ \Pi' \otimes SH(\cdot,\cdot;R)\circ \Pi'' \to SH(\cdot,\cdot;R) \circ U\quad\text{as functors } \mathcal{CTD} \to \cA_{\text{nu}}\,,$$
that is for $(K,K',K'') \in \mathcal{CTD}$ there is a degree zero algebra morphism
$$\wt * \fc SH(K,K';R) \otimes SH(K,K'';R) \to SH(K,K'\cup K'';R)\,,$$
compatible with restriction morphisms. The weaker version as stated above does not contain the requirement that $\wt *$ be an algebra morphism or that it be compatible with restrictions. Since the weaker version suffices in order to construct an IVQM as stated in our Main Theorem \ref{thm:QH_IVQM_is_IVQM}, we opted for only proving the weaker version.
\end{rem}

It is plausible that the triangle property can be generalized to a long exact sequence of a general triple, just like in ordinary cohomology. It is also conceivable that other axioms can be formulated and proved for the relative symplectic cohomology of a pair.

If $R$ is any flat $\Lambda_{\geq 0}$-algebra, for instance $R = \Lambda$, then
we have the following result.
\begin{coroll}\label{cor:SH_pairs_change_coeffs}
  Let $SH^*(\cdot,\cdot)$ be a relative symplectic cohomology of pairs with coefficients in $\Lambda_{\geq0}$. Then $SH^*(\cdot,\cdot;R):=SH^*(\cdot,\cdot)\otimes_{\Lambda_{\geq0}}R$ is a relative symplectic cohomology of pairs with coefficients in $R$. \qed
\end{coroll}

\subsection{Discussion}\label{ss:rel_SH_discussion}

\paragraph*{The product structure on $SH^*(\cdot)$.} In \cite{tonkonog2020super} Tonkonog and Varolg\"une\c s construct an associative skew-commutative product on $SH^*(K;\Lambda)$, and show that it possesses a unit, and that restriction maps are unital algebra morphisms. To this end they define so-called raised symplectic cohomology, and construct the product and the unit on the level of cohomology. They note in Remark 5.16, that it is possible to construct the product directly on $SH^*(K)$ without tensoring with $\Lambda$.

Such a product is constructed in Section \ref{ss:constructing_product}, together with a product on $SH(\cdot,\cdot;\Lambda_{\geq0})$ in a way which conforms to the Product axiom. Although this product on $SH^*(\cdot)$ cannot carry a unit, as explained in Remark 5.16 in \cite{tonkonog2020super}, it does have a similar structure, which upon tensoring with $\Lambda$ furnishes a unit. This structure consists of a family of elements $T^\lambda_K \in SH^0(K)$ for $\lambda>0$, with the property that $T^\lambda_K*\alpha = T^\lambda\alpha$ for $\alpha \in SH^*(K)$, and such that $\res^K_{K'}(T^\lambda_K) = T^\lambda_{K'}$ for $K'\subset K$. These elements are defined as follows: under the canonical isomorphism $SH^*(M) = H^*(M)\otimes \Lambda_{>0}$, let $T^\lambda_M \in SH^0(M)$ be the element corresponding to $1\otimes T^\lambda$, and put $T^\lambda_K:=\res^M_K(T^\lambda_M)$.

It is possible to show, using standard Floer-theoretic techniques, that the product
$$SH^*(M)\otimes SH^*(K) \to SH^*(K)$$
maps $T^\lambda_M\otimes\alpha \mapsto T^\lambda\alpha$. It then follows from the compatibility of the product with restrictions that
$$T^\lambda_K*\alpha=\res^M_K(T^\lambda_M)*\alpha = T^\lambda_M*\alpha=T^\lambda\alpha\,,$$
as claimed.

\paragraph*{The product structure on $SH^*(\cdot,\cdot)$ and the Mayer--Vietoris property.} Let us provide intuition for the fact that the descent property, which holds, for instance, when the sets Poisson commute, appears in the seemingly unrelated narrative about the product on symplectic cohomology of pairs. We employ an informal analogy between symplectic topology and basic algebraic topology via the following glossary: the relative symplectic cohomology corresponds to the singular cohomology, and the relative symplectic cohomology of a pair corresponds to the singular cohomology of a pair. Since the relative symplectic cohomologies of subsets in descent satisfy the Mayer-Vietoris property (see \cite{varolgunes2018mayer}), the starting point of our discussion is a pair of subspaces $A,B$ of a topological space $X$ satisfying the Mayer-Vietoris short exact sequence
$$
0 \rightarrow C_*(A\cap B) \rightarrow C_*(A)\oplus C_*(B) \rightarrow C_*(A \cup B) \rightarrow 0\,,$$
where $C_*$ stands for the singular complex. Compare it with the obvious short exact sequence
$$0 \rightarrow C_*(A\cap B) \rightarrow C_*(A)\oplus C_*(B) \rightarrow C_*(A)+C_*(B) \rightarrow 0\;,$$ where $C_*(A)+C_*(B) \subset C_*(X)$ stands for the sum of the subspaces $C_*(A),C_*(B) \subset C_*(X)$. From the long exact homology sequences and the 5-lemma, we see that the natural chain map
$$C_*(A)+C_*(B) \to C_*(A \cup B)$$
induces an isomorphism in homology. Thus $A,B$ is an {\it excisive pair} in the terminology of \cite[Definition 3.2]{davis2001lecture}. As explained \emph{ibid.}, for such a pair one has a well-defined cup product on relative cohomology
$$H^p(X,A) \otimes H^q(X,B) \to H^{p+q}(X, A \cup B)\;.$$
Roughly speaking, in Section \ref{ss:constructing_product} below we elaborate on the implication ``Varolg\"une\c s's Mayer--Vietoris $\Rightarrow$ product on $SH^*(\cdot,\cdot)$"
for subsets in descent in the context of symplectic cohomology, which is the subject of the product property stated above. The main technical difficulty is that various diagrams, including those containing the Mayer-Vietoris short exact sequence, commute only up to homotopy, forcing us to use tools of homotopical algebra.

\begin{rem}
We believe that for the regions appearing in Theorem \ref{thm:Liouville_dom_idx_bdd_SH} the existence of following exact triangle can be proved, relating $SH^*(M,K;\Lambda)$ to the $SH^+$-invariant of $K$ (see \cite{cieliebak2018symplectic}) and to $H^*(M,K;\Lambda)$:
\begin{equation*}
	\xymatrix{
		{H^*(M,K;\Lambda)} \ar[rr] &                                                  & {SH^*(M,K;\Lambda)} \ar[ld] \\
		                      & SH^{+,*}(\widehat{K};\Lambda)[-1] \ar[lu]^{+1} &
	}\,.
\end{equation*}
 Details will appear elsewhere.
\end{rem}

\section{Proof of Main Theorem \ref{thm:QH_IVQM_is_IVQM}}\label{s:proof_main_thm}

Let us recall the formulation of the theorem. For open $U \subset M$ put
$$\theta(U):=\ker \big(\res^M_{M\setminus U} \fc SH^*(M;\Lambda) \to SH^*(M\setminus U;\Lambda)\big) \subset SH^*(M;\Lambda) = QH^*(M)\,.$$
We have defined the set function $\tau$ with values in $\cI(QH^*(M))$ as follows:
$$\tau(K) = \bigcap_{U\text{ open}\supset K}\theta(U)\quad \text{and}\quad \tau(V) = \bigcup_{L\text{ compact}\subset V}\tau(L)$$
for compact $K \subset M$ and open $V\subset M$. The theorem asserts that $\tau$ is a $QH^*(M)$-IVQM.

\begin{rem}\label{rem:theta_monotone}
  Note that thanks to the functoriality of $SH^*$ with respect to inclusions, $\theta$ is monotone: $U\subset V$ implies $\theta(U)\subset \theta(V)$.
\end{rem}
\noindent Given an $\cI(QH^*(M))$-valued function $\eta$ defined on open subsets of $M$, its \tb{regularization} is
$$U\mapsto\bigcup\{\eta(U')\,|\,U'\text{ open with }\ol{U'}\subset U\}\,.$$
\noindent The following is a nice exercise which uses the definition of $\tau$, the monotonicity of $\theta$, and the fact that $M$ is a normal space:
\begin{lemma}
  The regularization of $\theta$ is $\tau$. \qed
\end{lemma}

%
%

Let us call a set function a \tb{weak $QH^*(M)$-IVQM} if it satisfies all the properties of a $QH^*(M)$-IVQM except continuity. Theorem \ref{thm:QH_IVQM_is_IVQM} is an immediate consequence of the following two propositions.
\begin{prop}\label{prop:theta_weak_IVQM}
  The function $\theta$ is a weak $QH^*(M)$-IVQM satisfying the invariance and vanishing properties.
\end{prop}
\begin{prop}\label{prop:regularization_weak_IVQM}
  The regularization of a weak $QH^*(M)$-IVQM is a $QH^*(M)$-IVQM. Moreover, regularization preserves invariance and vanishing.
\end{prop}
\noindent The rest of this section is dedicated to proving these propositions. We would like to point out that the most nontrivial property here is the quasi-multiplicativity of $\theta$, and this is where the product structure on the relative symplectic cohomology of pairs comes into play.

\begin{proof}[Proof of Proposition \ref{prop:theta_weak_IVQM}]
See Remark \ref{rem:theta_monotone} regarding the monotonicity of $\theta$. It remains to show that $\theta$ satisfies normalization, additivity, quasi-multiplicativity, intersection, invariance, and vanishing. We fix a relative symplectic cohomology of pairs with coefficients in $\Lambda$, $SH^*(\cdot,\cdot;\Lambda)$, whose existence is guaranteed by Theorem \ref{thm:existence_relative_SH_pairs} and Corollary \ref{cor:SH_pairs_change_coeffs} for the case $R = \Lambda$.

\tb{(Normalization):} Since $\res^M_M \fc SH^*(M;\Lambda)\to SH^*(M;\Lambda)$ is the identity, we conclude that $\theta(\varnothing)=0$; $\theta(M)=\ker(SH^*(M;\Lambda)\to SH^*(\varnothing;\Lambda)=0)=SH^*(M;\Lambda)$.	

\tb{(Additivity):} Let $U,V$ be disjoint open subsets, and let $A=M\setminus U$ and $B=M\setminus V$. Note that $A\cup B=M$. By definition we have
$$\theta(U)=\ker\res^M_{A},\,\,\theta(V)=\ker\res^M_{ B},\,\,\theta(U\cup V)=\ker\res^M_{A\cap B}.$$
Thus we need to show that
$$\ker\res^M_{A\cap B}=\ker\res^M_{A}+\ker\res^M_{B}.$$
Note that $A,B$ Poisson commute by Example \ref{ex:disjoint_sets_commute}. Since $M=A\cup B$, it follows that $SH^*(M,A\cup B;\Lambda) = SH^*(M,M;\Lambda) = 0$ by the Triangle axiom, and since $A,B$ Poisson commute, we can apply the Mayer--Vietoris sequence for $SH^*(\cdot,\cdot;\Lambda)$ to conclude that
$$f:=\res^{(M,A)}_{(M,A\cap B)} - \res^{(M,B)}_{(M, A\cap B)} \fc SH^*(M,A;\Lambda) \oplus SH^*(M,B;\Lambda) \to SH^*(M,A\cap B;\Lambda)$$
is an isomorphism. Consider the following commutative diagram with exact columns:
$$\xymatrix{SH^*(M,A;\Lambda) \oplus SH^*(M,B;\Lambda) \ar[rr]^-f \ar[d]_{\res^{(M,A)}_M \oplus \res^{(M,B)}_M} && SH^*(M,A\cap B;\Lambda) \ar[d]^{\res^{(M,A\cap B)}_M} \\ SH^*(M;\Lambda) \oplus SH^*(M;\Lambda) \ar[rr]^-{\pr_1-\pr_2} \ar[d]_{\res^M_A\oplus \res^M_B} && SH^*(M;\Lambda) \ar[d]^{\res^M_{A\cap B}} \\ SH^*(A;\Lambda)\oplus SH^*(B;\Lambda)\ar[rr]^-{\res^A_{A\cap B} - \res^B_{A\cap B}} && SH^*(A\cap B;\Lambda)}$$
where in the middle row $\pr_i$ is the projection to the $i$-th factor. The projections appear here because $\res^M_M = \id_{SH^*(M;\Lambda)}$. The commutativity of both squares follows from the functoriality of restriction maps with respect to inclusions---of compact sets for the bottom square, and of pairs of compact sets for the top square. The exactness of the columns follows from the Triangle axiom.

We now have
\begin{align*}
  \ker \res^M_{A\cap B} &= \im \res^{(M,A\cap B)}_M &&\text{by exactness}\\
                        &= \im \big(\res^{(M,A\cap B)}_M\circ f\big) && f\text{ is an isomorphism}\\
                        &= \im \big((\pr_1-\pr_2)\circ (\res^{(M,A)}_M \oplus \res^{(M,B)}_M)\big) &&\text{top square commutes}\\
                        &\stackrel{*}{=} \im \big(\res^{(M,A)}_M -\res^{(M,B)}_M \big)\\
                        &= \im \res^{(M,A)}_M + \im\res^{(M,B)}_M\\
                        &= \ker \res^M_A + \ker\res^M_B && \text{by exactness},
\end{align*}
as claimed. For $\stackrel{*}{=}$ we note that $(\pr_1-\pr_2)\circ \big(\res^{(M,A)}_M \oplus \res^{(M,B)}_M\big)=\res^{(M,A)}_M -\res^{(M,B)}_M$ as maps
$$SH^*(M,A;\Lambda) \oplus SH^*(M,B;\Lambda) \to SH^*(M;\Lambda)\,.$$
	
\tb{(Quasi-multiplicativity):} Let $U,V\subset M$ be open commuting subsets. Then their complements $M\setminus U,M\setminus V$ are compact commuting subsets. Let $a\in \theta(U)$ and $b\in\theta(V)$. By the Triangle property of a relative symplectic cohomology of pairs there are $x\in SH^*(M,M\setminus U;\Lambda)$ and $y\in SH^*(M,M\setminus V;\Lambda)$ such that $\res^{(M,M\setminus U)}_M(x)=a$ and $\res^{(M,M\setminus V)}_M(y)=b$. By the Product axiom we obtain
$$a*b=\res^{(M,M\setminus U)}_M(x)*\res^{(M,M\setminus V)}_M(y)=\res^{(M,M\setminus U\cup M\setminus V)}_M(x\wt*y)=\res^{(M,M\setminus (U\cap V))}_M(x\wt*y).$$
Using the Triangle property again, we obtain
$$\im \res^{(M,M\setminus (U\cap V)}_M =\ker \res^M_{M\setminus (U\cap V)}=\theta(U\cap V).$$
This implies that $a*b\in \theta(U\cap V)$, hence $\theta(U)*\theta(V)\subset \theta(U\cap V)$.

\tb{(Intersection):} Let $U,V$ be two open subsets which cover $M$. The diagram
$$\xymatrix{&SH^*(M;\Lambda) \ar[ld]|-{\res^M_{(U\cap V)^c}} \ar[rd]|-{\big(\res^M_{U^c},\res^M_{V^c}\big)}& \\ SH^*((U\cap V)^c;\Lambda) \ar[rr]_{\big(\res^{(U\cap V)^c}_{U^c},\res^{(U\cap V)^c}_{V^c}\big)}&& SH^*(U^c;\Lambda)\oplus SH^*(V^c;\Lambda)}$$
commutes thanks to the functoriality of restrictions with respect to inclusions. Since $U^c,V^c$ are disjoint compact subsets, they Poisson commute by Example \ref{ex:disjoint_sets_commute}. By the Mayer--Vietoris exact sequence (\cite{varolgunes2018mayer}) for $U^c,V^c$, the bottom map is an isomorphism, since the third group in the sequence is $SH^*(U^c\cap V^c;\Lambda) = SH^*(\varnothing;\Lambda)=0$. It follows that
\begin{align*}
  \theta(U\cap V) &= \ker \res^M_{(U\cap V)^c} \\
                  &= \ker \big(\res^M_{U^c},\res^M_{V^c}\big) && \text{from the diagram}\\
                  &= \ker \res^M_{U^c} \cap \ker \res^M_{V^c}\\
                  &= \theta(U) \cap \theta(V)\,,
\end{align*}
as required.

\tb{(Invariance):} As mentioned in \cite{varolgunes2018mayer}, a symplectomorphism $\phi$ of $M$ induces relabeling isomorphisms $\phi^K_* \fc SH^*(K;\Lambda)\cong SH^*(\phi(K);\Lambda)$, which commute with restrictions. It is not hard to show, using essentially Morse-theoretic arguments, that $\phi^M_*$ is the obvious action of $\phi$ on $SH^*(M;\Lambda) = H^{*\bmod 2N_M}(M;\Lambda)$, and that in particular for $\phi \in \Symp_0(M,\omega)$ we have $\phi^M_*=\id_{SH^*(M;\Lambda)}$, which implies for open $U\subset M$:
\begin{multline*}
\theta(\phi(U)) = \ker \res^M_{M\setminus \phi(U)} = \ker \big(\res^M_{M\setminus \phi(U)} \circ \phi^M_*\big)\\
 = \ker \big(\phi_*^{M\setminus U}\circ \res^M_{M\setminus U}\big) = \ker \res^M_{M\setminus U} = \theta(U)\,.
\end{multline*}

\tb{(Vanishing):} Let $K\subset M$ be a displaceable compact subset. Varolg\"une\c s proved that $SH^*(K;\Lambda)=0$, see  \cite[Theorem 1.3.1]{Varolgunes_2018_PhD}. This shows that
$$\theta(M\setminus K)=\ker\res^M_{M\setminus (M\setminus K)}=\ker\res^M_{K}=SH^*(M;\Lambda).$$
Additionally, since $K$ is a displaceable compact subset, there exists a displaceable region $W$ which contains $K$ in its interior. Let $Z$ be the closure of the complement of $W$ and denote $U=\Int(W)$. The compact sets $W,Z$ commute, and since $Z\cap W$ is displaceable, we have $SH^*(Z\cap W;\Lambda)=0$. Using Mayer-Vietoris we obtain
$$SH^*(M;\Lambda)=SH^*(Z\cup W;\Lambda)\xrightarrow[\cong]{\big(\res^M_Z,\res^M_W\big)} SH^*(Z;\Lambda)\oplus SH^*(W;\Lambda)=SH^*(Z;\Lambda).$$
Thus $\res^M_Z:SH^*(M;\Lambda)\to SH^*(Z;\Lambda)$ is an isomorphism, therefore
$$\theta(U)=\ker\res^M_{M\setminus U}=\ker\res^M_{Z}=0\,.$$	
\end{proof}

In order to prove Proposition \ref{prop:regularization_weak_IVQM}, we will introduce an auxiliary notion. Given an open $U\subset M$, an \tb{approximating chain} for $U$ is a sequence $\{U_j\}_j$ of open subsets of $U$ with $\ol{U_j}\subset U_{j+1}$ for all $j$ and such that $U = \bigcup_jU_j$. Note that every open set has an approximating chain. We will use the following elementary fact multiple times: if $K\subset U$ is compact and $\{U_j\}_j$ is an approximating chain for $U$, then there is $j$ such that $K\subset U_j$. In particular, any two approximating chains $\{U_j\}_j,\{U_j'\}_j$ dominate each other, meaning for each $j$ there is $j'$ such that $\ol{U_j}\subset U'_{j'}$ and vice versa.

\begin{proof}[Proof of Proposition \ref{prop:regularization_weak_IVQM}]
Let $\eta$ be a weak $QH^*(M)$-IVQM, and let $\sigma$ be its regularization. If $U\subset M$ is open and $\{U_j\}_j$ is an approximating chain for $U$, then it is easy to see that $\sigma(U) = \bigcup_j \eta(U_j)$. Since $QH^*(M)$ is finite-dimensional, the ascending chain of ideals $\{\eta(U_j)\}_j$ stabilizes, that is there is $I\in\cI(QH^*(M))$ such that $\eta(U_j) = I$ for all $j$ large enough. It follows that $\sigma(U) = I$. We can now prove that $\sigma$ is an IVQM.

\tb{(Normalization):} Trivial.

\tb{(Monotonicity):} Let $U,V\subset M$ be open and assume $U\subset V$. Let $\{U_j\}_j,\{V_j\}_j$ be approximating chains for $U,V$, respectively. We can pick $j$ such that $\sigma(U) = \eta(U_j)$, $\sigma(V)=\eta(V_j)$, and $U_j\subset V_j$. It follows that $\sigma(U)\subset \sigma(V)$ by the monotonicity of $\eta$.

\tb{(Continuity):} Let $W\subset M$ be open and let $$W^{(1)} \subset W^{(2)} \subset \cdots \subset W$$ be open subsets whose union equals $W$. Denote by $\{W^{(k)}_i\}$ an approximating chain for $W^{(k)}$. For every $k$, one can choose $i(k)$ large enough so that $\eta\big(W^{(k)}_{i(k)}\big) = \sigma(W^{(k)})$ and so that the sequence $\{W^{(k)}_{i(k)}\}$ is an approximating chain for $W$. It follows that $\sigma(W) = \sigma(W^{(k)})$ for $k$ large enough, and the desired continuity follows.

\tb{(Additivity):} Let $U,V\subset M$ be disjoint open sets and let $\{U_j\}_j,\{V_j\}_j$ be approximating chains for $U,V$, respectively. It is easy to show that $\{U_j\cup V_j\}_j$ is an approximating chain for $U\cup V$. There is $j$ such that $\sigma(U) = \eta(U_j)$, $\sigma(V) = \eta(V_j)$, $\sigma(U\cup V)=\eta(U_j\cup V_j)$. Since $U_j,V_j$ are disjoint, it follows from the additivity of $\eta$ that $\eta(U_j\cup V_j) = \eta(U_j) + \eta(V_j)$ and the claim follows.

\tb{(Quasi-multiplicativity):} Here we need have the following lemma, proved below.
\begin{lemma}\label{lem:Poisson_comm_sets_admit_comm_approx_chains}
  If $U,V\subset M$ are open commuting sets, then they have approximating chains $\{U_j\}_j,\{V_j\}_j$ such that $U_j,V_j$ commute for each $j$.
\end{lemma}
Assuming the lemma for the moment, let $U,V\subset M$ be open commuting sets and pick approximating chains as in the lemma. Since
$$\ol{U_j\cap V_j}\subset \ol{U_j}\cap \ol{V_j} \subset U_{j+1}\cap V_{j+1} \quad \text{and}\quad \bigcup_j(U_j\cap V_j) = \bigcup_j U_j \cap \bigcup_j V_j = U \cap V\,,$$
it follows that $\{U_j\cap V_j\}_j$ is an approximating chain for $U\cap V$, therefore there is $j$ such that $\sigma(U) = \eta(U_j)$, $\sigma(V) = \eta(V_j)$, $\sigma(U\cap V) = \eta(U_j\cap V_j)$, and thus
$$\sigma(U)*\sigma(V) = \eta(U_j)*\eta(V_j)\subset \eta(U_j\cap V_j) = \sigma(U\cap V)\,,$$
where the containment is thanks to the quasi-multiplicativity of $\eta$.

\tb{(Intersection):} Let $U,V\subset M$ be open sets which cover $M$ and let $\{U_j\}_j,\{V_j\}_j$ be approximating chains for them. We claim that there is $j$ such that $U_j,V_j$ cover $M$. To see this, note that $U^c\subset V$ is compact, therefore there is $k$ so that $U^c\subset V_k$, whence $V_k^c\subset U$, therefore there is $l$ so that $V_k^c\subset U_l$, which means that $V_k,U_l$ cover $M$, therefore \emph{a fortiori} $U_j,V_j$ cover $M$ for $j = \max\{k,l\}$.
We can increase $j$ so that in addition we have $\sigma(U) = \eta(U_j)$, $\sigma(V) = \eta(V_j)$. Since, as we have seen, $\{U_j\cap V_j\}_j$ is an approximating chain for $U\cap V$, we can further increase $j$ so that we also have $\sigma(U\cap V) = \eta(U_j\cap V_j)$. Since $U_j,V_j$ cover $M$, by the intersection property of $\eta$ we have
$$\sigma(U\cap V) = \eta(U_j\cap V_j) = \eta(U_j) \cap \eta(V_j) = \sigma(U) \cap \sigma(V)\,.$$

\tb{(Invariance):} If $U\subset M$ is open, $\{U_j\}_j$ is an approximating chain for $U$, and $\phi \in \Symp_0(M)$, then $\{\phi(U_j)\}_j$ is an approximation chain for $\phi(U)$. We can pick $j$ so that $\sigma(U) = \eta(U_j)$, $\sigma(\phi(U)) = \eta(\phi(U_j))$, and the invariance of $\sigma$ follows from that of $\eta$.

\tb{(Vanishing):} Let $K\subset M$ be displaceable and let $\{U_j\}_j$ be an approximating chain for $M \setminus  K$. It follows that for $j$ large enough, $M\setminus U_j$ is displaceable, whence $\eta(U_j) = \eta\big(M\setminus (M\setminus U_j)\big) = QH^*(M)$, and taking $j$ large enough so that $\sigma(M \setminus  K) = \theta(U_j)$ we obtain $\sigma(M \setminus  K) = QH^*(M)$. For the second part of the vanishing property, let $U\supset K$ be open such that $\eta(U) = 0$. If $V$ is an open set with $\ol V\subset U$, then by the monotonicity of $\eta$ we have $\eta(V) = 0$, whence
$$\sigma(U) = \bigcup\{\eta(V)\,|\,V\text{ open with }\ol V\subset U\} = 0\,.$$
\end{proof}
\noindent It remains to prove Lemma \ref{lem:Poisson_comm_sets_admit_comm_approx_chains}.
\begin{proof}[Proof of Lemma \ref{lem:Poisson_comm_sets_admit_comm_approx_chains}] By definition there are Poisson-commuting $f,g\in C^\infty(M,[0,1])$ such that $M\setminus U=f^{-1}(0),\,M\setminus V=g^{-1}(0)$. For every $\alpha\in (0,1]$ consider the open sets $U_\alpha=f^{-1}((\alpha,1])$ and $V_\alpha=g^{-1}((\alpha,1])$. Note that for every $0<\beta<\alpha\leq 1$ we have $\overline{U_\alpha}\subset U_\beta$ and $\overline{V_\alpha}\subset V_\beta$, moreover,     $$U=\bigcup_{\alpha\in (0,1]} U_\alpha,\,\,\,\text{and}\,\,\,V=\bigcup_{\alpha\in (0,1]} V_\alpha\,.$$
Thanks to Remark \ref{rem:preimages_involutive_maps}, $U_\alpha,V_\alpha$ commute. It follows that $\{U_j:=U_{1/j}\}_{j \in \N}$ and $\{V_j:=V_{1/j}\}_{j \in \N}$ are approximating chains as required.
\end{proof}

\section{Constructing a symplectic cohomology of pairs}\label{s:proofs_properties_SH_pairs}

In this section we construct a relative symplectic cohomology of pairs with coefficients in $\Lambda_{\geq0}$ and prove its properties as announced in Section \ref{s:relative_SH_pairs}, thereby proving Theorem \ref{thm:existence_relative_SH_pairs}. The definition appears in Section \ref{sss:rel_SH_pairs}. The product is constructed in Section \ref{ss:constructing_product}, and it is the highlight of this section and of the whole paper. The proof of the rest of the properties occupies Section \ref{ss:SH_pair_well_defined}. This relies on the algebraic language of cubes, which we outline in Sections \ref{ss:cubes}--\ref{ss:rays_telescopes}. The language of cubes was introduced by Varolg\"une\c s in \cite{varolgunes2018mayer}. It is an explicit realization of some $\infty$-categorical aspects of cochain complexes.

\subsection{Completion for $\Lambda_{\geq0}$-modules}\label{ss:completion}

Completion of $\Lambda_{\geq0}$-modules plays an essential role in Varolg\"une\c s's definition of symplectic cohomology, therefore we describe it here. First, the Novikov field $\Lambda$ carries a valuation $\nu \fc \Lambda \to \R\cup\{+\infty\}$ given by $\nu(0) = \infty$ and
$$\nu\Big(\textstyle \sum_{i=1}^{\infty}c_iT^{\lambda_i}\Big) = \lambda_1$$
provided the $\lambda_i$ form a strictly increasing sequence and $c_1 \neq 0$. Using $\nu$, we can define the \tb{interval modules}
$$\Lambda_{\geq r} = \nu^{-1}([r,\infty])\,,\qquad \Lambda_{>r}=\nu^{-1}((r,\infty])\,,\qquad \Lambda_{[a,b)}=\Lambda_{\geq a}/\Lambda_{\geq b}\text{ for }a< b\,.$$
Given a $\Lambda_{\geq0}$-module $A$, its \tb{completion} is the $\Lambda_{\geq 0}$-module
$$\wh A:=\varprojlim_{r\to\infty}\big(A\otimes\Lambda_{[0,r)}\big)\,.$$
Completion is an endofunctor on the category of $\Lambda_{\geq 0}$-modules. From the universal property of inverse limits we obtain a morphism
$$A \to \wh A\,,$$
also called \tb{completion}. We say that $A$ is \tb{complete} if the completion morphism is an isomorphism. The interval modules are typical examples. On the other hand, $\wh\Lambda = 0$. Since completion commutes with finite direct sums, a finite direct sum of interval modules is likewise complete.

We will use the following fact.
\begin{lemma}\label{lem:completion_exact_right_mod_flat}
  Let
  $$0 \to A' \to A \to A'' \to 0$$
  be an exact sequence of $\Lambda_{\geq0}$-modules, where $A''$ is flat. Then the corresponding sequence of completions is exact.
\end{lemma}
\begin{proof}
  Since $A''$ is flat, we have $\Tor_1^{\Lambda_{\geq0}}\big(A'',\Lambda_{[0,r)}\big)=0$ for all $r > 0$, therefore from the long exact sequence of Tor functors we conclude that each sequence
  \begin{equation}\label{eq:exact_seq_tensored_Lambda_0_r}
    0 \to A' \otimes\Lambda_{[0,r)} \to A \otimes\Lambda_{[0,r)}\to A'' \otimes\Lambda_{[0,r)}\to 0
  \end{equation}
  is exact. Since for $s \geq r$ the map $\Lambda_{[0,s)} \to \Lambda_{[0,r)}$ is surjective, and since the tensor product preserves surjectivity, we see that $A' \otimes\Lambda_{[0,s)} \to A'\otimes\Lambda_{[0,r)}$ is also onto, which means that the system of modules $\big(A'\otimes\Lambda_{[0,r)}\big)_r$ satisfies the Mittag--Leffler condition, which implies that the inverse limit of the sequences \eqref{eq:exact_seq_tensored_Lambda_0_r} is likewise exact, which is what the lemma asserts.
\end{proof}

\medskip

\noindent\tb{Convention:} \emph{We extend the completion functor to the categories of graded $\Lambda_{\geq0}$-modules and chain complexes by completing degree-wise.}

\subsection{Cubes}\label{ss:cubes}

The language of cubes is very convenient in order to work with the algebraic data arising when defining relative symplectic cohomology. Throughout we fix a unital commutative ring $R$ and we work in the category of graded $R$-modules and graded maps, where the grading is over $\Z$ or over $\Z_k$ with $k$ even. Recall that given graded modules $C,D$ with graded components $C^i,D^i$, a module map $f \fc C \to D$ is \tb{graded of degree $d$} if $f(C^i) \subset D^{i+d}$. We denote the degree of $f$ by $|f|$.

Fix a nonnegative integer $n$. We call a subset of $[0,1]^n$ a \tb{face} if it is given by setting some of the coordinates to either $0$ or $1$; the rest of the coordinates are referred to as the \tb{free coordinates of $F$}. Given a face $F$ its \tb{dimension}, denoted $|F|$, is the number of its free coordinates, while its \tb{initial vertex} $\ini F$ and \tb{terminal vertex} $\ter F$, are the points of $F$ closest to or farthest from the origin, respectively (relative to the Euclidean metric); in this case we write $F \fc v \to v'$, where $v = \ini F$, $v' = \ter F$. Note that $F$ is determined by its initial and terminal vertices. We also say that $v,v'$ \tb{span} $F$. If $F',F''$ are faces, we write $F = F'\cdot F''$ to denote the situation in which
$$\ini F = \ini F'\,, \quad \ter F' = \ini F''\,,\quad \ter F'' = \ter F\,.$$

\tb{An (algebraic) $n$-cube $\cC$} is a pair $(\{\cC^v\}_{v\in\{0,1\}^n},\{f^\cC_F\}_{F\subset [0,1]^n\text{ a face}})$, where each $\cC^v$ is a graded module, and for every face $F \subset [0,1]^n$ we have a graded module morphism $f^\cC_F \fc \cC^{\ini F} \to \cC^{\ter F}$ of degree $1-|F|$, subject to the condition that for each face $F$ we have
\begin{equation}\label{eqn:cube_relation}
\sum_{F = F'\cdot F''}(-1)^{|F'|}\sgn(F',F'')f^\cC_{F''} f^\cC_{F'} = 0\,,
\end{equation}
where $\sgn(F',F'')$ is a sign defined as follows. Any face of $[0,1]^n$ comes equipped with a natural orientation coming from the ordering of its free coordinates. Then $\sgn(F',F'')$ is the intersection index of $F',F''$ inside $F$. For a more explicit description, assume first that we have a linearly ordered finite set $S=\{s_1<\dots<s_{k+l}\}$. Recall that a $(k,l)$-shuffle on $S$ is a permutation $\sigma \in S_S$ such that $\sigma(s_1)<\dots<\sigma(s_k)$ and $\sigma(s_{k+1})<\dots<\sigma(s_{k+l})$. If $S=S'\sqcup  S''$ with $|S'| = k$, $|S''|=l$, then there is a unique $(k,l)$-shuffle $\sigma_{S',S''}$ such that $\sigma(\{s_1,\dots,s_k\})=S'$, $\sigma(\{s_{k+1},\dots,s_{k+l}\})=S''$. Now let $S\subset \{1,\dots,n\}$ be the set of free coordinates of $F$, $S'$ the set of free coordinates of $F'$ and $S''$ the set of free coordinates of $F''$, so that $S = S'\sqcup S''$. Endow $S$ with the order induced from $\{1,\dots,n\}$. Then $\sgn(F',F'') := \sgn \sigma_{S',S''}$.

The above relations mean in particular that for each vertex $v$, $(\cC^v,f^\cC_{\{v\}})$ is a cochain complex, each $1$-dimensional edge yields a cochain map between its vertex modules, each $2$-dimensional face yields a homotopy between the two compositions of the maps running along its perimeter, and so on.

Note that the above sign only depends on the internal ordering of the free coordinates of $F',F''$, which means that for any face $G\subset [0,1]^n$ the pair $(\{\cC^v\}_{v},\{f^\cC_F\}_F)$ where $v$ ranges over the vertices of $G$ while $F$ over the subfaces of $G$, is itself a cube, provided we renumber the coordinates in their natural order in $\{1,\dots,n\}$. We will refer to such a cube as a \tb{subcube of $\cC$}, or, when the face $G$ is given, \tb{the subcube obtained by restricting $\cC$ to $G$}, and we denote it by $\cC|_G$.

If $\cA_0,\cA_1$ are $n$-cubes, a \tb{map from $\cA_0$ to $\cA_1$} is an $(n+1)$-cube $\cC$ such that $\cA_i = \cC|_{[0,1]^n\times\{i\}}$, $i=0,1$. In this case we write $\cA_0\xrightarrow{\cC}\cA_1$. We define the corresponding negated map $\cA_0\xrightarrow{-\cC}\cA_1$ by negating all the maps going from the vertices of $\cA_0$ to those of $\cA_1$. A trivial check shows that this is indeed a map. Note that if $F\subset [0,1]^n$ is a face and
$\cC' = \cC|_{F\times[0,1]},\cA'_i=\cA_i|_{F}=\cC|_{F\times\{i\}}$, $i=0,1$, then $\cC'$ is a map from $\cA_0'$ to $\cA_1'$: $\cA_0'\xrightarrow{\cC'}\cA_1'$.

A \tb{partial cube} is given by a collection of modules associated to some of the vertices of $[0,1]^n$, as well as maps between them corresponding to some of the faces of $[0,1]^n$, subject to the condition that the cube relation \eqref{eqn:cube_relation} is satisfied whenever all the maps in it are defined. This notion will be useful in Section \ref{ss:folding_pullbacks}, where we define pullbacks for V-shaped diagrams, which are a particular case of partial cubes.

\subsection{Direct sums and tensor products}\label{ss:direct_sums_tensor_products_cubes}

In various constructions related to the definition and proof of properties of the symplectic cohomology of a pair we will need direct sums and tensor products of cubes.

If $\cA,\cB$ are $n$-cubes, their \tb{direct sum} $\cA \oplus \cB$ is defined as the cube with vertices $(\cA\oplus\cB)^v = \cA^v\oplus \cB^v$ and face maps $f^{\cA\oplus\cB}_F = f^\cA_F\oplus f^\cB_F$. A trivial verification shows that this is indeed a cube. Note that in particular taking direct sums commutes with passing to subcubes, thus restricting $\cA\oplus \cA' \xrightarrow{\cC\oplus \cC'}\cB\oplus \cB'$ to a face $\wh F = F\times[0,1]$ yields $\cA|_F\oplus \cC'|_F \xrightarrow{\cC|_{\wh F}\oplus \cC|_{\wh F}}\cB|_F\oplus \cB'|_F$.

Before we define tensor products of cubes, let us recall the notion of graded tensor product of graded maps: if $V,V',W,V'$ are graded modules and if $f \fc V \to W$, $f' \fc V'\to W'$ are graded maps, their \tb{graded tensor product} $f\otimes f' \fc V\otimes V' \to W\otimes W'$ is defined by $(f\otimes f')(v\otimes v') = (-1)^{|f'||v|}f(v)\otimes f'(v')$. \emph{All the tensor products of graded maps below are taken in the graded sense.}

If $(C,d)$, $(C',d')$ are modules endowed with differentials, that is degree $1$ maps squaring to zero, their tensor product is the graded module $(C\otimes C',d\otimes \id_{C'}+\id_{C}\otimes d')$, where the tensor product is taken in the graded sense. This amounts to the usual Leibnitz rule for differentials on the tensor product of graded modules.

Let now $\cA,\cB$ be a $k$-cube and an $l$-cube, respectively. We are going to define their \tb{tensor product} $\cA\otimes \cB$, which is a $(k+l)$-cube, as follows. The vertex modules are
$$(\cA\otimes\cB)^{i_1\dots i_{k+l}} = \cA^{i_1\dots i_k}\otimes \cB^{i_{k+1}\dots i_{k+l}}\,.$$
The face maps $f^{\cA\otimes\cB}_F$ are defined as follows. If $F=\{(v_1,\dots,v_k,w_1,\dots,w_l)\}$ has dimension zero and we let $v=(v_1,\dots,v_k)$, $w=(w_1,\dots,w_l)$, then the corresponding face map, that is the differential on the module $\cA^v\otimes \cB^w$, is simply the differential on the tensor product (in the graded sense!):
$$f^{\cA\otimes\cB}_F = f^\cA_{\{v\}}\otimes \id_{\cB^w} + \id_{\cA^v}\otimes f^\cB_{\{w\}}\,.$$

If $F$ has positive dimension, let $F'\subset [0,1]^k$, $F''\subset [0,1]^l$ be faces such that $F = F'\times F''\subset [0,1]^k\times[0,1]^l = [0,1]^{k+l}$. We have three cases:
\begin{enumerate}
  \item $|F'|,|F''|>0$: in this case $f^{\cA\otimes\cB}_F = 0$.
  \item $|F''| = 0$: let $F'' = \{w\}$; then
      $$f^{\cA\otimes\cB}_F = f^\cA_{F'}\otimes \id_{\cB^w} \fc \cA^v\otimes \cB^w \to \cA^{v'}\otimes \cB^w\,,$$
      where $v=\ini F',v'=\ter F'$;
  \item $|F'| = 0$: let $F'=\{v\}$; then
      $$f^{\cA\otimes\cB}_F = \id_{\cA^v} \otimes f^\cB_{F''} \fc \cA^v\otimes \cB^w \to \cA^v \otimes \cB^{w'}\,,$$
      where $w = \ini F''$, $w' = \ter F''$.
\end{enumerate}
The following is obtained by unraveling the definitions:
\begin{prop}
  The tensor product of cubes is a cube. \qed
\end{prop}
We will also use a more general tensor product corresponding to a $(k,l)$-shuffle $\sigma \in S_{k+l}$, denoted $\cA\otimes_\sigma\cB$. The above case corresponds to the identity shuffle. The vertex modules of the general tensor product are given by
$$(\cA \otimes_\sigma \cB)^{i_1\dots i_{k+l}} = \cA^{i_{\sigma(1)}\dots i_{\sigma(k)}} \otimes \cB^{i_{\sigma(k+1)}\dots i_{\sigma(k+l)}}\,,$$
and the face maps are defined analogously to the case $\sigma=\id$, except we need to renumber coordinates according to $\sigma$. The formulas are exactly the same, but the notation becomes more complicated.

We will need special cases of the tensor product construction: tensoring with a chain complex, the identity map, the diagonal map and the sum map.

\paragraph*{Tensoring with a chain complex} Let $\cC$ be a cube and let $(A,d)$ be a chain complex, that is a $0$-cube. We can then form their tensor product $\cC \otimes A$. Its vertices are $(\cC \otimes A)^v = \cC^v\otimes A$, the differential on such a vertex is the above tensor of the respective differentials, and if $F$ is a positive-dimensional face of $[0,1]^n$, then $f^{\cC\otimes A}_F = f^\cC_F\otimes \id_A$. We can likewise form the tensor product $A \otimes \cC$, which differs from $\cC\otimes A$ by signs of the differentials.

The simplest example is as follows. Let $R$ be viewed as a graded module over itself, concentrated in degree zero, and given the zero differential. Then, using the canonical isomorphism $C \otimes_R R = C = R \otimes_R C$ for an $R$-module $C$, we obtain canonically $\cC \otimes R = \cC = R \otimes \cC$.

Similarly we can consider the module $R \oplus R$ sitting in degree zero, also with the zero differential. In this case, given a cube $\cC$, we have a canonical isomorphism $\cC \oplus \cC\equiv \cC \otimes (R\oplus R)$.

\paragraph*{The identity map} Given an $n$-cube $\cC$ and a direction $i\in\{1,\dots,n+1\}$, the \tb{identity map in the $i$-th direction} is the tensor product $\cC \otimes_{\sigma_i}\id_R$, where $\sigma_i$ is the unique $(n,1)$-shuffle mapping $n+1\mapsto i$, and $\id_R \fc R \to R$ is the identity map, viewed as a $1$-cube. The most common case we will need is when $i=n+1$, in which case $\sigma = \id$. In this case we will denote $\cC \otimes \id_R \equiv \cC \xrightarrow{\id}\cC$. Note that this is a map from $\cC$ to itself in the above sense.

Note that if $F\subset [0,1]^n$ is a face, then the restriction of the identity map $\cC \xrightarrow{\id}\cC$ to the face $F \times [0,1]$ yields the identity map $\cC|_F\xrightarrow{\id}\cC|_F$.

\paragraph*{The diagonal map} Consider the diagonal map $\Delta_R \fc R \to R \oplus R$ as a $1$-cube. If $\cC$ is an $n$-cube, the corresponding \tb{diagonal map} is the $(n+1)$-cube $\cC \otimes \Delta_R$. Computing, we see that it is a map $\cC \xrightarrow{\Delta_\cC}\cC \oplus \cC$, where the only nonzero maps in the $(n+1)$-st direction are the diagonal maps $\cC^v \to \cC^v\oplus \cC^v$.

\paragraph*{The sum map} Consider now the sum map $\Sigma \fc R \oplus R \to R$, again considered as a $1$-cube. Tensoring a cube $\cC$ with it, we obtain, after suitable identifications, the \tb{sum map} $\cC \oplus \cC \xrightarrow{\Sigma_\cC}\cC$, whose only nonzero horizontal maps are sums $\cC^v\oplus \cC^v \to \cC^v$.

\subsection{Cones, cocones}\label{ss:cones_cocones}

In homotopy theory, cones and cocones afford explicit models for homotopy cokernels and homotopy kernels, respectively. Here we review the corresponding notions for cubes, as defined in \cite{varolgunes2018mayer}.

Fix $i \in \{1,\dots,n\}$. Identify $\R^{n-1}$ with the hyperplane $\{x_i=0\} \subset \R^n$, let $\pi \fc \R^n \to \R^{n-1}$ be the map turning the $i$-th coordinate to $0$, and let $\iota_j \fc \R^{n-1}\to\R^n$ be the inclusions $(x_1,\dots,0,\dots,x_n) \mapsto (x_1,\dots,j,\dots,x_n)$ for $j=0,1$. For $\ol F \subset [0,1]^{n-1}\subset \R^{n-1}$ let $F = (\pi|_{[0,1]^n})^{-1}(\ol F)$, $F_j = \iota_j(\ol F)$. Let $\cC$ be an $n$-cube. Its \tb{cone in the $i$-th direction} is the $(n-1)$-cube $\cone_i\cC$, defined as follows. For a vertex $v \in \{0,1\}^{n-1}$ we have
$$(\cone_i\cC)^v = \cC^{\iota_0(v)}[1]\oplus \cC^{\iota_1(v)}\,.$$
For a face $\ol F \subset [0,1]^{n-1}$ the corresponding face map $f^{\cone_i\cC}_{\ol F}$ is given by the triangular matrix
$$\left(\begin{array}{rl} - (-1)^{|F_0|}f^\cC_{F_0} & 0 \\ -(-1)^{\sharp(i,F)}f^\cC_F & f_{F_1} \end{array}\right)\,,$$
where $\sharp(i,F)$ denotes the position of $i$ in the set of free coordinates of $F$ relative to the order induced from $\{1,\dots,n\}$. Note in particular that the differentials are negated in the shifted modules, while those in unshifted ones retain their sign, and that face maps corresponding to edges always come with the positive sign.

The following is obtained from the definitions:
\begin{prop}
  $\cone_i\cC$ is a cube. \qed
\end{prop}

It also follows that the cone operation behaves well with respect to restrictions: if $F\subset [0,1]^n$ is a face containing the $i$-the direction as one of its free coordinates, then for any $n$-cube $\cC$ we have
$$(\cone_i\cC)|_{\pi(F)} = \cone_{\sharp(i,F)}(\cC|_F)\,.$$

We will need the $k$-th iterated cone in the first direction, $\cone^{\circ k}:=\cone_1\circ\dots\circ \cone_1$, which maps $n$-cubes into $(n-k)$-cubes. 

\begin{defin}
  Let us call an $n$-cube \tb{acyclic} if its $n$-th iterated cone is an acyclic complex.
\end{defin}

We also need to discuss a particular property of cones, also described in \cite{varolgunes2018mayer}. Given a module $A$, we call it \tb{$k$-coniform} if it comes with a decomposition into $2^k$ submodules:
$$A = \bigoplus_{i_1,\dots,i_k=0}^1A^{i_1\dots i_k}\,.$$
Given two such modules $A,B$, we say that a map $f \fc A \to B$ is \tb{$k$-coniform} if the composition
$$A^{i_1\dots i_k} \to A \xrightarrow{f} B \to B^{i_1'\dots,i_k'}$$
vanishes unless $i_j\leq i_j'$ for $j=1,\dots,k$; here the first map is the inclusion while the last one is the projection. Finally, we say that a cube $\cC$ is \tb{$k$-coniform} if so is every vertex module and every face map of $\cC$. The very definition of cones implies the following
\begin{prop}
  The iterated cone operation $\cone^{\circ k}$ establishes a bijective correspondence between $n$-cubes and $k$-coniform $(n-k)$-cubes. In particular, given an $n$-coniform cochain complex $(A,d)$, there is a well-defined $n$-cube $(\cone^{\circ n})^{-1}A$. \qed
\end{prop}

If $\cC$ is a cube, we let $\cC[k]$ be the cube whose vertex modules have been shifted by $k$ in degree, and whose structure maps have all been multiplied by $(-1)^k$. It is trivially a cube. We define the \tb{cocone of $\cC$ in the $i$-th direction} by
$$\co_i\cC:=\cone_i\cC[-1]\,.$$
Combining the fact that shifts commute with passage to subcubes, we see that cocones are also compatible with it, namely if $F\subset [0,1]^n$ is a face containing the $i$-th direction as one of its free coordinates, then
$$(\co_i\cC)|_{\pi(F)} = \co_{\sharp(i,F)}(\cC|_F)\,.$$

\begin{rem}\label{rem:cocones_htpy_ker_square}
  Let $f \fc A \to B$ be a cochain map. As we mentioned, its cocone $\co (f)$ is an explicit model of its homotopy kernel. In particular by the defining property of homotopy kernels, there is a natural map $\co(f) \to A$ whose composition with $f$ is nullhomotopic. Explicitly, this map $\co(f) = A\oplus B[-1] \to A$ is simply the projection to $A$. An explicit nullhomotopy is given by $h \fc \co(f) \to B$, $h(a,b) = -b$. This can be expressed in the form of the following $2$-cube:
  $$\xymatrix{\co(f) \ar@{=}[r] \ar[d] \ar[rd]^{h} & \co(f) \ar[d]^{f\circ\pr} \\ 0 \ar[r] & B}$$
\end{rem}

\begin{rem}\label{rem:cocones_turn_ids_minus_ids} This is a rather lengthy remark, relevant to the constructions described in Sections \ref{ss:constructing_product}, \ref{ss:SH_pair_well_defined}.
  Although a cocone of a cube is again a cube, it has some peculiar properties. For instance if $\cC$ is a cube which is the identity map in a direction $i$, its cocone in any other direction is rather \emph{minus} the identity map. This is because any cone in a direction other than $i$ is the identity, and in the cocone all the maps are negated. The main consequence for us is that when trying to define the symplectic cohomology of a pair of subsets, the comparison maps between the various homologies, coming from cocones of certain maps, need to be negated in order to form a direct system whose direct limit we need to take. The reader is invited to refer back to this remark when this happens.

  In practice, this takes the following form. Varolg\"une\c s defines special kinds of cubes: triangles and slits. An $n$-triangle is an $n$-cube of the form
  $$\xymatrix{\cA \ar[r]\ar[d]_{\id} \ar[rd] & \cB \ar[d] \\ \cA \ar[r]& \cC}$$
  where the diagram directions are the last two ones, so that $\cA,\cB,\cC$ are the corresponding subcubes of dimension $n-2$. If we take its cocone in any direction other than the last two ones, we obtain a cube of the form
  $$\xymatrix{\co_i \cA \ar[r]\ar[d]_{-\id} \ar[rd] & \co_i\cB \ar[d]\\ \co_i\cA \ar[r]& \co_i\cC}$$
  which is not a triangle \emph{because of the minus sign} in $-\id$. We will only need to use the case when this last cube is actually a square:
  $$\xymatrix{A \ar[r]^f\ar[d]_{-\id} \ar[rd]^h & B \ar[d]^g\\ A \ar[r]_k& C}$$
  This cube simply means that $g\circ f$ and $-k$ are homotopic. Another way of saying this is that $(-g)\circ(-f)$ and $-k$ are homotopic, that is we have the following square:
  $$\xymatrix{A \ar[r]^{-f}\ar[d]_{\id} \ar[rd]^h & B \ar[d]^{-g}\\ A \ar[r]_{-k}& C}$$
  which \emph{is} in fact a $2$-triangle. Thus we see that if we have a $3$-triangle and we take its cocone in the first direction, we obtain a square which can be modified as above to obtain a triangle, which is equivalent to a homotopy commutative triangle involving maps which are negated.

  A similar remark applies to slits. An $n$-slit is an $n$-cube of the form
  $$\xymatrix{\cA \ar[r]\ar[d]_{\id} \ar[rd] & \cB \ar[d]^{\id} \\ \cA \ar[r]& \cB}$$
  It can be thought of as two maps $\cA\to\cB$ and a homotopy between them. If we have a $3$-slit like this, taking its cocone in the first direction results in a square of the form
  $$\xymatrix{A \ar[r]^{f}\ar[d]_{-\id} \ar[rd]^h & B \ar[d]^{-\id}\\ A \ar[r]_{k}& B}$$
  or, equivalently, a square of the form
  $$\xymatrix{A \ar[r]^{-f}\ar[d]_{\id} \ar[rd]^h & B \ar[d]^{\id}\\ A \ar[r]_{-k}& B}$$
  which \emph{is} a $2$-slit. These squares express the fact that $f,k\fc A \to B$ are homotopic maps, which is all we need.
\end{rem}

\subsection{Compositions of maps of cubes}\label{ss:compositions_maps_cubes}

In the construction of the product on the relative symplectic cohomology of pairs in Section \ref{ss:constructing_product} we'll need to use compositions of maps of cubes. Recall that a map between two $n$-cubes is simply an $(n+1)$-cube whose restrictions to the corresponding $n$-faces are the given $n$-cubes.

Let $\cA\xrightarrow{\cF}\cB\xrightarrow{\cG}\cC$ be two maps of $n$-cubes. We will define their composition $\cA \xrightarrow{\cG\circ\cF}\cC$ as follows. Taking the $n$-th iterated cone in the first direction, we arrive at a sequence of chain complexes and chain maps
$$\cone^{\circ n} \cA \xrightarrow{\cone^{\circ n}\cF}\cone^{\circ n}\cB\xrightarrow{\cone^{\circ n}\cG}\cone^{\circ n}\cC\,.$$
The chain maps are $n$-coniform, as is their composition $\cone^{\circ n}\cG\circ\cone^{\circ n}\cF$, therefore we can apply $(\cone^{\circ n})^{-1}$ to the chain map
$$\cone^{\circ n}\cA\xrightarrow{\cone^{\circ n}\cG\circ\cone^{\circ n}\cF}\cone^{\circ n}\cC\,,$$
and the result is the desired composition
$$\cA \xrightarrow{\cG\circ\cF}\cC\,.$$

The following material will be relevant in Section \ref{ss:SH_pair_well_defined}, where we prove that the relative symplectic cohomology of a pair is well-defined, as well as in Section \ref{ss:constructing_product}, where we construct the product.

If $F \subset [0,1]^n$ is a face, we let $\wh F:= F \times [0,1] \subset [0,1]^{n+1}$. The main property of compositions we'll need is that the $1$-dimensional edges in the direction of the composition simply compose and they acquire no signs. That is if $v$ is a vertex, then
$$f^{\cG\circ\cF}_{\wh {\{v\}}} = f^{\cG}_{\wh {\{v\}}}\circ f^{\cF}_{\wh {\{v\}}}\,,$$

Let us call a map of $n$-cubes $\cA \xrightarrow{\cF}\cB$ \tb{straight} if, whenever $v,v' \in \{0,1\}^n$ are vertices and $F\fc v \to v'$ is the corresponding face, the map $f^\cF_{\wh F}$ vanishes unless $v=v'$. Examples of straight maps include the identity, diagonal, and the sum map on a given cube, and more generally the tensor product of any cube with a $1$-cube.

Another property of straight maps we'll use in the sequel is as follows. Assume that $\cA \xrightarrow{\cF}\cB$ is a straight map between two $n$-cubes. It follows that $\cF$ is completely determined by the structure maps of $\cA,\cB$ and the maps $\wh f^\cF_v:=f^\cF_{\wh {\{v\}}} \fc \cA^v \to \cB^v$ for the vertices $v$ of $[0,1]^n$. If we are given two cubes $\cA,\cB$ and a collection of maps $\wh f_v \fc \cA^v \to \cB^v$, they are the structure maps of a straight map $\cA \to \cB$ if and only if for every face of $[0,1]^{n+1}$ of the form $\wh F$, for $F \subset [0,1]^n$, we have
$$f^\cB_F \circ \wh f_{\ini F} = \wh f_{\ter F}\circ f^\cA_F\,.$$
We will apply this fact in the following form. Let $\cA\xrightarrow{\cF}\cB$ be a map of $n$-cubes and let $\wh\cF = \cone_{n+1}\cF$ be the corresponding cone, which is an $n$-cube. We claim that there are natural straight maps
$$\cB \xrightarrow{\iota}\wh\cF \xrightarrow{\pi}\cA[1]\,.$$
These are defined as follows. Let $v \in \{0,1\}^n$ be a vertex. Then the defining morphisms of the straight maps $\iota,\pi$ are as follows:
$$\iota_v \fc \cB^v \to \wh\cF^v = \cA^v[1]\oplus \cB^v\quad\text{is the inclusion map}\,,$$
$$\pi_v \fc \wh\cF^v = \cA^v[1]\oplus \cB^v \to \cA^v[1]\quad\text{is }(-1)^{\|v\|_1}\cdot\text{the projection map}\,,$$
where $\|v\|_1$ is the $\ell_1$-norm of $v$, that is the number of coordinates of $v$ equal to $1$.

Shifting everything by $-1$, we see that we also have natural straight maps
\begin{equation*}
\cB[-1] \xrightarrow{\iota[-1]}\wh\cF[-1]=\co_{n+1}\cF \xrightarrow{\pi[-1]}\cA\,,
\end{equation*}
which we'll use in the definition of relative symplectic cohomology of pairs. We'll also need another feature of this sequence, namely its exactness. Let's call a sequence
$$0\to\cA\xrightarrow{\cF}\cB\xrightarrow{\cG}\cC\to 0$$
of maps between $n$-cubes \tb{exact} if the following sequence is:
$$0\to\cone^{\circ n}\cA\xrightarrow{\cone^{\circ n}\cF}\cone^{\circ n}\cB\xrightarrow{\cone^{\circ n}\cG}\cone^{\circ n}\cC\to 0\,.$$
We claim that
\begin{equation}\label{eqn:exact_sequence_cocone}
0\to \cB[-1] \xrightarrow{\iota[-1]}\co_{n+1}\cF \xrightarrow{\pi[-1]}\cA\to 0
\end{equation}
is exact. In fact, since the maps here are straight, applying $\cone^{\circ n}$ to the sequence results in
$$0\to\bigoplus_{v}\cB^v[-1] \to \bigoplus_v (\cA^v\oplus \cB^v) \to \bigoplus_v \cA^v\to 0\,,$$
where the horizontal maps are the direct sums of the components of $\iota[-1]$ and $\pi$, and this sequence is clearly exact, whence the claim. In particular if both $\cA,\cB$ are acyclic, then so is $\co_{n+1}\cF$.

\subsection{Folding and pullbacks of V-shaped diagrams}\label{ss:folding_pullbacks}

When defining the product on the relative symplectic cohomology of a pair in Section \ref{ss:constructing_product}, we'll need to use the technique of folding a cube and then taking its cocone. Here we describe this technique.

Consider an $(n+2)$-cube, $n\geq 0$, of the form
\begin{equation}\label{eqn:diagram_of_cubes_folding}
  \xymatrix{\cA \ar[r]^{\cF} \ar[d]_{\cG} \ar[rd]^{\cH} & \cB \ar[d]^{\cI} \\ \cC \ar[r]_{\cK} & \cD}
\end{equation}

Here $\cA,\dots,\cD$ are $n$-cubes, and $\cF,\cG,\cI,\cK$ are maps of $n$-cubes, that is $(n+1)$-cubes where the last direction is the one marked with the corresponding letter; $\cH$ stands for all the maps running from vertices of $\cA$ to those of $\cD$. The horizontal and the vertical directions have numbers $n+1,n+2$, respectively. We define the corresponding \tb{folded cube} to be as follows:
$$\xymatrix{\cA \ar[r]^-{(\cF, -\cG)} \ar[d] \ar[rd]^{\cH} & \cB\oplus \cC \ar[d]^{\cI+\cK} \\ 0 \ar[r] & \cD}$$
Here $\cA \xrightarrow{(\cF, -\cG)}\cB\oplus \cC$ is the composition of the diagonal map $\cA \to \cA \oplus \cA$ and the direct sum $\cF\oplus(-\cG) \fc \cA \oplus \cA \to \cB \oplus \cC$, while $\cI+\cK$ is the composition of the direct sum $\cI \oplus \cK\fc \cB \oplus \cC \to \cD\oplus \cD$ with the sum map $\cD \oplus \cD \to \cD$. It is a matter of routine verification that this is indeed a cube. Note that folding commutes with passage to subcubes in the obvious sense.

A \tb{V-shaped diagram} is a partial cube of the form
$$\xymatrix{ & \cB\ar[d]^{\cI} \\ \cC\ar[r]_{\cK} & \cD}$$
We define its \tb{pullback} to be the $n$-cube $\cQ = \co_{n+1}(\cB\oplus \cC\xrightarrow{\cI+\cK}\cD)$.

If we have a cube of the form \eqref{eqn:diagram_of_cubes_folding}, we can fold it and then apply cocone in the last ($(n+2)$-nd) direction. We then obtain an $(n+1)$-cube of the form
$$\co_{n+1}(\cA\to 0)\to\cQ = \co_{n+1}(\cB\oplus \cC \xrightarrow{\cI + \cK}\cD)\,,$$
which is the point of the current section.\footnote{Note that even though the vertex modules of $\co_{n+1}(\cA \to 0)$ are all canonically isomorphic to those of $\cA$, the structure maps gain signs which depend on the dimension of the corresponding face (in fact if $F\subset [0,1]^n$ is a face, $f^{\co_{n+1}(\cA\to 0)}_F = (-1)^{|F|}f^\cA_F$).}

\subsection{Rays and telescopes}\label{ss:rays_telescopes}

For $n\geq 0$, an \tb{$(n+1)$-ray} is a diagram of the form
$$\cR = \cA_1\xrightarrow{\cF_1}\cA_2\xrightarrow{\cF_2}\cA_3\to\dots$$
consisting of $n$-cubes $\cA_1,\cA_2,\dots$ and maps between them $\cF_1,\cF_2,\dots$ If $F \subset [0,1]^n$ is a face, such a ray defines a restriction to $F$, which is the $(|F|+1)$-ray
$$\cR|_F = \cA_1|_F\xrightarrow{\cF_1|_{\wh F}}\cA_2|_F\to\dots$$

Here we will define the \tb{telescope} $\tel \cR$ of such a diagram, which is an $n$-cube. It will have the property that for any face $F \subset [0,1]^n$ we have
\begin{equation}\label{eqn:telescope_commutes_subcubes}
  \tel (\cR|_F) = (\tel \cR)|_F\,,
\end{equation}
which is crucial in the applications of telescopes below.
\begin{rem}
  To motivate the definition of telescopes, recall that given modules $A_i$, $i \in \N$, and module maps $f_i \fc A_i \to A_{i+1}$, the corresponding direct limit $\varinjlim_i A_i$ can be taken as the cokernel of the map
  $$\id - f \fc \bigoplus_i A_i \to \bigoplus_i A_i\,,\quad \text{where}\quad f(a_1,a_2,\dots,)=(0,f_1(a_1),\dots)\,.$$
  In this paper we are working with homotopical constructions, therefore we need the analog of the direct limit in homotopy theory, also known as the \emph{homotopy colimit}. The telescope of a $1$-ray is a model for it. Since the above direct limit is the cokernel of a map, it is expected that the corresponding homotopy colimit is given by the homotopy cokernel of a map, or in other words, by its cone.
\end{rem}

Consider the map of $n$-cubes
$$\bigoplus_{i=1}^\infty\cA_i \xrightarrow{\cF}\bigoplus_{i=1}^\infty\cA_i\,,$$
symbolically defined as $\cF=\id-\bigoplus_i\cF_i$, where $\bigoplus_i\cF_i$ is explicitly given as follows. First, both the domain and the target cubes have their own structure maps given by the direct sums of the structure maps of the $\cA_i$. Now, given a vertex $v \in \{0,1\}^n$, both the domain and target have the corresponding vertex module
$$\left(\bigoplus_{i=1}^\infty\cA_i\right)^v = \bigoplus_{i=1}^\infty\cA_i^v\,.$$
Given vertices $v,v' \in \{0,1\}^n$ spanning a face $F$ with $v = \ini F$, $v'=\ter F$, the corresponding face map
$$f^{\bigoplus_i\cF_i}_{\wh F} \fc \bigoplus_{i=1}^\infty\cA_i^v \to \bigoplus_{i=1}^\infty\cA_i^{v'}$$ has matrix representation
$$\left(\begin{array}{rrrrr}0 & 0 & 0 & 0 & \dots \\ f^{\cF_1}_{\wh F} & 0 & 0 & 0 & \dots \\ 0 & f^{\cF_2}_{\wh F} & 0 & 0 & \dots \\ &&\dots&&\end{array}\right)\,.$$

We can now define the telescope of $\cR$:
$$\tel \cR := \cone_{n+1} \cF\,.$$

It can be shown that the telescope of a subray is the corresponding subcube of the telescope, as claimed in equation \eqref{eqn:telescope_commutes_subcubes}. Another feature is that if we have an $(n+2)$-ray consisting of the identity maps between $n$-cubes as follows:
$$\xymatrix{\cA_1 \ar[r]^{\cF_1} \ar[d]^{\id} & \cA_2 \ar[r]^{\cF_2} \ar[d]^{\id} & \cA_3 \ar[r] \ar[d]^{\id} & \dots \\ \cA_1 \ar[r]^{\cF_1}& \cA_2 \ar[r]^{\cF_2}& \cA_3 \ar[r] & \dots }$$
then its telescope is
$$\tel \cR \xrightarrow{\id} \tel \cR\,,$$
where $\cR = \cA_1\to \cA_2\to\dots$

Another crucial property of telescopes is their behavior relative to tensor products. If $\cR = A_1\xrightarrow{f_1} A_2\to\dots$ and $\cR' = A_1'\xrightarrow{f_1'} A_2' \to \dots$ are $1$-rays, their tensor product is defined to be the $1$-ray $\cR \otimes \cR' = A_1\otimes A_1' \xrightarrow{f_1\otimes f_1'} A_2 \otimes A_2' \to \dots$.
There is a canonical quasi-isomorphism $\tel (\cR\otimes \cR') \to \tel \cR \otimes \tel\cR'$, see
\cite{varolgunes2018mayer} and \cite[Lemma 0.1]{greenlees1992derived}. Moreover, the induced map on completions, $\wh{\tel (\cR \otimes \cR')} \to \wh{\tel \cR \otimes \tel\cR'}$, is also a quasi-isomorphism, provided all the modules in sight are free \cite[Corollary 2.3.6]{varolgunes2018mayer}. Assume now that there are two $2$-rays $\cT,\cT'$:
$$\xymatrix{A_1 \ar[r] \ar[d] \ar[rd] & A_2 \ar[r] \ar[d] & \dots \\ B_1 \ar[r] & B_2 \ar[r] & \dots} \qquad \qquad \xymatrix{A_1' \ar[r] \ar[d] \ar[rd] & A_2' \ar[r] \ar[d] & \dots \\ B_1' \ar[r] & B_2' \ar[r] & \dots}$$
and let $\cA,\cB,\cA'\,\cB'$ be the $1$-rays comprised of the $A_i$, $B_i$, $A_i'$, $B_i'$, respectively. We obtain a natural $3$-ray whose constituent $1$-rays are the tensor products $\cA\otimes \cA'$, $\cA\otimes \cB'$, $\cB\otimes\cA'$, and $\cB\otimes \cB'$. Taking the telescope, we obtain the square
\begin{equation}\label{eqn:telescope_of_tensor_product}
{\xymatrix{\tel(\cA\otimes \cA')\ar[r] \ar[d] \ar[rd] & \tel(\cA\otimes \cB') \ar[d] \\ \tel(\cB\otimes\cA') \ar[r] & \tel(\cB\otimes \cB')}}
\end{equation}
On the other hand, we have the tensor product of the cochain maps $\tel \cT = \tel\cA \to \tel \cB$, $\tel \cT' = \tel\cA' \to \tel \cB'$, that is the square
\begin{equation}\label{eqn:tensor_products_telescopes}
\xymatrix{\tel\cA\otimes \tel\cA'\ar[r] \ar[d] \ar[rd]& \tel\cA\otimes \tel\cB' \ar[d] \\ \tel\cB\otimes\tel\cA' \ar[r] & \tel\cB\otimes \tel\cB'}
\end{equation}
The point is that the above quasi-isomorphism extends to a map of cubes from the square \eqref{eqn:telescope_of_tensor_product} to the square \eqref{eqn:tensor_products_telescopes}, in the sense that each edge map going in the direction of the cube map is a quasi-isomorphism.

\subsection{Floer data, rays, and symplectic cohomology of pairs}\label{ss:Floer_data_rays}

Here we discuss the notions of Floer data, acceleration data, the resulting rays of Floer complexes, and show how these are used to define the symplectic cohomology of a pair. Also we recall Varolg\"une\c s's notion of descent. Throughout this section $(M,\omega)$ is a fixed closed symplectic manifold.

\subsubsection{Hamiltonian rays}

In \cite{varolgunes2018mayer}, Varolg\"une\c s defined monotone cubes, triangles, and slits of Hamiltonians. Let us recall what this means. Consider a strictly increasing Morse function $\rho_1 \fc [0,1]\to\R$ with exactly two critical points at $0,1$. It gives rise to the Morse function $\rho_n \fc [0,1]^n \to \R$ by $\rho_n(x)=\rho_1(x_1)+\dots+\rho_1(x_n)$. We also endow the cube with the standard Riemannian metric. A \tb{monotone cube} of Hamiltonians is a suitably smooth map $[0,1]^n\to C^\infty(M)$, which is monotone nondecreasing along the gradient lines of $\rho_n$; Varolg\"une\c s's definition in particular assumes that \emph{the Hamiltonians at the vertices of such a shape are nondegenerate}, and that they are also constant on a neighborhood of each vertex. We will impose these assumptions throughout. Similarly monotone triangles and slits are defined as smooth maps into $C^\infty(M)$ defined on subsets of $[0,1]^n$ called triangles and slits, see \emph{ibid.}
\begin{defin}
  An \tb{$n$-ray of Hamiltonians} or a \tb{Hamiltonian $n$-ray} is a sequence $\cH_i$, $i\in\N$, of monotone $n$-cubes of Hamiltonians, such that the face of $\cH_i$ corresponding to $\{x_n=1\}$ coincides with the face of $\cH_{i+1}$ corresponding to $\{x_n=0\}$. Similarly we define \tb{triangular $n$-rays} and \tb{slit-like $n$-rays} of Hamiltonians: a triangular $n$-ray of Hamiltonians, defined for $n\geq 3$, is a sequence of monotone $n$-triangles of Hamiltonians which agree along faces as above. A slit-like $n$-ray of Hamiltonians, defined for $n\geq 3$, is likewise a sequence of monotone $n$-slits of Hamiltonians which agree along the appropriate faces.
\end{defin}

\begin{rem}
  Such configurations will give rise to rays, triangular and slit-like rays in the algebraic sense, as we will describe below. To associate such an algebraic object to a configuration of Hamiltonians, additional structure is needed, such as choices of almost complex structures and Pardon data, depending on the level of generality. A suitable choice of such structures always exists given a configuration of Hamiltonians, therefore we will suppress such choices both from notation and from the discussion below.
\end{rem}

Varolg\"une\c s proves the following result, itself a consequence of Pardon's constructions \cite{pardon2016algebraic,pardon2019contact}:
\begin{thm}
  Given a monotone $n$-cube $\cH$ of Hamiltonians and a suitable choice of additional data such as almost complex structures or Pardon data, there is an $n$-cube whose vertices are the Floer complexes of the Hamiltonians at the vertices of $\cH$, and whose higher maps are obtained by counting elements of suitable moduli spaces of parametrized Floer equations. Similarly, a monotone $n$-triangle gives rise to an $n$-triangle of Floer complexes, and a monotone $n$-slit yields an $n$-slit. \qed
\end{thm}

Since an $n$-ray of Hamiltonians consists of monotone $n$-cubes of Hamiltonians glued along the $n$-th direction, we have the following
\begin{coroll}
  An $n$-ray of Hamiltonians $\cH$ gives rise to an $n$-ray whose vertices are the Floer complexes of the Hamiltonians at the vertices of $\cH$. Similarly, a triangular or a slit-like $n$-ray of Hamiltonians gives rise to a triangular or a slit-like $n$-ray of Floer complexes, respectively. \qed
\end{coroll}
\noindent See Remark \ref{rem:cocones_turn_ids_minus_ids} and \cite{varolgunes2018mayer} for the definition of triangular and slit-like cubes. This corollary is applied as follows: we take the telescope of such an $n$-ray to obtain an $(n-1)$-cube, to which we then apply the completion functor.

\subsubsection{Weighted Floer complexes}\label{sss:weighted_HF}

Here we specify the Floer complexes we will be using. If $H \in C^\infty(M\times S^1)$ is a non-degenerate Hamiltonian, its Floer complex, generated by the set $\cP^\circ(H)$ of its contractible $1$-periodic orbits, was defined as $CF^*(H) = \bigoplus_{x\in\cP^\circ(H)}\Lambda_{\geq0}\cdot x$ in equation \eqref{eqn:defin_Floer_cx}. This complex is graded over $\Z_{2N_M}$ by the Conley--Zehnder index. The differential is determined by its matrix elements, given by
$$\langle dx,y\rangle = \sum_{A \in \pi_2(M,x,y)}\#\cM(H;x,y;A)\,T^{E(A)}\,,$$
where $\pi_2(M,x,y)$ is the set of homotopy classes of smooth maps of the cylinder $\R\times S^1$ to $M$ which are asymptotic at $\pm\infty$ to $x,y$, $\cM(H;x,y;A)$ stands for the moduli space of unparametrized Floer trajectories corresponding to $H$, running from $x$ to $y$, and representing the class $A$, such that every solution has index $1$; $\#\cM(H;x,y;A)$ is a suitable virtual count as in Pardon \cite{pardon2016algebraic,pardon2019contact} or just a signed count if $M$ is assumed to be semipositive. Finally $E(A) = \int_{S^1}\big(H_t(y(t))-H_t(x(t))\big)\,dt-\langle \omega,A\rangle$ is the topological energy of solutions in class $A$, or equivalently the increase in the action along such solutions. Continuation maps between such complexes corresponding to monotone nondecreasing homotopies of Hamiltonians are defined in a similar manner, with weighting by suitable powers of $T$.

\begin{rem}
  The name \emph{weighted Floer complex} comes from the inclusion of the $T$-weights in the differential.
\end{rem}

\subsubsection{Acceleration data and descent}\label{sss:accel_data_descent}

We need the following definitions.
\begin{defin}
\begin{itemize}
  \item An \tb{acceleration datum} is a $1$-ray of Hamiltonians. Given an acceleration datum $\cH$, we will write $\cH = (H_i)_{i=1}^\infty$, meaning that $H_i$ is the nondegenerate Hamiltonian at the $i$-th vertex of the ray, with the monotone $1$-cubes of Hamiltonians between them being implicit.
  \item Given two acceleration data $\cH = (H_i)_i$ and $\cH' = (H_i')_i$, we write $\cH \preceq \cH'$ if $H_i \leq H_i'$ for all $i$. Note that we do not require anything regarding the monotone $1$-cubes between the Hamiltonians at the vertices.
  \item If $\cH,\cH'$ are two acceleration data and $\cH\preceq \cH'$, a \tb{filling} $\cF \fc \cH \to \cH'$ is a $2$-ray of Hamiltonians whose top and bottom $1$-rays are $\cH,\cH'$.
\end{itemize}
\end{defin}

\begin{rem}\label{rem:exist_fillings}
  The existence of fillings follows from the results of \cite[Section 3.2.2]{varolgunes2018mayer}. More generally, whenever we have a partially defined monotone cube of Hamiltonians which is Floer-theoretic in the sense of \cite[Section 3.2.3]{varolgunes2018mayer}, which means that the given Hamiltonians do not decrease along broken gradient flow lines of our Morse function $\rho_n$, it can be filled to a monotone cube of Hamiltonians. The same is true of other kinds of Floer-theoretic configurations of Hamiltonians---slits, triangles, rays, and so on. We will use this existence in several places in what follows.
\end{rem}

\begin{rem}
  For the rest of Section \ref{s:proofs_properties_SH_pairs} we drop the superscript indicating the grading from all Floer complexes and complexes derived therefrom.
\end{rem}

\begin{notation}
  Given an acceleration datum $\cH = (H_i)_i$ we denote the corresponding Floer $1$-ray by $CF(\cH) = CF(H_1) \to CF(H_2) \to \dots$
\end{notation}

\begin{defin}\label{def:accel_datum_K}
  Let $K\subset M$ be compact. We say that an acceleration datum $\cH=(H_i)_i$ is an \tb{acceleration datum for $K$} if $(H_i)_i$ is a cofinal sequence in $C^\infty_{K\subset M} = \{H \in C^\infty(M\times S^1)\,|\,H|_{K\times S^1}<0\}$ relative to the usual order on functions.\footnote{This is a variant of the notation used in \cite{varolgunes2018mayer}.}
\end{defin}

\begin{defin}(Varolg\"une\c s \cite{varolgunes2018mayer})\label{def:rel_SH_subset}
  Let $K\subset M$ be compact and let $\cH = (H_i)_i$ be an acceleration datum for $K$. The corresponding complex is defined to be
  $$SC(\cH):=\wh\tel \,CF(\cH)\,,$$
  where the completion is done degree-wise, in accordance with our conventions, see Section \ref{ss:completion}. \tb{The relative symplectic cohomology of $K$ inside $M$} is
  $$SH(K):=H\big(SC(\cH)\big)\,.$$
  If $R$ is a commutative $\Lambda_{\geq0}$-algebra, \tb{the symplectic cohomology of $K$ with coefficients in $R$} is defined to be $SH(K;R):=H\big(SC(\cH)\otimes R\big)$.
\end{defin}
\begin{rem}
  Varolg\"une\c s proves in \cite{varolgunes2018mayer} that this is well-defined, in the sense that given another acceleration datum $\cH'$ for $K$, the two cohomology modules are canonically isomorphic. We will provide details of this proof in Section \ref{ss:SH_pair_well_defined} below. The proofs are only given for $R=\Lambda_{\geq0}$, but they work \emph{verbatim} for any $R$.
\end{rem}

Let us now recall what it means for two sets to be in descent; see \cite{varolgunes2018mayer}. Let $K,K'\subset M$ be compact subsets. Choose an acceleration datum $\cH^\bullet$ for $\bullet=K,K',K\cup K', K\cap K'$, such that $\cH^\bullet \preceq \cH^{\bullet'}$ whenever $\bullet \supset \bullet'$. These acceleration data can be fitted into a Hamiltonian $3$-ray by choosing suitable fillings for the corresponding Hamiltonian cubes, as proved in \cite{varolgunes2018mayer}, see also Remark \ref{rem:exist_fillings}. Passing to the corresponding $3$-ray of Floer complexes, taking its telescope and completing yields the following square:
$${\xymatrix{SC(K \cup K') \ar[r]\ar[d] &  SC(K')\ar[d] \\ SC(K)\ar[r] & SC(K'\cap K')}}$$
where $SC(\bullet)$ stands for $SC(\cH^\bullet)$. Denoting this square by $\cC$, we say that $K,K'$ are \tb{in descent} if $\cC$ acyclic, meaning that its repeated cone $\cone^{\circ 2}\cC$ is an acyclic complex.

\subsubsection{The relative symplectic cohomology of a pair}\label{sss:rel_SH_pairs}

Next we define one of the main characters in our story, the \emph{relative symplectic cohomology of a pair}.
\begin{defin}\label{def:rel_complex_pair}
  Let $(K,K')$ be a compact pair in $M$. Let $\cH,\cH'$ be acceleration data for $K,K'$, respectively, assume that $\cH \preceq \cH'$, and fix a filling $\cF \fc \cH \to \cH'$, which exists by Remark \ref{rem:exist_fillings}. This filling defines a $2$-ray of Floer complexes whose top and bottom $1$-rays are $CF(\cH)$ and $CF(\cH')$, respectively. Taking its completed telescope, we arrive at a cochain map $\Phi_\cF \fc SC(\cH) \to SC(\cH')$. We define the relative complex corresponding to $\cF$ to be
  $$SC(\cF):=\co \Phi_\cF\,,$$
  and the \tb{relative symplectic cohomology of the pair} $(K,K')$ as its cohomology:
  $$SH(K,K'):=H\big(SC(\cF)\big)\,.$$
\end{defin}

Theorem \ref{thm:existence_relative_SH_pairs} is an immediate consequence of the following result:
\begin{thm}\label{thm:SH_pair_well_defd_satisfies_axioms}
  The relative symplectic cohomology of a pair is well-defined independently of the chosen data. Moreover, it is a relative symplectic cohomology of pairs with coefficients in $\Lambda_{\geq 0}$.
\end{thm}
\noindent The proof of this theorem occupies Section \ref{ss:SH_pair_well_defined}, where we prove that the symplectic cohomology of a pair is well-defined and its properties are proved, with the exception of the Product axiom, to which Section \ref{ss:constructing_product} is dedicated.

\subsection{Well-definedness and properties of $SH(K,K')$}\label{ss:SH_pair_well_defined}

Here we prove that $SH^*(\cdot,\cdot)$ is well-defined, that is it is independent of the various choices of acceleration data, fillings, and so on, and then we prove that it satisfies the properties announced in Section \ref{ss:axioms_rel_SH_pairs}, thereby proving Theorems \ref{thm:SH_pair_well_defd_satisfies_axioms} and \ref{thm:existence_relative_SH_pairs}.

In \cite{varolgunes2018mayer} it was proved that $SH(K)$ is well-defined. Here we elaborate on that argument, and the goal is to construct a framework for the analogous proof for $SH(K,K')$.
\begin{defin}
\begin{itemize}
  \item If $\cH = (H_i)_i$ is an acceleration datum, we say that an acceleration datum $\wt\cH = (\wt H_j)_j$ is a \tb{subdatum} of $\cH$ if there is a strictly increasing sequence of natural numbers $i(j)$ such that $\wt H_j = H_{i(j)}$; note that we do not require anything of the $1$-cubes between them. We write $\wt \cH \subset \cH$. Note that $\cH \preceq \wt\cH$.
  \item Let us call acceleration data $\cH,\cH'$ \tb{equivalent} if there is a subdatum $\wt\cH \subset \cH$ such that $\cH' \preceq \wt\cH$.
\end{itemize}
\end{defin}
We let
  $$\cS_K = \{\cH = (H_i)_i\,|\, \cH\text{ is an acceleration datum for }K\}\,.$$
The collection of all acceleration data is partially ordered by $\preceq$, and $\cS_K$ is a directed subset. Moreover, any two acceleration data for $K$ are equivalent.

Below is a reformulation of Varolg\"une\c s's construction of the relative symplectic cohomology.
\begin{enumerate}
  \item Given an acceleration datum $\cH$, in Section \ref{ss:Floer_data_rays} we have defined the corresponding $1$-ray of Floer complexes $CF(\cH)$, and the corresponding complex $SC(\cH)$ and homology $SH(\cH) = H(SC(\cH))$.
  \item Given another acceleration datum $\cH'$ such that $\cH\preceq \cH'$, and a filling $\cF \fc \cH \to \cH'$, there is a $2$-ray of Floer complexes such that $CF(\cH)$ and $CF(\cH')$ are its top and bottom $1$-rays. Taking its telescope and completing, we arrive at a $1$-cube of the form $\Phi_\cF \fc SC(\cH) \to SC(\cH')$, which is simply a chain map between the complexes of $\cH,\cH'$. We use the same notation for the induced map on cohomology: $\Phi_\cF \fc SH(\cH) \to SH(\cH')$.
  \item If $\wt \cF \fc \cH \to \cH'$ is another filling, $\cF$ and $\wt\cF$ fit into a slit-like $3$-ray of Floer data like this:
      $$\xymatrix{\cH \ar@{->}@/^0.5pc/[r]^{\cF} \ar@{->}@/_0.5pc/[r]_{\wt\cF} \ar@{}|{\cG}[r] & \cH'}\,,$$
      where the third dimension is perpendicular to the page, and where $\cG$ is a suitable filling of the partially defined slit, which exists thanks to Remark \ref{rem:exist_fillings}. This yields a slit-like $3$-ray of Floer complexes, whose completed telescope is a $2$-slit of the form
      $$\xymatrix{SC(\cH) \ar@/^/[r]^{\Phi_\cF}\ar@/_/[r]_{\Phi_{\wt\cF}}& SC(\cH')}\,,$$
      in other words, we obtain a homotopy between $\Phi_\cF,\Phi_{\wt\cF}$, and therefore the two induce the same morphism on homology. It follows that the maps induced on homology by various fillings all coincide, and we denote the resulting map by $\Phi_{\cH}^{\cH'} \fc SH(\cH) \to SH(\cH')$.
  \item If $\cH\preceq \cH' \preceq \cH''$ are acceleration data and $\cF_0 \fc \cH \to \cH'$, $\cF_1 \fc \cH' \to \cH''$, $\cF_2 \fc \cH \to \cH''$ are fillings, they fit into a triangular $3$-ray of Hamiltonians, where the existence of the filling is thanks to Remark \ref{rem:exist_fillings}. To it there corresponds a triangular $3$-ray of Floer complexes, and its completed telescope is a $2$-triangle
      $$\xymatrix{SC(\cH) \ar[r]^{\Phi_{\cF_0}} \ar[rd]_{\Phi_{\cF_2}}& SC(\cH') \ar[d]^{\Phi_{\cF_1}} \\ & SC(\cH'')}$$
      which yields the relation $\Phi_{\cH'}^{\cH''} \circ \Phi_{\cH}^{\cH'} = \Phi_{\cH}^{\cH''}$ on homology.
  \item If $\wt\cH \subset \cH$ is a subdatum, Varolg\"une\c s proves in \cite{varolgunes2018mayer} that $\Phi_\cH^{\wt\cH}$ is an isomorphism. In particular, since $\Phi_\cH^\cH\circ\Phi_\cH^\cH=\Phi_\cH^\cH$, we see that $\Phi_\cH^\cH$ is the identity. All of the above means that $(SH(\cH),\Phi_\cH^{\cH'})$ is a direct system.
  \item If $\cH\preceq\cH'$ are equivalent, then there are subdata $\wt\cH \subset \cH$, $\wt \cH' \subset \cH'$ such that $\cH \preceq \cH'\preceq \wt \cH \preceq \wt\cH'$. Considering the resulting commutative diagram
      $$\xymatrix{SH(\cH) \ar[r]_{\Phi_\cH^{\cH'}} \ar@/^1pc/[rrr]^{\Phi_\cH^{\wt\cH}} & SH(\cH') \ar[rr]_>>>{\Phi_{\cH'}^{\wt\cH}} \ar@/_1.5pc/[rrr]_{\Phi_{\cH'}^{\wt\cH'}} && SH(\wt\cH)\ar[r]^{\Phi_{\wt\cH}^{\wt\cH}} & SH(\wt\cH')}$$
      and the fact that $\Phi_{\cH}^{\wt\cH}$ and $\Phi_{\cH'}^{\wt\cH'}$ are isomorphisms, we see that $\Phi_\cH^{\cH'}$ is likewise an isomorphism.
  \item If $K\subset M$ is a compact subset, we define
      $$SH(K) := \varinjlim_{\cH \in \cS_K}SH(\cH)\,,$$
      where the connecting maps are the above morphisms. Since all the connecting maps are isomorphisms, every natural morphism $SH(\cH) \to SH(K)$ is in fact an isomorphism.
  \item If $L\subset K$ is another compact set, the restriction morphism
      $$\res^K_L \fc SH(K) \to SH(L)$$
      is defined as follows. Let $\cH \in \cS_K$. There is $\cH' \in \cS_L$ with $\cH \preceq \cH'$. Consider the composition $SH(\cH) \xrightarrow{\Phi_\cH^{\cH'}} SH(\cH') \to SH(L)$, where the second map is the natural morphism into the direct limit. Thus we get a map $SH(\cH) \to SH(L)$. Using the cocycle identities for the comparison maps $\Phi$ as above, as well as properties of direct limits, we can show that this map is independent of the choice of $\cH'$. Moreover if $\cH_1\in\cS_K$ is such that $\cH\preceq\cH_1$, then the map $SH(\cH)\to SH(L)$ equals the composition $SH(\cH)\xrightarrow{\Phi_\cH^{\cH_1}}SH(\cH_1) \to SH(L)$. It follows from the universal property of colimits that we have described a map $\res^K_L \fc SH(K) \to SH(L)$. Moreover, since natural maps $SH(\cH) \to SH(K)$ and $SH(\cH') \to SH(L)$ for $\cH \in \cS_K,\cH'\in\cS_L$ are isomorphisms, the diagram
      $$\xymatrix{SH(\cH) \ar[r] \ar[d] & SH(K) \ar[d] \\ SH(\cH') \ar[r] & SH(L)}$$
      commutes, a fact which is convenient when we actually have to compute the restriction. It is also clear that $\res^K_K = \id$ and that restrictions satisfy $\res^L_P\circ\res^K_L = \res^K_P$ if $P\subset L\subset K$.
\end{enumerate}

We will now prove that $SH(K,K')$ is well-defined, define the restriction morphisms, and prove the properties formulated in Section \ref{ss:axioms_rel_SH_pairs}. Items (i-viii) contain the proof of well-definedness, restrictions are the subject of item (ix), while items (x-xii) contain the proofs of the normalization, triangle, and the Mayer--Vietoris properties.

The main obstacle to overcome is to prove that this cohomology is independent of the choice of acceleration data and the filling. This is done using the same scheme as in the above proof that $SH(K)$ is well-defined independently of the acceleration datum used, except that we have to add a dimension throughout, and that now we also have to invoke properties of cocones.

\begin{enumerate}
  \item If $\cF \fc \cH \to \cH'$ is a filling between acceleration data, we let $SC(\cF):=\co \Phi_\cF$. Note that we have an exact sequence
      $$0 \to SC(\cH')[-1] \to SC(\cF) \to SC(\cH) \to 0\,,$$
      see equation \eqref{eqn:exact_sequence_cocone}. We let $SH(\cF)=H(SC(\cF))$.
  \item If $\cH_i,\cH_i'$, $i=0,1$ are acceleration data such that $\cH_i\preceq \cH_i'$ for $i=0,1$, $\cH_0\preceq \cH_1$ and $\cH_0'\preceq \cH_1'$, consider fillings $\cF_i \fc \cH_i \to \cH_i'$. These fit into a $3$-ray $\cG$ of Hamiltonians, thanks to Remark \ref{rem:exist_fillings}, which then produces a $3$-ray of Floer complexes. Taking its telescope and completing, we arrive at a $2$-cube, where the vertical direction is the first one and the horizontal one is the second one:
      $$\xymatrix{SC(\cH_0) \ar[r]^{\Phi_{\cF_0}} \ar[d] & SC(\cH_0')\ar[d] \\ SC(\cH_1)\ar[r]^{\Phi_{\cF_1}} & SC(\cH_1')}$$
      Taking the cocone in the horizontal direction, we obtain a chain map $SC(\cF_0) \to SC(\cF_1)$. We let $B_\cG$ be the negative of this map. We use the same notation for the induced map on cohomology. Note that we have the following commutative diagram with rows being exact sequences:
      \begin{equation}\label{eqn:B_G}
        \xymatrix{0 \ar[r] & SC(\cH_0')[-1] \ar[r] \ar[d] & SC(\cF_0)\ar[r] \ar[d]^{-B_\cG} & SC(\cH_0) \ar[r] \ar[d] & 0 \\ 0 \ar[r] & SC(\cH_1')[-1] \ar[r] & SC(\cF_1) \ar[r] & SC(\cH_1)\ar[r] & 0}
      \end{equation}

      which is a particular case of equation \eqref{eqn:exact_sequence_cocone}.
  \item If in the previous situation we have another $3$-ray $\cG'$ extending the given fillings $\cF_i$, the two $3$-rays fit into a slit-like $4$-ray of Hamiltonians, and taking its telescope and completion, we obtain a $3$-slit as follows:
      $$\xymatrix{SC(\cH_0) \ar[r]^{\Phi_{\cF_0}} \ar@/^/[d]\ar@/_/[d] & SC(\cH_0')\ar@/^/[d]\ar@/_/[d] \\ SC(\cH_1)\ar[r]^{\Phi_{\cF_1}} & SC(\cH_1')}$$
      Taking cocone in the horizontal direction yields a $2$-slit giving a homotopy between $B_\cG,B_{\cG'}$ (see Remark \ref{rem:cocones_turn_ids_minus_ids}). We denote the resulting well-defined map on homology by $B_{\cF_0}^{\cF^1} \fc SH(\cF_0) \to SH(\cF_1)$.
  \item If we have three pairs of acceleration data $\cH_{ij}$, $i=0,1$, $j=0,1,2$, such that $\cH_{0j}\preceq \cH_{1j}$ and $\cH_{i0}\preceq\cH_{i1}\preceq\cH_{i2}$, and we have fillings $\cF_j \fc \cH_{0j} \to \cH_{1j}$, then all of this fits into a triangular $4$-ray of Hamiltonians, thanks to Remark \ref{rem:exist_fillings}, whose completed telescope is the following $3$-triangle:
      $$\xymatrix{SC(\cH_{00})\ar[rr]^{\Phi_{\cF_0}} \ar[rd] \ar[rdd] & & SC(\cH_{10})\ar[rd] \ar[rdd]|\hole \\ & SC(\cH_{01}) \ar[rr]^{\Phi_{\cF_1}} \ar[d] &&SC(\cH_{11})\ar[d] \\  & SC(\cH_{02})\ar[rr]^{\Phi_{\cF_2}} &&SC(\cH_{12})}$$
      Taking cocone in the horizontal direction yields the homotopy-commutative triangle (see Remark \ref{rem:cocones_turn_ids_minus_ids}):
      $$\xymatrix{SC(\cF_0) \ar[r]^{B_{\cF_0}^{\cF_1}} \ar[rd]_{B_{\cF_0}^{\cF_2}} & SC(\cF_1) \ar[d]^{B_{\cF_1}^{\cF_2}} \\ & SC(\cF_2)}$$
      It follows that on homology we have $B_{\cF_1}^{\cF_2}\circ B_{\cF_0}^{\cF_1} = B_{\cF_0}^{\cF_2}$.
  \item If $\cF \fc \cH \to \cH'$ is a filling and $\wt\cF \fc \wt\cH \to \wt \cH'$ is a filling between subdata, then using any $3$-ray $\cG$ of Hamiltonians extending $\cF,\wt\cF$, we arrive at the following diagram, which is a particular case of \eqref{eqn:B_G}:
      $$\xymatrix{0 \ar[r] & SC(\cH')[-1] \ar[r] \ar[d] & SC(\cF)\ar[r] \ar[d]^{-B_\cG} & SC(\cH) \ar[r] \ar[d] & 0 \\ 0 \ar[r] & SC(\wt\cH')[-1] \ar[r] & SC(\wt\cF) \ar[r] & SC(\wt\cH)\ar[r] & 0}$$
      Since the right and the left vertical arrows are quasi-isomorphisms, so is the middle one. Therefore $B_\cF^{\wt\cF} \fc SH(\cF) \to SH(\wt\cF)$ is an isomorphism. In particular $B_\cF^\cF$ is the identity map.

  \item Given a pair of acceleration data $\cH,\cH'$ with $\cH \preceq \cH'$, consider the set of fillings $\{\cF \fc \cH \to \cH'\}$. It parametrizes the system $(SH(\cF),B_{\cF}^{\cF'})$ of modules and isomorphisms satisfying the cocycle identity. It follows that we can take its colimit:
      $$SH(\cH,\cH'):=\varinjlim_{\cF\fc \cH\to\cH'}SH(\cF)\,,$$
      and that for any $\cF$ the natural map $SH(\cF) \to SH(\cH,\cH')$ is an isomorphism.
  \item If $(\cH_{ij})_{i,j=0,1}$ are acceleration data with $\cH_{0j} \preceq \cH_{1j}$ and $\cH_{i0}\preceq \cH_{i1}$, then the above discussion yields a natural map $SH(\cH_{00},\cH_{01}) \to SH(\cH_{10},\cH_{11})$. If we have a third such pair, then the corresponding maps obey the cocycle identity. Moreover the natural map $SH(\cH,\cH') \to SH(\cH,\cH')$ is the identity.
  \item If $K' \subset K \subset M$ are compact sets, put
      $$\cS_{K,K'}=\{(\cH,\cH') \in \cS_K \times \cS_{K'}\,|\, \cH \preceq \cH'\}\,.$$
      Abusing notation, let us denote by $\preceq$ the order on this set induced by the product order on $\cS_K\times\cS_{K'}$. This turns $\cS_{K,K'}$ into a directed set. Put
      $$SH(K,K'):=\varinjlim_{(\cH,\cH') \in \cS_{K,K'}}SH(\cH,\cH')\,.$$
      Using reasoning as above, we see that for any $(\cH,\cH') \in \cS_{K,K'}$ the natural morphism $SH(\cH,\cH') \to SH(K,K')$ is an isomorphism.
  \item If $(L,L') \subset (K,K')$ is a compact subpair, we can define the corresponding restriction morphism
      $$\res^{(K,K')}_{(L,L')} \fc SH(K,K') \to SH(L,L')$$
      similarly to the restriction map for the absolute case. Namely, we pick $(\cH,\cH') \in \cS_{K,K'}$ and $(\cG,\cG') \in \cS_{L,L'}$ with $(\cH,\cH')\preceq(\cG,\cG')$. We have the composition
      $$SH(\cH,\cH') \to SH(\cG,\cG') \to SH(L,L')\,,$$
      the latter being the natural map into the direct limit. It is easy to show that this composition is independent of the choice of $(\cG,\cG')$, and that moreover these maps $SH(\cH,\cH') \to SH(L,L')$ form a morphism from the direct system $(SH(\cH,\cH'))_{(\cH,\cH')\in\cS_{K,K'}}$ to $SH(L,L')$. In particular it yields a morphism $\res^{(K,K')}_{(L,L')}$ as claimed. It also follows that $\res^{(K,K')}_{(K,K')}$ is the identity and that these restriction morphisms satisfy the cocycle identity.
  \item Let us prove normalization: if $(\cH,\cH') \in \cS_{K,\varnothing}$, and $\cH'$ consists of a given $C^2$-small Morse function plus $i$, then $SC(\cH') = 0$ as Varolg\"une\c s shows in \cite{varolgunes2018mayer}. Therefore the canonical map $SH(\cH,\cH') \to SH(\cH)$ is an isomorphism. It is also easy to see the compatibility with restrictions.
  \item Let us prove the triangle property. Let $K'\subset K$ and pick $\cH \in \cS_K$, $\cH' \in \cS_{K'}$ such that $\cH \preceq \cH'$. Pick a filling $\cF \fc \cH \to \cH'$. The morphism $SH(K') \to SH(K,K')[1]$ is defined as follows. We have a natural morphism $SC(\cH')[-1]\to SC(\cF)$, that is $SC(\cH') \to SC(\cF)[1]$ after shifting. Passing to homology, we obtain $SH(\cH') \to SH(\cF)[1]$. Using the above methods, it is easy to show that this morphism is compatible with the various comparison morphisms, and therefore induces a well-defined map $SH(K') \to SH(K,K')[1]$. We have the short exact sequence
      $$0 \to SC(\cH')[-1] \to SC(\cF) \to SC(\cH) \to 0\,.$$
      Its long exact homology sequence reads
      $$\dots \to SH(\cH')[-1] \to SH(\cF) \to SH(\cH) \to \dots$$
      It is also compatible with the various comparison morphisms, therefore we have the long exact sequence
      $$\dots SH(K')[-1]\to SH(K,K') \to SH(K)\to\dots$$
      as claimed.
  \item Let us prove the Mayer--Vietoris property. Fix acceleration data for $K,K',K'',K'\cup K'',K'\cap K''$, fill out the resulting partially defined Hamiltonian $3$-ray, and consider the following $3$-cube obtained from it, where $SC(\cdot)$ means $SC$ for the corresponding acceleration datum:
      $$\xymatrix{SC(K) \ar[rr] \ar@{=}[rd] \ar@{=}[dd] && SC(K'\cup K'') \ar[rd] \ar[dd]|\hole \\ & SC(K)\ar@{=}[dd] \ar[rr] && SC(K'') \ar[dd] \\ SC(K)\ar@{=}[rd] \ar[rr]|(0.45)\hole  && SC(K')\ar[rd] \\ & SC(K)\ar[rr] && SC(K'\cap K'')}$$
      Let us call the square coming from the left face of this cube by $\cA$, the right square by $\cB$, and the resulting map by $\cA \xrightarrow{\cF}\cB$. There results a short exact sequence of squares
      $$0 \to \cB[-1] \to \co_3\cF \to \cA\to 0\,.$$
      Since $\cA$ is clearly acyclic, and since $\cB$ is acyclic by \cite{varolgunes2018mayer}, it follows that $\co_3\cF$ is acyclic, which, thanks to  \cite{varolgunes2018mayer} results in a long exact homology sequence, which is precisely the Mayer--Vietoris triangle.
\end{enumerate}

\subsection{Constructing the product}\label{ss:constructing_product}

The goal here is to prove the existence of a diagram
\begin{equation}\label{eqn:constructing_product}
  \xymatrix{SH(K,K') \otimes SH(K,K'') \ar@{-->}[r]^-{\wt*} \ar[d] & SH(K,K'\cup K'') \ar[d] \\ SH(K) \otimes SH(K) \ar[r]^-{*} & SH(K)}
\end{equation}
provided $(K,K',K'') \in \mathcal{CTD}$, that is $K',K'' \subset K$ and $K',K''$ satisfy descent. Here the vertical arrows are the canonical restrictions while the bottom arrow is the Tonkonog--Varolg\"une\c s product \cite{tonkonog2020super}. We will construct the top arrow and prove that the resulting diagram commutes.

\begin{rem}\label{rem:SH_indeed_algebra}
  Specializing to the case $K' =K''$, we obtain a product structure on $SH^*(K,K')$. \emph{This is the nonunital algebra structure with which we equip $SH^*(K,K')$ to make it conform to the axiomatic framework of Section \ref{ss:axioms_rel_SH_pairs}.}
\end{rem}

In the following description $SC(K)$ and so on stand for $SC$ of a suitably chosen acceleration datum for $K$. The above diagram is constructed as follows:

\noindent\tb{Step 1:} Using the construction of Section \ref{ss:folding_pullbacks}, we will construct a module $Q$, which is the pullback of the V-shaped diagram consisting of $SC(K'),SC(K''),SC(K'\cap K'')$ and restrictions, and a natural morphism $p \fc SC(K) \to Q$. In addition, we will establish a commutative triangle:
\begin{equation}\label{eqn:comm_diag_cocone_p_SC_K_K_prime_dble_prime}
  \xymatrix{\co(p) \ar[d] & SC(K,K'\cup K'')\ar[l] \ar[ld] \\ SC(K)}
\end{equation}
where the vertical arrow is the natural projection, while the diagonal arrow is the restriction. The key point here is that \emph{the top arrow is a quasi-isomorphism; it is here that Varolg\"une\c s's Mayer--Vietoris sequence enters}, as alluded to in Section \ref{ss:rel_SH_discussion}.

\noindent\tb{Step 2:} We will construct a ``zigzag'' diagram, which on passing to cohomology yields a commutative diagram
      \begin{equation}\label{eqn:comm_diagram_coho_product}
        \xymatrix{H\big(SC(K,K')\otimes SC(K,K'')\big) \ar[r]\ar[d] & H\big(\co(p)\big)\ar[d] \\  H\big(SC(K)\otimes SC(K)\big)\ar[r] & SH(K)}
      \end{equation}

\noindent\tb{Step 3:} Using the natural transformation $H(\cdot)\otimes H(\cdot) \to H(\cdot\otimes \cdot)$, diagram \eqref{eqn:comm_diagram_coho_product}, and applying the cohomology functor to diagram \eqref{eqn:comm_diag_cocone_p_SC_K_K_prime_dble_prime}, we arrive at the diagram
      $$\scriptsize\xymatrix{SH(K,K') \otimes SH(K,K'')\ar[r] \ar[d] & H\big(SC(K,K')\otimes SC(K,K'')\big) \ar[r]\ar[d] & H\big(\co(p)\big)\ar[d] & SH(K,K'\cup K'')\ar[ld] \ar[l]_{\cong} \\ SH(K) \otimes SH(K) \ar[r] & H\big(SC(K)\otimes SC(K)\big)\ar[r] & SH(K)}$$
      The top rightmost arrow is an isomorphism as a consequence of the discussion in Step 1 above. Inverting it and suitably composing the other arrows, we arrive at the desired diagram \eqref{eqn:constructing_product}. In the next two subsections we will describe Steps 1, 2 in more detail.
\begin{rem}
  It is possible to show, using the techniques appearing in Section \ref{ss:SH_pair_well_defined}, that the product we construct is independent of the choices of acceleration data, almost complex structures, and so on, and that it indeed commutes with restriction morphisms.
\end{rem}

\subsubsection{Step 1}

In what follows we implicitly choose acceleration data for all the sets appearing in the diagrams, such that if two sets $A,B$ satisfy $B\supset A$, then the corresponding acceleration data are related by $\preceq$. Moreover, we choose suitable fillings between those acceleration data, as well as higher-dimensional Hamiltonian rays as necessary, which is possible by Remark \ref{rem:exist_fillings}. Consider the cube
$$\xymatrix{SC(K) \ar[rr] \ar[dd] \ar[rd] && SC(K'') \ar@{=}[rd] \ar[dd]|\hole & \\ & SC(K'\cup K'')\ar[rr]\ar[dd] && SC(K'')\ar[dd] \\ SC(K') \ar[rr]|(0.463)\hole \ar@{=}[rd] &&  SC(K'\cap K'') \ar@{=}[rd] & \\ & SC(K')\ar[rr] && SC(K'\cap K'') }$$
Here the order of coordinates is as follows: the first one is toward the reader, the second one is to the right, and the third one is down. This cube is obtained by choosing a $4$-ray of Hamiltonians as indicated in the previous paragraph, passing to the corresponding $4$-ray of Floer complexes, and taking its completed telescope. Note that the non-identity arrows are chain level restriction maps.

Let us denote
$$Q = \co\big(SC(K'')\oplus SC(K') \xrightarrow{\res^{K''}_{K'\cap K''}+\res^{K''}_{K'\cap K''}}SC(K'\cap K'')\big)\,;$$
this is the pullback of the suitable V-shaped diagram, which can be seen as part of the front and the back faces of the above cube. Now let us apply folding to the above cube and then take cocone in the vertical (that is third) direction. We obtain a square
$$\xymatrix{SC(K) = \co(SC(K)\to 0) \ar[r]^-p \ar[d]_{-\res^K_{K'\cup K''}} & Q \ar[d]^{-\id} \\ SC(K'\cup K'') = \co(SC(K'\cup K'')\to 0)\ar[r] & Q}$$
where the minus signs come out of the definition of cocones, see Remark \ref{rem:cocones_turn_ids_minus_ids}. Negating the vertical and diagonal maps, we obtain a homotopy commutative triangle, from which we can build the following square:
$$\xymatrix{SC(K) \ar@{=}[d] \ar[r]^-{\res} & SC(K'\cup K'')  \ar[d] \\ SC(K)\ar[r]_p & Q}$$
Applying cocone in the horizontal direction, we arrive at the following morphism of short exact sequences, see equation \eqref{eqn:exact_sequence_cocone}:
$$\xymatrix{0 \ar[r] & SC(K'\cup K'')[-1] \ar[r] \ar[d] & SC(K,K'\cup K'') \ar[r] \ar[d]  & SC(K)\ar[r] \ar@{=}[d] & 0 \\ 0\ar[r] & Q[-1]\ar[r] & \co(p)\ar[r] & SC(K)\ar[r] &  0}$$
Here the first vertical arrow is a quasi-isomorphism by the Mayer--Vietoris property for sets in descent \cite{varolgunes2018mayer}. Since the last vertical arrow is also a quasi-isomorphism, we arrive at the conclusion that so is $SC(K,K'\cup K'') \to \co(p)$. The desired commutative triangle is just the rightmost square. This completes Step 1.

\subsubsection{Step 2}

Here we will prove the existence of the following commutative ``zigzag'' diagram:
\begin{equation}\label{eqn:zigzag_diagram}
\xymatrix{SC(K,K')\otimes SC(K,K'')\ar@{=}[r] \ar[d] & SC(K,K')\otimes SC(K,K'') \ar[d] \\ \cdot \ar[r]\ar[d] & SC(K) \otimes SC(K) \ar[d]\\ \cdot \ar[r] &\cdot \\  \cdot \ar[u]^{\text{qis}} \ar[d] \ar[r] & \cdot \ar[u]^{\text{qis}} \ar[d]  \\ \co(p) \ar[r] & SC(K) }
\end{equation}
where $\cdot$ stands for unspecified modules, which we will describe below, and ``qis'' stands for ``quasi-isomorphism.'' Passing to cohomology, inverting the isomorphisms resulting from quasi-isomorphisms in the above diagram, and composing the other morphisms, we will obtain the diagram \eqref{eqn:comm_diagram_coho_product} above as claimed. This will complete the construction of the diagram \eqref{eqn:constructing_product}.

The diagram \eqref{eqn:zigzag_diagram} is obtained as follows. Below we describe seven $3$-cubes numbered I--VII. We will compose cubes I through IV, juxtapose the result with cubes V, VI, VII, obtaining four $3$-cubes written side-by-side as a ``zigzag'' diagram of $2$-cubes and maps between them of the form $\cdot \to \cdot \to\cdot \leftarrow \cdot \to \cdot$. To this diagram we apply folding and cocone as in Section \ref{ss:folding_pullbacks}, and the above diagram \eqref{eqn:zigzag_diagram} is obtained as a result. Let us now describe this in detail.

In the cubes, the order of coordinates is as follows: the first is down, the second is right, and the third is perpendicular to the page. Cubes I--IV are tensor products of lower-dimensional ones, cubes V, VI come from the functoriality of completed telescopes (see Section \ref{ss:rays_telescopes}), while the last cube comes from Floer theory.

\tb{Cube I:}
$$\small\xymatrix{SC(K,K') \otimes SC(K,K'')\ar[rr] \ar[dd]\ar@{=}[dr] && 0\ar[dr] \ar@{=}[dd]|\hole & \\ & SC(K,K') \otimes SC(K,K'') \ar[rr] \ar[dd]  && SC(K,K') \otimes SC(K'') \ar[dd]  \\ 0 \ar@{=}[rr]|(0.62)\hole \ar@{=}[rd] && 0\ar@{=}[rd] & \\ & 0 \ar@{=}[rr] && 0}$$
This cube is obtained as follows. Consider the square
$$\xymatrix{SC(K,K'') \ar@{=}[d] \ar[r] & 0\ar[d] \\ SC(K,K'')\ar[r]_-{\res\circ\pr} & SC(K'')}$$
as in Remark \ref{rem:cocones_htpy_ker_square} and tensor it with $SC(K,K')$ to obtain
$$\xymatrix{SC(K,K')\otimes SC(K,K'') \ar[r] \ar@{=}[d] & 0\ar[d] \\ SC(K,K')\otimes SC(K,K'') \ar[r] & SC(K,K')\otimes SC(K'')}$$
This is the top face of our cube and it remains to append the bottom face consisting of zeros.

\tb{Cube II:}
$$\tiny\xymatrix{SC(K,K') \otimes SC(K,K'') \ar@{=}[rd] \ar[rr] \ar[dd] && SC(K,K') \otimes SC(K'') \ar@{=}[rd] \ar[dd]|\hole & \\ & SC(K,K') \otimes SC(K,K'')  \ar[rr] \ar[dd] && SC(K,K') \otimes SC(K'')\ar[dd] \\
0\ar@{=}[rr]|(0.51)\hole \ar[rd] && 0\ar[rd] & \\ & SC(K') \otimes SC(K,K'') \ar[rr] && SC(K') \otimes SC(K'')}$$
This is obtained by tensoring the square
$$\xymatrix{SC(K,K') \ar@{=}[r] \ar[d] & SC(K,K') \ar[d]^{\res\circ\pr}  \\ 0\ar[r] & SC(K')}$$
as in Remark \ref{rem:cocones_htpy_ker_square} by the map $SC(K,K'') \xrightarrow{\res\circ\pr} SC(K'')$, using the $(2,1)$-shuffle $(1)(23)$.

\tb{Cube III:}
$$\tiny\xymatrix{SC(K,K') \otimes SC(K,K'') \ar[rr] \ar[rd] \ar[dd] && SC(K,K') \otimes SC(K'') \ar[rd] \ar[dd]|\hole & \\ & SC(K) \otimes SC(K,K'')\ar[rr] \ar[dd] && SC(K) \otimes SC(K'')\ar[dd] \\
SC(K') \otimes SC(K,K'')\ar@{=}[rd] \ar[rr]|(0.51)\hole && SC(K') \otimes SC(K'') \ar@{=}[rd]& \\ & SC(K') \otimes SC(K,K'')\ar[rr] && SC(K') \otimes SC(K'')}$$
This is obtained by tensoring the square
$$\xymatrix{SC(K,K') \ar[r]^-{\pr} \ar[d]_-{\res\circ\pr} & SC(K)\ar[d]^{\res} \\ SC(K')\ar@{=}[r] & SC(K')}$$
by the map $SC(K,K'') \xrightarrow{\res\circ\pr} SC(K'')$, using the $(2,1)$-shuffle $(1)(23)$.

\tb{Cube IV:}
$$\scriptsize\xymatrix{SC(K) \otimes SC(K,K'') \ar[rr] \ar[dd] \ar[rd] && SC(K) \otimes SC(K'') \ar@{=}[rd] \ar[dd]|\hole & \\ & SC(K) \otimes SC(K)\ar[rr] \ar[dd] && SC(K) \otimes SC(K'') \ar[dd] \\
SC(K') \otimes SC(K,K'') \ar[rd] \ar[rr]|(0.51)\hole && SC(K') \otimes SC(K'') \ar@{=}[rd]& \\ & SC(K') \otimes SC(K) \ar[rr] && SC(K') \otimes SC(K'')}$$
This is obtained by tensoring the map $SC(K) \xrightarrow{\res} SC(K')$ by the square
$$\xymatrix{SC(K,K'')\ar[r]^-{\pr} \ar[d]_{\res\circ\pr} & SC(K)\ar[d]^{\res} \\ SC(K'')\ar@{=}[r] & SC(K'')}$$

\tb{Cube V:}
$$\scriptsize\xymatrix{SC(K) \otimes SC(K) \ar[rd] \ar[rr] \ar[dd]  && SC(K) \otimes SC(K'')\ar[rd]  \ar[dd]|\hole & \\ & \wh{\tel CF(\cH) \otimes \tel CF(\cH)} \ar[rr] \ar[dd] && \wh{\tel CF(\cH) \otimes \tel CF(\cH'')} \ar[dd]  \\ SC(K') \otimes SC(K)\ar[rd]  \ar[rr]|\hole && SC(K') \otimes SC(K'') \ar[rd]& \\ & \wh{\tel CF(\cH') \otimes \tel CF(\cH)} \ar[rr]  && \wh{\tel CF(\cH') \otimes \tel CF(\cH'')} }$$
Here $\cH,\cH',\cH''$ are acceleration data for $K,K',K''$, respectively. The cube is obtained by applying the natural transformation $\wh\cdot \otimes \wh\cdot \to \wh{\cdot\otimes \cdot}$ termwise. Note that the cube is straight in the direction perpendicular to the page.

\tb{Cube VI:}
$$\scriptsize\xymatrix{\wh{\tel CF(\cH) \otimes \tel CF(\cH)}\ar[rr] \ar[dd]  && \wh{\tel CF(\cH) \otimes \tel CF(\cH'')} \ar[dd]|\hole & \\ & \wh\tel(CF(\cH)\otimes CF(\cH)) \ar[ul]\ar[rr] \ar[dd] && \wh\tel(CF(\cH)\otimes CF(\cH'')) \ar[dd] \ar[ul] \\ \wh{\tel CF(\cH')\otimes \tel CF(\cH)}\ar[rr]|\hole && \wh{\tel CF(\cH')\otimes \tel CF(\cH'')} & \\ & \wh\tel(CF(\cH')\otimes CF(\cH)) \ar[rr] \ar[ul] && \wh\tel(CF(\cH')\otimes CF(\cH'')) \ar[ul]}$$
The construction of this cube before completion is outlined at the end of Section \ref{ss:rays_telescopes}. Note that even after completion, the diagonal arrows remain quasi-isomorphisms, thanks to \cite[Corollary 2.3.6 (3)]{varolgunes2018mayer}.

\tb{Cube VII:}
$$\xymatrix{\wh\tel(CF(\cH)\otimes CF(\cH)) \ar[rr] \ar[dd] \ar[rd] && \wh\tel(CF(\cH)\otimes CF(\cH'')) \ar[rd] \ar[dd]|\hole & \\ & SC(K) \ar[rr] \ar[dd] && SC(K'')\ar[dd] \\
\wh\tel(CF(\cH')\otimes CF(\cH)) \ar[rd] \ar[rr]|\hole && \wh\tel(CF(\cH')\otimes CF(\cH'')) \ar[rd] & \\ & SC(K')\ar[rr] && SC(K'\cap K'')}$$
This is the Floer theoretic cube corresponding to products, which are the arrows perpendicular to the page, and restriction maps, which are the rest of the arrows. We note here that the back face of this cube coincides with the front face of cube VI.

We now compose cubes I-IV, take the resulting cube and juxtapose it with cubes V, VI, VII. There results a diagram of five $2$-cubes and four maps between them going as follows: $\cdot \to \cdot \to\cdot \leftarrow \cdot \to \cdot$, that is we have a ``zigzag'' in the middle. We can now apply folding to this diagram and then take cocone in the vertical direction, which results in the diagram on the left. The diagram on the right is obtained from it by taking its cocone in the horizontal direction and appealing to the natural map from the cocone to the domain of a map, see equation \eqref{eqn:exact_sequence_cocone}:

$$\small\xymatrix{SC(K,K')\otimes SC(K,K'') \ar[r] \ar[d] & 0 \ar[d] \\ SC(K) \otimes SC(K) \ar[d] \ar[r] & \cdot \ar[d] \\ \wh{\tel CF(\cH)\otimes \tel CF(\cH)} \ar[r] & \cdot \\ \wh\tel(CF(\cH)\otimes CF(\cH)) \ar[u]^{\text{qis}} \ar[r] \ar[d] & \cdot \ar[u]^{\text{qis}} \ar[d] \\ SC(K) \ar[r]^p & Q} \xymatrix{SC(K,K')\otimes SC(K,K'')\ar@{=}[r] \ar[d] & SC(K,K')\otimes SC(K,K'') \ar[d] \\ \cdot \ar[d] \ar[r] & SC(K) \otimes SC(K) \ar[d] \\ \cdot \ar[r] & \wh{\tel CF(\cH)\otimes \tel CF(\cH)}  \\ \cdot \ar[u]^{\text{qis}} \ar[d] \ar[r] & \wh\tel(CF(\cH)\otimes CF(\cH)) \ar[u]^{\text{qis}} \ar[d]  \\ \co(p) \ar[r] & SC(K) }$$
The required diagram \eqref{eqn:zigzag_diagram} is the one appearing on the right. This completes Step 2 and therefore the construction of the product on relative symplectic cohomology.

\section{Symplectic rigidity and computations}\label{s:examples_computations}

\subsection{Proof of Theorem \ref{thm:torus_cross}}\label{ss:pf_thm_torus_cross}

Here we prove Theorem \ref{thm:torus_cross}, based on Theorem \ref{thm:IVQM_annuli_in_tori}, which is proved in Section \ref{ss:SH_special_domains}. The other ingredient we need is Lemma \ref{lem:IVQM_products}, whose proof appears below.
\begin{proof}[Proof of Theorem \ref{thm:torus_cross}]
    Recall the statement of the theorem: any involutive map $f \fc \T^6\times S^2 \to B$, where $B$ is a surface, has a fiber which intersects all the sets of the form $T(a,b,c)\times\text{equator}$.

	Let $\tau$ be the quantum cohomology IVQM on $\T^6\times S^2$. The K\"unneth formula yields
	$$QH^*(\T^6 \times S^2) = QH^*(\T^6) \otimes QH^*(S^2) = H^{*\bmod 4}(\T^6;\Lambda)\otimes\Lambda\langle 1,h\rangle\,.$$
	Let $\alpha = [dq_1\wedge dq_2],\beta=[dp_1\wedge dp_3],\gamma=[dp_2\wedge dq_3] \in QH^*(\T^6)$. Consider the graded ideal $I\subset QH^*(\T^6\times S^2)$ generated by $\alpha\otimes 1,\beta\otimes 1,\gamma\otimes 1$. Since
	$$\alpha*\beta*\gamma = [dq_1\wedge dq_2\wedge dq_3\wedge dp_1\wedge dp_2\wedge dp_3] \neq 0\,,\quad\text{we have}$$
	$$(\alpha\otimes 1)*(\beta\otimes 1)*(\gamma\otimes 1) = (\alpha*\beta*\gamma)\otimes 1 \neq 0\,,$$
	which implies that $I^3 \neq 0$. Theorem \ref{thm:symplectic_Tverberg} thus implies that there is $b_0 \in B$ such that the corresponding fiber $f^{-1}(b_0)$ of $f$ intersects each compact $Z\subset \T^6\times S^2$ with $I\subset\tau(Z)$. It remains to show that
    $$I\subset \tau\big(T(a,b,c)\times\text{equator}\big)\,.$$

 Thanks to Theorem \ref{thm:IVQM_annuli_in_tori}\footnote{We need to rearrange the coordinates on $\T^6$ to apply the theorem for $\gamma$.} we have
	$$\alpha \in \tau(T_1(a))\,,\quad \beta \in \tau(T_2(b))\,,\quad \gamma \in \tau(T_3(c))\,,$$
	which by monotonicity implies that $\alpha,\beta,\gamma \in \tau(T(a,b,c))$. For any equator $L\subset S^2$, the complement $S^2\setminus L$ is the union of two displaceable open disks, and therefore by Lemma \ref{lem:IVQM_products} we have the crucial conclusion that $\alpha\otimes 1,\beta\otimes 1,\gamma\otimes 1 \in \tau(T(a,b,c)\times L)$, and consequently that $I \subset \tau(T(a,b,c)\times L)$, as required.
\end{proof}

It remains to prove Lemma \ref{lem:IVQM_products}.
\begin{proof}[Proof of Lemma \ref{lem:IVQM_products}]In this proof we abbreviate $\res_K\equiv\res^M_K$ and similarly for $N$ and $M\times N$. For the first assertion it suffices to show that for any neighborhood $V$ of $L$ we have $SH^*(V^c;\Lambda) = 0$. Let $N \setminus  L = W_1\cup\dots\cup W_k$ be a decomposition into pairwise disjoint displaceable open sets. It follows that $V^c = \bigcup_{i=1}^k (V^c\cap W_i)$, and moreover that each $V^c\cap W_i$ is displaceable, being contained in $W_i$, and compact, because it is the complement in $V^c$ of $\bigcup_{j\neq i}(V^c\cap W_j)$, which is open in $V^c$. Thus $SH^*(V^c\cap W_i;\Lambda) = 0$, and by the Mayer--Vietoris property
$$SH^*(V^c;\Lambda) = \bigoplus_i SH^*(V^c\cap W_i;\Lambda) = 0\,.$$

For the second assertion it is enough to prove that if $\alpha\in\tau(K)$ and $\beta\in \tau(L)=SH^*(N;\Lambda)$, then for every pair of neighborhoods $U\supset K$, $V \supset L$ we have
\[
	\res_{(U\times V)^c}(\psi(\alpha\otimes \beta))=0\,,
\]
since sets of the form $(U\times V)^c$ are cofinal in the collections of compacts disjoint from $K\times L$. To prove this, we need the following

\tb{Claim:} The restriction $\res^{(U\times V)^c}_{U^c\times N}$ is an isomorphism.

Assuming this for a moment, and using the naturality of the K\"unneth morphism with respect to restrictions, we obtain the commutative diagram
$$\xymatrix{SH^*(M;\Lambda) \otimes SH^*(N;\Lambda) \ar[r]^-\psi \ar[d]_{\res_{U^c} \otimes \id} & SH^*(M\times N;\Lambda) \ar[rr]^{\res_{(U\times V)^c}} \ar[d]_{\res_{U^c\times N}} && SH^*((U\times V)^c;\Lambda)\ar[dll]^{\res^{(U\times V)^c}_{U^c\times N}} \\ SH^*(U^c;\Lambda)\otimes SH^*(N;\Lambda) \ar[r]^-\psi & SH^*(U^c\times N;\Lambda)}$$
It follows that
\begin{align*}
  \res_{(U\times V)^c}(\psi(\alpha\otimes\beta)) &= \big(\res^{(U\times V)^c}_{U^c\times N}\big)^{-1}\big(\psi((\res_{U^c}\otimes \id)(\alpha\otimes\beta))\big)\\
  &= \big(\res^{(U\times V)^c}_{U^c\times N}\big)^{-1}\big(\psi(\underbrace{\res_{U^c}(\alpha)}_{=0}\otimes \beta)\big) = 0\,,
\end{align*}
as claimed. Here we used that $\res_{U^c}(\alpha)=0$, which follows from $\alpha \in \tau(K)$.

It remains to prove the above claim. First, we claim that the sets $U^c\times N,M\times V^c$ commute. Indeed, let $f\fc M \to[0,1]$ and $g\fc N\to[0,1]$ be smooth functions with $f^{-1}(0) = U^c$ and $g^{-1}(0) = V^c$; then the map $f\times g \fc M \times N \to [0,1]^2$ is involutive and
$$U^c\times N = (f\times g)^{-1}(\{0\}\times[0,1])\,,\quad M\times V^c = (f\times g)^{-1}([0,1]\times \{0\})\,,$$
and therefore these sets commute thanks to Remark \ref{rem:preimages_involutive_maps}. Noting that $(U^c\times N)\cup(M\times V^c) = (U\times V)^c$ and $(U^c\times N)\cap(M\times V^c) = U^c\times V^c$, the corresponding exact Mayer--Vietoris triangle reads
$$\xymatrix{SH^*\big((U\times V)^c;\Lambda\big) \ar[rr]^-{\big(\res^{(U\times V)^c}_{U^c\times N},\res^{(U\times V)^c}_{M\times V^c}\big)} && SH^*(U^c\times N;\Lambda)\oplus SH^*(M\times V^c;\Lambda) \ar[ld] \\ & SH^*(U^c\times V^c;\Lambda) \ar[ul]}$$
Since $V^c$ is a finite union of pairwise disjoint displaceable compact sets, so are $M\times V^c$ and $U^c\times V^c$, therefore $SH^*(M\times V^c;\Lambda) = SH^*(U^c\times V^c;\Lambda) = 0$ by the Mayer--Vietoris property, whence the top arrow in the triangle is the desired isomorphism
$$SH^*\big((U\times V)^c;\Lambda\big) \xrightarrow{\res^{(U\times V)^c}_{U^c\times N}} SH^*(U^c\times N;\Lambda)\,.$$
\end{proof}

\subsection{Proof of Theorem \ref{thm:IVQM_annuli_in_tori}}\label{ss:SH_special_domains}

Recall the formulation of the theorem: If $(M=\T^{2n},\omega=dp\wedge dq)$ is the standard symplectic torus and $S\subset \T^n$ is closed or open, then
$$\tau(S\times\T^n) = \mu_M(S\times\T^n)=\mu_{\T^n}(S)\otimes H^*(\T^n;\Lambda)\subset H^*(\T^n;\Lambda)\otimes H^*(\T^n;\Lambda)= H^*(M;\Lambda)\,,$$
where the last equality is the K\"unneth formula, while for a space $X$, $\mu_X$ stands for the cohomology IVM on $X$. Also recall that $QH^*(M)=H^*(M;\Lambda)$.

Let us first prove this when $S$ is closed. Let $h\fc \T^n\to[0,1]$ be a smooth function which vanishes exactly on $S$. Let $A\subset (0,1)$ be the set of its regular values and define $Q_\alpha=\{h\geq\alpha\}$ for $\alpha \in A$. The collection $(Q_\alpha\times\T^n)_{\alpha\in A}$ is cofinal in the family of compacts which are disjoint from $S\times \T^n$. It follows that
$$\tau(S\times \T^n) = \bigcap_{\alpha\in A}\ker \big(QH^*(M)\to SH^*(Q_\alpha\times\T^n;\Lambda)\big)\,.$$
Since $\partial Q_\alpha\subset \T^n$ is a smooth coorientable hypersurface, the inclusion $Q_\alpha\hookrightarrow\T^n$ can be extended to an embedding $\iota \fc (-\epsilon,\epsilon)\times\partial Q_\alpha\hookrightarrow \T^n$ for some $\epsilon>0$. The induced embedding $\wt\iota:=\iota\times\id_{\T^n} \fc (-\epsilon,\epsilon)\times \partial Q_\alpha\times\T^n \hookrightarrow M$ then has the property that for each $\rho\in(-\epsilon,\epsilon)$, $\wt\iota(\{\rho\}\times\partial Q_\alpha\times\T^n)$ has no closed contractible characteristics. Indeed, if $\Sigma\subset \T^n$ is any hypersurface, then each closed charateristic of $\Sigma\times\T^n$ has the form $\pt\times\gamma$, where $\gamma\subset \T^n$ is a straight circle, which implies that it is noncontractible. Since $\wt\iota(\{\rho\}\times\partial Q_\alpha\times\T^n) = \iota(\{\rho\}\times \partial Q_\alpha)\times\T^n$, our claim follows, and thus Theorem \ref{thm:domain_no_contr_orbits_bdry_SH_equal_H} applies to the region $Q_\alpha\times\T^n$, since its boundary is $\partial (Q_\alpha\times\T^n)=\partial Q_\alpha\times \T^n$, and we obtain
$$\ker \big(QH^*(M)\to SH^*(Q_\alpha\times\T^n;\Lambda)\big) = \ker\big(H^*(M;\Lambda)\to H^*(Q_\alpha\times\T^n;\Lambda)\big)\,,$$
whence
$$\tau(S\times \T^n) = \bigcap_{\alpha\in A}\ker\big(H^*(M;\Lambda)\to H^*(Q_\alpha\times\T^n;\Lambda)\big)\,,$$
which equals $\mu_M(S\times\T^n)$ by the same cofinality property. Finally, the equality
$$\mu_M(S\times\T^n) = \mu_{\T^n}(S)\otimes H^*(\T^n;\Lambda)$$
follows from the cofinality of $(Q_\alpha)_{\alpha\in A}$ in the family of compact subsets of $\T^n$ which are disjoint from $S$, and the K\"unneth formula.

If $S\subset \T^n$ is open, then, since the collection of sets of the form $K\times\T^n$, where $K\subset S$ is compact, is cofinal in the family of compacts contained in $S\times\T^n$, it follows that
\begin{multline*}
\tau(S\times\T^n)=\bigcup_{K\text{ cpt}\subset S}\tau(K\times\T^n)=\bigcup_{K\text{ cpt}\subset S}\mu_M(K\times\T^n)\\=\bigcup_{K\text{ cpt}\subset S}\mu_{\T^n}(K)\otimes H^*(\T^n;\Lambda)=\mu_{\T^n}(S)\otimes H^*(\T^n;\Lambda)\,,
\end{multline*}
as claimed. The proof is complete.

\subsection{Proof of Theorem \ref{thm:Liouville_dom_idx_bdd_SH}}\label{ss:pf_thm_SH_Liouv_domains}

Recall that the theorem asserts that given a contact-type region $K$ with incompressible index-bounded boundary in a closed symplectically aspherical symplectic manifold $(M,\omega)$, we have:
\begin{enumerate}
  \item There is a canonical isomorphism $SH^*(K;\Lambda) \cong SH^*_{\cl}(\wh K;\Lambda)$, where $SH^*_{\cl}(\wh K;\Lambda)$ is the classical symplectic cohomology of the completion $\wh K$;
  \item $\ker \big(H^*(M;\Lambda) \to H^*(K;\Lambda)\big) \subset \ker \big(\res^M_K \fc SH^*(M;\Lambda) \to SH^*(K;\Lambda)\big)$.
\end{enumerate}

First, in Section \ref{sss:class_weighted_HF} we describe the relation between, on the one hand, the Floer complexes we use in this paper, as described in Section \ref{ss:Floer_data_rays}, where the differentials, continuation maps, and so on, carry weights which are suitable powers of the Novikov parameter $T$, and, on the other hand, unweighted Floer complexes. Then in Section \ref{sss:SH_Liouv_domains} we recall the definition of the classical symplectic cohomology of the completion $\wh K$ of $K$ in two incarnations---the weighted $SH^*(\wh K;\Lambda)$ and the unweighted $SH^*_{\cl}(\wh K;\Lambda)$---and establish a canonical identification between them. In Section \ref{sss:SH_K_equals_SH_completion_K} we construct an isomorphism $SH^*(\wh K;\Lambda)\cong SH^*(K;\Lambda)$, thereby proving item (i). Finally, in Section \ref{sss:pf_ker_res_M_K} we prove the containment assertion (ii). We omit the almost complex structures from the notation throughout.

\subsubsection{Weighted and unweighted Floer cohomology}\label{sss:class_weighted_HF}

Here we establish an isomorphism between the weighted and the unweighted Floer theories. Given a nondegenerate Hamiltonian $H$, its unweighted Floer cochain complex over $\Lambda$ is
$$CF^*_{\text{uw}}(H;\Lambda) = \bigoplus_{x \in \cP^\circ(H)}\Lambda\cdot x$$
carrying the differential
\begin{equation}\label{eqn:Floer_diff_unweighted}
  d_{\text{uw}}x = \sum_{y \in \cP^\circ(H)}\#\cM(H;x,y)\,y\,,
\end{equation}
where $\cM(H;x,y)$ is the moduli space of Floer trajectories of $H$ from $x$ to $y$ of index difference $1$. We let $HF^*_{\text{uw}}(H;\Lambda)$ be the cohomology of $(CF_{\text{uw}}^*(H;\Lambda),d_{\text{uw}})$.

As described in Section \ref{ss:Floer_data_rays}, the complex we use in this paper is the $\Lambda_{\geq0}$-module
$$CF^*(H) = \bigoplus_{x \in \cP^\circ(H)}\Lambda_{\geq 0}\cdot x$$
with the weighted differential
$$dx = \sum_{y \in \cP^\circ(H)}\#\cM(H;x,y)T^{\cA_H(y) - \cA_H(x)}\,y\,,$$
where $\cA_H$ is the action functional, which is well defined on contractible loops due to asphericity. Note that in Section \ref{ss:Floer_data_rays} the weight we mentioned is the topological energy of a Floer cylinder, which in the aspherical case equals the action difference between the orbits to which it is asymptotic.

The following relates the two constructions.
\begin{prop}\label{prop:Weighted_vs_classical_CF}
     The map
     $$\psi \fc CF^*_{\text{\rm uw}}(H;\Lambda)\to CF^*(H)\otimes_{\Lambda_{\geq0}}\Lambda \,,\qquad x \mapsto x\otimes T^{\cA_H(x)}$$
     is a chain isomorphism. It is compatible with continuation maps on both sides. It induces a complete identification between unweighted Floer theory with coefficients in $\Lambda$ and its weighted counterpart: $HF^*_{\text{\rm uw}}(\cdot;\Lambda)\cong HF^*(\cdot)\otimes\Lambda\,.$

\end{prop}
\begin{proof}
    Since the two theories count the same moduli spaces in the same way, with the difference lying in the $T$-weights, the first assertion is easily verified by keeping track of the powers of $T$ in appropriate commutative diagrams. For the second assertion, note the flatness of $\Lambda$ as a $\Lambda_{\geq0}$-module, which we use in the last isomorphism in the following:
    \begin{equation}
    \label{eqiso-vsp}
    HF^*_{\text{\rm uw}}(H;\Lambda) = H^*(CF^*_{\text{\rm uw}}(H;\Lambda)) \cong H^*(CF^*(H)\otimes\Lambda)= HF^*(H)\otimes\Lambda\,.\end{equation}
\end{proof}


\subsubsection{Symplectic cohomology of regions with contact-type boundary}\label{sss:SH_Liouv_domains}

Here we take $K\subset M$ to be a region with contact-type  boundary, not necessarily incompressible or index-bounded. We denote by $Y$ the Liouville vector field defined on a neighborhood of $\Sigma=\partial K$, let $\lambda=\iota_Y\omega$ be the corresponding Liouville form, and let $\alpha=\lambda|_\Sigma$ be the induced contact form. The symplectization of $(\Sigma,\alpha)$ is then $(\Sigma\times(0,\infty)_r,d(r\alpha))$. We say that $a \in \R$ is \tb{noncharacteristic} if it is not the period of a Reeb orbit of $\Sigma$, which is contractible in $K$.

We fix $\epsilon > 0$ such that the map $\Sigma \times (1-\epsilon,1+\epsilon)_r \to M$, $(z,r)\mapsto \phi_Y^{\ln r}(z)$ is a well-defined smooth embedding, where $\phi_Y^t$ is the local flow of $Y$. This map is then a symplectomorphism onto its image, where $\Sigma\times(1-\epsilon,1+\epsilon)$ is endowed with the restriction of $d(r\alpha)$. In what follows we identify this part of the symplectization with the image of the above map in $M$.

Throughout we pick a generic almost complex structure which is cylindrical near $\Sigma$, that is in Liouville coordinates in the neighborhood $\Sigma\times (1-\epsilon,1+\epsilon)$ it is $r$-invariant, preserves the contact structure on $\Sigma$, and maps $Y$ to the Reeb vector field along $\Sigma\times \lbrace 1 \rbrace$. We omit it from notation.

Let $\widehat{K}$ denote the completion of $K$, obtained by attaching the positive end of the symplectization of $\Sigma$, that is $(\Sigma \times [1,\infty) , d(r\alpha))$, to $K$ along $\Sigma$. Let us first precisely define $SH^*(\widehat{K};\Lambda)$ and $SH_{\cl}^*(\widehat{K};\Lambda)$---the weighted and the unweighted classical symplectic cohomology of $\wh K$---in our setting.

Similarly to Definition \ref{def:accel_datum_K}, we define an \tb{asymptotically linear acceleration datum for} $K\subset \wh K$ to be an acceleration datum $(H_n)_n$ for $K$ in $\wh K$, with the additional requirement that each $H_n$ be of the form $a_nr+b_n$ outside a neighborhood of $K$. Note that our conventions in particular imply that each $H_n$ is nondegenerate, which implies that $a_n$ is noncharacteristic, and that we have fixed nondecreasing homotopies $(H_n^s)_{s\in\R}$ from $H_n$ to $H_{n+1}$.

%

We consider the Floer complexes $CF^*(H_n)$, defined in Section \ref{ss:Floer_data_rays}, and for each $n$ we define a continuation map $\Phi_n \fc CF^*(H_n) \to CF^*(H_{n+1})$ by counting the corresponding continuation solutions, weighted by their topological energy, namely,
\[
	\Phi_n(x) = \sum_{y \in \mathcal{P}^\circ(H_{n+1})} \#\mathcal{M}((H_n^s)_s;x,y)\cdot T^{\mathcal{A}_{H_{n+1}}(y) - \mathcal{A}_{H_n}(x)}\cdot y\,.
\]
These yield a $1$-ray:
\[
  CF^*(H_1) \xrightarrow{\Phi_1} CF^*(H_2) \xrightarrow{\Phi_2} CF^*(H_3) \to \dots
\]
We put
\begin{equation}\label{eq:def_SC_classical_weighted}
  SH^*(\widehat{K};\Lambda) := H\Big(\varinjlim_{n\to\infty} CF^*(H_n)\otimes \Lambda\Big) \cong H^*\Big(\varinjlim_{n\to\infty} CF^*(H_n)\Big) \otimes \Lambda\,,
\end{equation}
where the last isomorphism is due to the flatness of $\Lambda$.

The ``classical'' symplectic cohomology $SH^*_{\cl}(\wh K;\Lambda)$ was defined by Viterbo in \cite{viterbo1999functorsI} using unweighted differentials and continuation maps, namely:
\begin{equation}\label{eq:def_SH_classical_unweighted}
SH_{\cl}^*(\widehat{K};\Lambda) := H\left( \varinjlim_{n\to\infty} CF_{\text{\rm uw}}^*(H_n;\Lambda) \right)=\varinjlim_{n\to\infty}HF^*_{\text{\rm uw}}(H_n;\Lambda)\,.
\end{equation}
Since by Proposition \ref{prop:Weighted_vs_classical_CF}, $HF_{\text{\rm uw}}^*(\cdot;\Lambda) = HF^*(\cdot) \otimes_{\Lambda_{\geq0}} \Lambda$, the two versions agree:
	\begin{multline}\label{eq:SH_cl_eq_SH_domain}
		SH_{\cl}^*(\widehat{K};\Lambda) = \varinjlim_{n\to\infty}HF^*_{\text{\rm uw}}(H_n;\Lambda) = \varinjlim_{n\to\infty}(HF^*(H_n)\otimes\Lambda)\\
        =\Big( \varinjlim_{n\to\infty} HF^*(H_n)\Big)\otimes\Lambda = H\Big(\varinjlim_{n\to\infty} CF^*(H_n)\Big)\otimes\Lambda   = SH^*(\widehat{K};\Lambda)\,.
	\end{multline}
\noindent Note that any other asymptotically linear acceleration datum for $K\subset \wh K$ yields the same homology groups, as can be shown using continuation morphisms between the Hamiltonians in the two acceleration data.

\subsubsection{Relative $SH$ for index bounded regions}\label{sss:SH_K_equals_SH_completion_K}
This section is dedicated to proving the first assertion of Theorem \ref{thm:Liouville_dom_idx_bdd_SH}. We keep the notations from the beginning of the previous subsection. Here we assume in addition that $K$ has incompressible and index-bounded boundary. The incompressibility implies that a loop in $\Sigma$ is contractible in $\Sigma$ if and only if it is in $K$ and in $M$. Thus $a \in \R$ is noncharacteristic if and only if it is not the period of a Reeb orbit of $\Sigma$ which is contractible in $\Sigma$, $K$, or $M$.

Since we have just shown that $SH^*_{\cl}(\wh K;\Lambda)=SH^*(\wh K;\Lambda)$, it remains to construct an isomorphism $SH^*(\wh K;\Lambda) \cong SH^*(K;\Lambda)$, which is what we do here.

To compute $SH^*(K;\Lambda)$, we construct a suitable acceleration datum where we can separate the generators lying inside and outside $K$ by action. Recall from the beginning of the previous subsection that we have identified a neighborhood of $\Sigma$ with the portion $\Sigma\times(1-\epsilon,1+\epsilon)_r$ of the symplectization of $\Sigma$. Choose a positive decreasing sequence $\epsilon_n\to 0$ with $4\epsilon_1<\epsilon$. We choose an increasing sequence of smooth autonomous Hamiltonians on $M$ as follows:
\begin{equation}\label{eq:def_H_n_tilde}
  \wt H_n(x)=\begin{cases}
    -1/n\,, & x \in K \\
    f_n(r)\,, & x\in \Sigma\times[1,1+4\epsilon_n] \\
    n\,, & x \notin K\cup\Sigma\times[1,1+4\epsilon_n]\,,
  \end{cases}
\end{equation}
where $f_n \fc [1,1+4\epsilon_n] \to \R$ is a monotone increasing smooth function such that $f_n(r) = a_nr+b_n$ on $[1+\epsilon_n,1+3\epsilon_n]$ for some constants $a_n,b_n$ with $a_n>0$ being noncharacteristic. We also require that $\epsilon_na_n$ be a bounded sequence. Note that the contractible $1$-periodic orbits of $\wt H_n$ in $\Sigma\times[1,1+4\epsilon_n]$ all have the form $\gamma\times\{r\}$, where $\gamma$ is a contractible Reeb orbit of $\Sigma$ while $r\in(1,1+4\epsilon_n)$ is such that $f_n'(r)$ is the period of $\gamma$. The action of such an orbit is $\cA_{\wt H_n}(\gamma\times\{r\}) = f_n(r) - rf_n'(r)$, the $y$-intercept of the tangent to the graph of $f_n$ at $r$.

By our nondegeneracy assumption, the contractible Reeb orbits of $\alpha$ are isolated, there is only a finite number of geometrically distinct such orbits, and the set of periods is discrete. It follows that the $1$-periodic orbits of $\wt H_n$ we have just described come as a finite collection of isolated circles, and if two such circles share the same value of $r$, they are disjoint in $\Sigma$.

We now construct a nondegenerate perturbation $H_n$ of $\wt H_n$, as follows: we multiply $\widetilde{H}_n\vert^{}_K$ by a positive $C^2$-small Morse function in $K$ (the same function for all $n$), add a $C^2$-small Morse function in $M\setminus (K\cup \Sigma \times (1,1+4\varepsilon_n))$, leave the linear part untouched, and perform a small time-dependent perturbation in the nonlinear parts, such that the perturbation only happens in a small enough neighborhood of each isolated circle of $1$-periodic orbits of $\wt H_n$, and such that each such circle splits into two $1$-periodic orbits of $H_n$.

The contractible $1$-periodic orbits of $H_n$ therefore fall into four groups: $\cP_{\mathrm L}(H_n)$ and $\cP_{\mathrm H}(H_n)$ (`L' and `H' for `low' and `high'), which are the critical points in $K$ and in $M\setminus (K\cup\Sigma\times[1,1+4\epsilon_n])$, respectively, and $\cP_\downarrow(H_n),\cP_\uparrow(H_n)$, which are the nonconstant orbits in $\Sigma\times[1,1+\epsilon_n]$ and $\Sigma\times[1+3\epsilon_n,1+4\epsilon_n]$, respectively. Due to the asphericity of $M$, these orbits can be assigned integer-valued Conley--Zehnder indices. We will denote by $\cP^j_\bullet(H_n)$ the subset of $\cP_\bullet(H_n)$ consisting of orbits of index $j$.

Note that to each $x \in \cP_\downarrow(H_n)\cup\cP_\uparrow(H_n)$ we can uniquely assign a circle of Reeb orbits of $\Sigma$ from which $x$ originates, and that the Conley--Zehnder index of a Reeb orbit $\gamma$ in such a circle is a linear function of the Conley--Zehnder index of $x$ with universally bounded coefficients. Moreover, $\cA_{H_n}(x)$ is very close to $\cA_{\wt H_n}(\gamma\times\{r\})$, where $f_n'(r)$ is the period of $\gamma$.

Thus for $j\in\Z$, the set of Conley--Zehnder indices of Reeb orbits corresponding to $\cP_\downarrow^j(H_n)\cup\cP_\uparrow^j(H_n)$ is contained in a fixed interval independent of $n$. The same is then true of the set of periods of such Reeb orbits, due to the index-boundedness assumption. Therefore there are constants $\beta_1(j),\beta_2(j),\beta_3(j)>0$, independent of $n$, such that
$$\cA_{H_n}(\cP_\downarrow^j(H_n)) \subset [-\beta_1(j),0]\,,\qquad \cA_{H_n}(\cP_\uparrow^j(H_n)) \subset [n-\beta_2(j),n+\beta_3(j)]\,.$$
Here we have implicitly used the fact that $\epsilon_na_n$ is a bounded sequence.

\begin{figure}[h]
	\centering
	\includegraphics[scale=1]{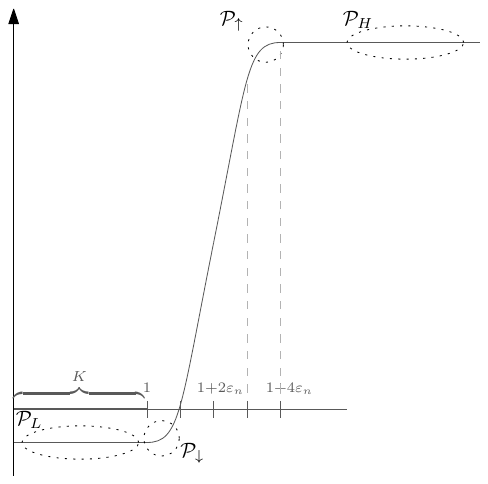}
	\caption{\small{The Hamiltonian $H_n$ and the four distinguished sets of 1-periodic orbits
	}}
	\label{fig:HamiltonianGreyN}
\end{figure}

Fix $j \in \Z$. We will construct an isomorphism $SH^j(\wh K;\Lambda) \cong SH^j(K;\Lambda)$.
\begin{notation}
Given a $\Z$-graded module $A^*$, we let $A^{\sim j}$ be the graded module such that $(A^{\sim j})^i = A^i$ for $i=j-1,j,j+1$, and $(A^{\sim j})^i=0$ otherwise. If $A^*$ is a cochain complex, we endow $A^{\sim j}$ with the obvious differentials. Note that $H^j(A^*) = H^j(A^{\sim j})$. If $S^* =\bigsqcup_{i\in\Z}S^i$ is a $\Z$-graded set, we let $S^{\sim j}=\bigsqcup_{i\in\{j-1,j,j+1\}}S^i$.
\end{notation}
\noindent It follows that in order to compute $SH^j(K;\Lambda)$, it is enough to consider the complexes $CF^{\sim j}(H_n)$. Denote
$$\cP_-(H_n)=\cP_{\mathrm L}(H_n)\cup\cP_\downarrow(H_n)\,,\qquad \cP_+(H_n)=\cP_{\mathrm H}(H_n)\cup\cP_\uparrow(H_n)\,,$$
and similarly for the corresponding sets of orbits of fixed indices. We let $CF^*_+(H_n)$ be the submodule of $CF^*(H_n)$ generated by the subset $\cP_+(H_n)\subset \cP^\circ(H_n)$, and put $CF^*_-(H_n)=CF^*(H_n)/CF^*_+(H_n)$.

The above discussion implies that there is $n_0$ such that for all $n\geq n_0$:
\begin{itemize}
  \item All the actions $\cA_{H_n}(\cP_+^{\sim j}(H_n))$ are strictly above all the actions $\cA_{H_n}(\cP_-^{\sim j}(H_n))$, and since the Floer differential does not decrease actions, it follows that $CF^{\sim j}_+(H_n)$ is a \emph{subcomplex} of $CF^{\sim j}(H_n)$, and that $CF^{\sim j}_-(H_n) = CF^{\sim j}(H_n)/CF^{\sim j}_+(H_n)$ inherits the quotient differential;
  \item All the actions $\cA_{H_n}(\cP_+^{\sim j}(H_n))$ lie strictly above all the actions $\cA_{H_{n+1}}(\cP_-^{\sim j}(H_{n+1}))$, and since the homotopy from $H_n$ to $H_{n+1}$ is nondecreasing, the corresponding continuation morphisms do not decrease actions, and thus map $CF^{\sim j}_+(H_n)$ into $CF^{\sim j}_+(H_{n+1})$.
\end{itemize}

We thus arrive at the following commutative diagram of cochain complexes and cochain maps, where the rows are short exact sequences, the middle vertical arrows are the continuation maps, and where the left and right ones are induced from them:
\[
\xymatrix{
  0 \ar[r] &  CF^{\sim j}_+(H_n) \ar[d] \ar[r]^-{i} &  CF^{\sim j}(H_n) \ar[d] \ar[r]^-{p} &  CF_-^{\sim j}(H_n)  \ar[d] \ar[r] & 0 \\
  0 \ar[r] &  CF^{\sim j}_+(H_{n+1}) \ar[d] \ar[r]^-{i} &  CF^{\sim j}(H_{n+1}) \ar[d] \ar[r]^-{p} &  CF_-^{\sim j}(H_{n+1}) \ar[d]  \ar[r] & 0  \\
   & \vdots & \vdots & \vdots &
}
\]
One can verify from the definitions that taking telescope of the $1$-rays involved preserves the exactness, namely we obtain the following short exact sequence:

\[
\xymatrix{
  0 \ar[r] & \operatorname*{tel}\limits_n CF^{\sim j}_+(H_n) \ar[r]^-{i} & \operatorname*{tel}\limits_n  CF^{\sim j}(H_n)  \ar[r]^-{p} & \operatorname*{tel}\limits_n  CF_-^{\sim j}(H_n) \ar[r] & 0
}
\]
Since the rightmost module is free, thus in particular flat, Lemma \ref{lem:completion_exact_right_mod_flat} implies that the exactness is preserved by the completion:
\begin{equation}\label{eq:exact_seq_completed_tel}
\xymatrix{
  0 \ar[r] & \widehat{\operatorname*{tel}\limits_n}\,CF^{\sim j}_+(H_n) \ar[r]^-{i} & \widehat{\operatorname*{tel}\limits_n}\,CF^{\sim j}(H_n) \ar[r]^-{p} &  \widehat{\operatorname*{tel}\limits_n}\,CF_-^{\sim j}(H_n) \ar[r] & 0
}\,.
\end{equation}
Note that if we replace the sequence $(H_n)_{n\geq n_0}$ by any subsequence $(H_{n_k})_k$, this whole algebraic argument works \emph{mutatis mutandis}. In particular, we can choose a subsequence in such a way that the aforementioned action separation guarantees the existence of some $c>0$ such that the continuation morphisms map $CF^{\sim j}_+(H_{n_k})$ into $T^cCF^{\sim j}_+(H_{n_{k+1}})$. At this point we denote this subsequence again by $(H_n)_n$ by abuse of notation.

Now we will use the following fact, whose proof is an immediate verification from the definitions. We will also use this in the sequel.
\begin{lemma}\label{lem:vanishing_in_completion}
	Consider a $1$-ray of modules $C_i$ over $\Lambda_{\geq0}$,
	\[
	  C_1 \xrightarrow{\psi_1} C_2 \xrightarrow{\psi_2} C_3 \xrightarrow{\psi_3} \cdots\, .
	\]
	Let $x_1\in C_1$ and denote its images in the modules $C_i$ by $x_i$, namely, $x_i = \psi_{i-1}(x_{i-1})$. Assume that there exists $c>0$, so that for all $i$, $\psi_i (x_i) \in  T^c\cdot C_{i+1}$.
	
	Then, the image of $x_1$ in $\widehat{\varinjlim C_i}$ is zero, and moreover, if for every $i$, we have $\psi_i (C_i) \subset  T^c\cdot C_{i+1}$, then $\widehat{\varinjlim C_i} = 0$. \qed
\end{lemma}
\noindent The lemma then implies that $\wh{\varinjlim}{}_n CF^{\sim j}_+(H_n) = 0$, in particular it is acyclic, and therefore so is the complex $\widehat{\operatorname*{tel}_n}\,CF^{\sim j}_+(H_n)$, since the two are quasi-isomorphic. It follows that the map $p$ in \eqref{eq:exact_seq_completed_tel} is a quasi-isomorphism.

Now, in order to compare $SH^*(K;\Lambda)$ and $SH^*(\wh K;\Lambda)$, we will define an asymptotically linear acceleration datum for $K\subset \wh K$, starting from $H_n$, as follows. Recall that in the definition of $\wt H_n$, equation \eqref{eq:def_H_n_tilde}, they are of the form $a_nr+b_n$ in $\Sigma\times[1+\epsilon_n,1+3\epsilon_n]$, and that $H_n$ is obtained by perturbing $\wt H_n$ \emph{outside} this region. We now let
$$\wh H_n(z) =\left\{\begin{array}{ll}H_n(z)\,, & z \in K \cup \Sigma\times [1,1+3\epsilon_n] \\ a_nr+b_n\,, & z \in \Sigma \times [1+\epsilon_n,\infty)\end{array}\right.$$
This clearly is an asymptotically linear acceleration datum for $K\subset \wh K$, thus by definition
$$SH^j(\wh K;\Lambda) = H^j\Big(\varinjlim_n CF^{\sim j}(\wh H_n)\otimes \Lambda\Big)\,.$$
Note the obvious identification $\cP^\circ(\wh H_n) = \cP_-(H_n)$. On the other hand, the differential of $CF_-^{\sim j}(H_n)$ counts Floer trajectories in $M$ connecting orbits in $\cP_-(H_n)$. Since our almost complex structures are cylindrical in the ``neck'' region $\Sigma\!\times\!(1,1+4\varepsilon_n)$, by the ``no escape" lemma, \cite[Lemma 3.4]{ganor2020floer} \cite[Lemma 2.2]{cieliebak2018symplectic}, all these Floer solutions are contained in $K$, and the same applies to the continuation solutions going between those generators. It follows that
\begin{equation}\label{eq:CF_H_n_identified_CF_wh_H_n}
  CF^{\sim j}(\wh H_n) = CF^{\sim j}_-(H_n)
\end{equation}
as complexes, and this identification commutes with continuation maps on both sides. We thus have the following isomorphisms:
\begin{multline}\label{eq:SH_K_hat_limit_CF_H_n}
  SH^j(\wh K;\Lambda) = H^j\big(\varinjlim_n CF^{\sim j}(\wh H_n)\otimes \Lambda\big)\\ = H^j\big(\varinjlim_n CF^{\sim j}_-(H_n)\otimes \Lambda\big) = H^j\big(\varinjlim_n CF^{\sim j}_-(H_n)\big)\otimes\Lambda\,.
\end{multline}

\noindent To finish the argument, we need the following
\begin{lemma}\label{lem:CF_minus_complete}
  The module $\varinjlim_n CF^{\sim j}_-(H_n)$ is complete.
\end{lemma}
\noindent Assuming this for a moment, and picking up at the end of \eqref{eq:SH_K_hat_limit_CF_H_n}, we have
\begin{multline*}
  \textstyle SH^j(\wh K;\Lambda) = H^j\big(\varinjlim_n CF^{\sim j}_-(H_n)\big)\otimes\Lambda = H^j\big(\wh\varinjlim_n CF^{\sim j}_-(H_n)\big)\otimes\Lambda\\
   = H^j\big(\wh\tel_n CF^{\sim j}_-(H_n)\big)\otimes\Lambda \stackrel{*}{=} H^j\big(\wh\tel_n CF^{\sim j}(H_n)\big)\otimes\Lambda = SH^j(K;\Lambda)\,,
\end{multline*}
where $\stackrel{*}{=}$ is thanks to the fact that $p$ in \eqref{eq:exact_seq_completed_tel} is a quasi-isomorphism.

This finishes the proof of item (i) of Theorem \ref{thm:Liouville_dom_idx_bdd_SH}, modulo Lemma \ref{lem:CF_minus_complete}, which we will now prove. First, we have the following algebraic result.
\begin{lemma}\label{lem:direct_limit_continuations_complete}
  Let $(\Phi_n)_n$ be a sequence of endomorphisms of the free module $\Lambda_{\geq0}^k$, such that for each $n$ we have
  $$\Phi_n = D_n + T^{c_n}B_n\,,$$
  where $D_n =\diag(T^{\epsilon_{n1}},\dots,T^{\epsilon_{nk}})$ with $\epsilon_{ni}>0$ for all $n,i$, such that for each $i$ we have $\sum_n \epsilon_{ni}<\infty$, where $c_n \geq \max_{1\leq i\leq k}\epsilon_{ni}$ for each $n$, and $B_n$ is a strictly upper triangular matrix with coefficients in $\Lambda_{\geq 0}$. Then the direct limit of
  $$\Lambda_{\geq0}^k\xrightarrow{\Phi_1}\Lambda_{\geq0}^k\xrightarrow{\Phi_2}\Lambda_{\geq0}^k\to\dots$$
  is isomorphic to $\Lambda_{>0}^k$, and in particular it is complete.
\end{lemma}
\begin{proof}
  The conditions on $D_n$ and $B_n$ imply that we can perform Gauss elimination on $\Phi_n$ with the result being $D_n$, which means that there is an invertible matrix $C_n$ such that $C_n\Phi_n=D_n$. This amounts to a change of basis in each copy of $\Lambda_{\geq0}^k$, such that the corresponding direct system is now
  $$\Lambda_{\geq0}^k\xrightarrow{D_1}\Lambda_{\geq0}^k\xrightarrow{D_2}\Lambda_{\geq0}^k\to\dots$$
  The basis changes yield an isomorphism between the two direct systems, and thus an isomorphism between their direct limits. It therefore remains to compute the limit of the latter system. Since each $D_n$ is diagonal, the system splits into the direct sum of the systems
  $$\Lambda_{\geq0}\xrightarrow{T^{\epsilon_{1i}}}\Lambda_{\geq0}\xrightarrow{T^{\epsilon_{2i}}}\Lambda_{\geq0}\to\dots\,,$$
  whose limit can be shown to be isomorphic to $\Lambda_{>0}$, using the remaining assumptions on the $\epsilon_{ni}$.
\end{proof}
\noindent For the sake of brevity, we will only sketch a proof of Lemma \ref{lem:CF_minus_complete}. The idea is that the sets $\cP^{\sim j}_\downarrow(H_n)$ stabilize for $n$ large enough, that is the orbits in them are obtained by perturbing the same set of Reeb circles. Moreover, we can order the elements of $\cP^{\sim j}_-(H_n)$ by decreasing action, which identifies $CF^{\sim j}_-(H_n) = \Lambda_{\geq0}^k$ for $k=|\cP^{\sim j}_-(H_n)|$. Let us outline an argument which shows that $\Phi_n \fc CF^{\sim j}_-(H_n) \to CF^{\sim j}_-(H_{n+1})$ is of the form specified in Lemma \ref{lem:direct_limit_continuations_complete}.

Since $H_n|_K,H_{n+1}|_K$ are negative multiples of the same Morse function, the continuation map between their critical points connects each one only to itself. Next, there are no continuation trajectories from $\cP^{\sim j}_L(H_n)$ to $\cP^{\sim j}_\downarrow(H_{n+1})$ due to action. The continuation trajectories from $\cP^{\sim j}_\downarrow(H_n)$ to $\cP^{\sim j}_L(H_{n+1})$ all have topological energy which is at least the minimal period of a Reeb orbit on $\Sigma$, perhaps minus a small correction. Finally, as we have indicated, there is a natural bijection $\cP^{\sim j}_\downarrow(H_n)\cong\cP^{\sim j}_\downarrow(H_{n+1})$. The matrix element of $\Phi_n$ counting continuation trajectories connecting two orbits which correspond to one another by this bijection can be shown, using action window arguments, such as \cite[Theorem 2.1]{oancea2004surveyfloertheory} to have coefficient $\pm1$ with weight $T^\delta$ for some small $\delta$. The rest of the matrix elements have weights $T^\mu$ with $\mu$ at least the smallest gap in the period spectrum of $\Sigma$ for a bounded set of Conley--Zehnder indices, and thus $\mu$ is at least some positive constant.

This shows that $\Phi_n$ is indeed as in the lemma, and thus $\varinjlim_n CF^{\sim j}_-(H_n)$ is complete. This finishes the proof of Lemma \ref{lem:CF_minus_complete}.

\subsubsection{The kernel of $\res^M_K$}\label{sss:pf_ker_res_M_K}

Here we prove item (ii) of Theorem \ref{thm:Liouville_dom_idx_bdd_SH}, that is that
   \[ \ker \left(H^*(M;\Lambda) \to H^*(K;\Lambda) \right) \subset \ker \res^M_{K}.\]

We retain the acceleration datum $(H_n)_n$ for $K$ specified in the previous subsection. We will compute the chain-level restriction map $\res^M_K$ using a suitable acceleration datum $(L_n)_n$  for $M$, and monotone homotopies from $L_n$ to  $H_n$. The Hamiltonians $L_n$ are assumed to satisfy the following:
\begin{enumerate}
	\item There is a $C^2$-small everywhere negative Morse function $F$ such that $L_n=s_n F$, where $s_n\to0$ a decreasing positive sequence;
	\item The values of $F$ on $K$ are strictly smaller than those on $K^c$;
	\item $\Sigma$ is a regular level set of $F$, and $dF$ and induces the outward coorientation;
	\item By the $C^2$-smallness assumption, the only $1$-periodic orbits of $L_n$ are the critical points of $F$;
	\item Their levels are in relation with those of $H_n$ as seen in Figure \ref{fig:HamiltonianGreyNContinuation}.
\end{enumerate}

\begin{figure}[h]
	\centering
	\includegraphics[scale=1]{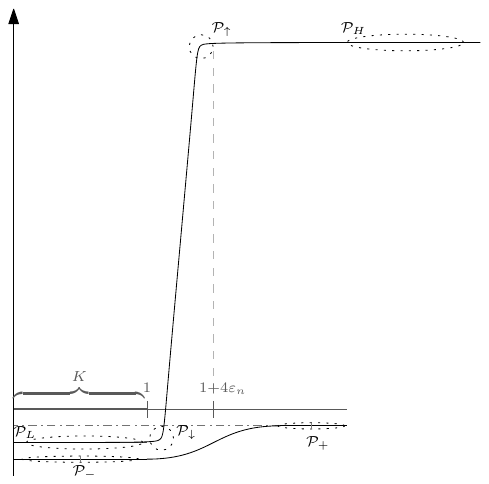}
	\caption{\small{The Hamiltonians $L_n$ and $H_n$ with the distinguished sets of 1-periodic orbits
	}}
	\label{fig:HamiltonianGreyNContinuation}
\end{figure}
We denote the critical points of $L_n$ that lie in $K$ by $\cP_-(L_n)$, while $\cP_+(L_n)$ stands for the set of critical points in $M\setminus K$. Fix $j \in \Z$ and define $CF^{\sim j}_+(L_n)\subset CF^{\sim j}(L_n)$ to be the subcomplex generated by $\cP^{\sim j}_+(L_n)$ and let $CF_-^{\sim j}(L_n) := CF^{\sim j}(L_n)/CF^{\sim j}_+(L_n)$ be the quotient complex. Note that by our assumption (v), the actions $L_n(\cP_+(L_n))$ are strictly larger than the actions $\cA_{H_n}(\cP_-(H_n))$, which means that the continuation map from $L_n$ to $H_n$ maps $CF^{\sim j}_+(L_n)$ into $CF^{\sim }_+(H_n)$. Arguing similarly to the end of Section \ref{sss:SH_K_equals_SH_completion_K}  we obtain the following homotopy-commutative diagram
\begin{equation}\label{eqn:CM_and_SC_diagram}
	\begin{gathered}
		\xymatrix{
		  0 \ar[r] &  \widehat{\operatorname*{tel}\limits_n}\,CF^{\sim j}_+(L_n) \ar[d] \ar[r]^-{i} &  	\widehat{\operatorname*{tel}\limits_n}\,CF^{\sim j}(L_n) \ar[d] \ar[r]^-{p} &  \widehat{\operatorname*{tel}\limits_n}\,CF_-^{\sim j}(L_n)  \ar[d] \ar[r] & 0 \\
		  0 \ar[r] &  \widehat{\operatorname*{tel}\limits_n}\,CF^{\sim j}_+(H_n)  \ar[r]^-{i} &  \widehat{\operatorname*{tel}\limits_n}\,CF^{\sim j}(H_n) \ar[r]^-{p} &  \widehat{\operatorname*{tel}\limits_n}\,CF_-^{\sim j}(H_n) \ar[r] & 0
		}
	\end{gathered}
\end{equation}
Let us identify the relevant complexes, homologies, and maps. Let us start with the map $p$ in the top row. Recall the calculation in \cite{varolgunes2018mayer}: the Floer differential on $CF^*(L_n)$ is simply the Morse differential of $F$ with suitable weights, which disappear in the limit, that is $\wh \varinjlim_n CF^*(L_n) \cong CM^*(F) \otimes \Lambda_{>0}$, where the latter complex is given the usual unweighted Morse differential, and thus computes $H^*(M)\otimes\Lambda_{>0}$. Similarly, $\wh \varinjlim_n CF^*_-(L_n) \cong CM^*(F|_K)\otimes\Lambda_{>0}$, again with the usual Morse differential. This latter complex computes $H^*(K)\otimes\Lambda_{>0}$.

We have a morphism of direct systems from $(CF^*(L_n))_n$ to $(CF^*_-(L_n))_n$ given by termwise quotient maps. The induced morphism in the limit is, using the identifications we have just described, $H^*(M)\otimes\Lambda_{>0} \to H^*(K)\otimes\Lambda_{>0}$, the restriction map on singular cohomology. This morphism of direct systems fits into a commutative square
$$\xymatrix{\wh\tel_n CF^*(L_n) \ar[r]^{p} \ar[d] & \wh\tel_n CF^*_-(L_n) \ar[d] \\ \wh\varinjlim_n CF^*(L_n) \ar[r]& \wh\varinjlim_n CF^*_-(L_n)}$$
where the vertical maps are the canonical quasi-isomorphisms as in \cite[Lemma 2.3.7]{varolgunes2018mayer}. It follows that the map on $j$-th cohomology, induced by $p$ in the top row of \eqref{eqn:CM_and_SC_diagram}, is the singular restriction map.

Similarly to the considerations of Section \ref{sss:SH_K_equals_SH_completion_K}, the complex $\wh\tel_n\,CF^{\sim j}_+(H_n)$ is acyclic. Therefore, taking the $j$-th cohomology of the right square of \eqref{eqn:CM_and_SC_diagram} we finally arrive at the commutative diagram
$$\xymatrix{SH^j(M) = H^j(M)\otimes\Lambda_{>0} \ar[r]^-{p_*} \ar[d]_{\res^M_K} & H^j(K)\otimes\Lambda_{>0}\ar[d] \\ SH^j(K) \ar[r]^-{\cong} & \cdot}$$
where $p_*$ has just been shown to equal the restriction on singular cohomology. It follows from this diagram that $\ker p^* \subset \ker \res^M_K$, and the assertion (ii) of Theorem \ref{thm:Liouville_dom_idx_bdd_SH} follows from this upon tensoring with $\Lambda$. This completes the proof of Theorem \ref{thm:Liouville_dom_idx_bdd_SH}.

\subsection{Proof of Theorem \ref{thm:domain_no_contr_orbits_bdry_SH_equal_H}}\label{ss:pf_SH_no_contr_orbits}

Recall that the theorem states the following: assume that $K$ is a region with boundary $\Sigma$ whose neighborhood is the image of a smooth embedding $(-\epsilon,\epsilon)\times \Sigma \hookrightarrow M$ extending $\id \fc \Sigma = \{0\}\times \Sigma \to M$, such that no $\{\rho\}\times \Sigma$ carries closed characteristics contractible in $M$; then $SH^*(K;\Lambda) = H^*(K;\Lambda)$ and $\res^M_K \fc SH^*(M;\Lambda) \to SH^*(K;\Lambda)$ coincides with the restriction on singular cohomology $H^*(M;\Lambda)\to H^*(K;\Lambda)$.

We pick acceleration data $(H_n)_n$, $(L_n)_n$ for $K, M$, respectively, as in Section \ref{ss:pf_thm_SH_Liouv_domains}, only this time $r$ is the first coordinate of the aforementioned embedding $(-\epsilon,\epsilon)\times\Sigma\hookrightarrow M$. In addition, we require that $L_n$ and $H_n$ coincide on $K$. The absence of closed contractible characteristics near $\Sigma$ means that $\cP^\circ(H_n) = \cP_{\mathrm L}(H_n)\cup\cP_{\mathrm H}(H_n)$ and $\cP_{\downarrow,\uparrow}(H_n) = \varnothing$, where we use the notations from Section \ref{ss:pf_thm_SH_Liouv_domains}. Note as well that $\cP_-(L_n) = \cP_-(H_n)$. See Figure \ref{fig:HamiltonianGreyContinuationNoContractible}.

\begin{figure}[h]
	\centering
	\includegraphics[scale=1]{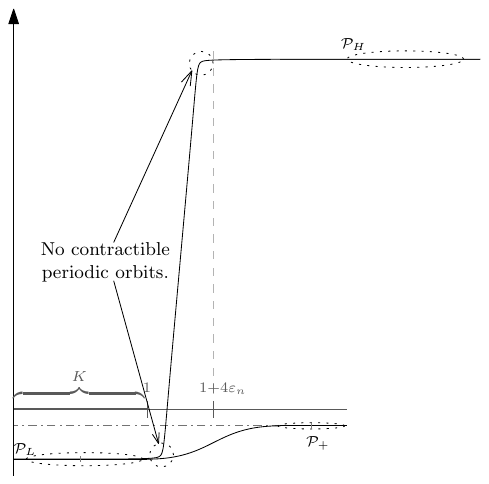}
	\caption{\small{The Hamiltonians $L_n$ and $H_n$ with the distinguished sets of 1-periodic orbits
	}}
	\label{fig:HamiltonianGreyContinuationNoContractible}
\end{figure}

As in Section \ref{sss:pf_ker_res_M_K}, we obtain two short exact sequences with morphisms between them:
\[
\xymatrix{
  0 \ar[r] &  \widehat{\operatorname*{tel}\limits_n}\,CF_+^*(L_n) \ar[d] \ar[r]^-{i} &  \widehat{\operatorname*{tel}\limits_n}\,CF^*(L_n) \ar[d] \ar[r]^-{p} &  \widehat{\operatorname*{tel}\limits_n}\,CF_-^*(L_n)  \ar[d]^{\Phi^H_L} \ar[r] & 0 \\
  0 \ar[r] &  \widehat{\operatorname*{tel}\limits_n}\,CF_+^*(H_n)  \ar[r]^-{i} &  \widehat{\operatorname*{tel}\limits_n}\,CF^*(H_n) \ar[r]^-{p} &  \widehat{\operatorname*{tel}\limits_n}\,CF_-^*(H_n) \ar[r] & 0 \, .
}
\]
Note that we do not use truncated complexes here due to the absence of nonconstant contractible orbits. As before, $\widehat{\operatorname*{tel}_n}\,CF_+^*(H_n)$ is acyclic, and the map $p$ in the top row computes the restriction on singular cohomology $H^*(M)\otimes \Lambda_{>0}\to H^*(K)\otimes \Lambda_{>0}$. What is different is the map $\Phi_L^H$ in this diagram, which we will now show to be an isomorphism. Fix $n$ and consider the map $\Phi_n^-\fc CF_-^*(L_n) \to CF_-^*(H_n)$. Since $L_n\equiv H_n$ on $K$, the constant Floer solution sitting at a generator $x \in \cP_{\mathrm L}(L_n)$ contributes to the continuation map with coefficient $1$ and is the only energy zero connecting trajectory, therefore $\Phi_n^-(x) = x+$ terms with positive powers of $T$. This means that the reduction of $\Phi_n^-$ to the residue field of $\Lambda_{\geq0}$ is just the identity, in particular invertible, which in turn implies that $\Phi_n^-$ is invertible. Therefore the $2$-ray consisting of the complexes $CF_-^*(L_n),CF_-^*(H_n)$ has the property that the maps in the finite direction are all isomorphisms. It is not hard to show that the induced map $\tel_n CF_-^*(L_n) \to \tel_n CF_-^*(H_n)$ is then an isomorphism, therefore so is the map on completions $\Phi_L^H\fc\wh{\tel}_n CF_-^*(L_n) \to \wh{\tel}_n CF_-^*(H_n)$.

Taking the homology of the right square, we arrive at the following commutative diagram:
$$\xymatrix{SH^*(M) = H^*(M)\otimes\Lambda_{>0} \ar[r]^-{p_*} \ar[d]^{\res^M_K} & H^*(K)\otimes\Lambda_{>0} \ar[d]^{\cong} \\ SH^*(K) \ar[r]^-{\cong} & \cdot}$$
Composing the right arrow and the inverse of the bottom arrow, we conclude that $SH^*(K)\cong H^*(K)\otimes\Lambda_{>0}$ and that under this isomorphism the restriction maps on singular and symplectic cohomologies coincide. In particular $\ker \res^M_K = \ker p_*$. This finishes the proof of Theorem \ref{thm:domain_no_contr_orbits_bdry_SH_equal_H}.

\subsection{Proof of Theorem \ref{thm:heavy_implies_SH_heavy}}\label{ss:heavy_implies_SH_heavy}

Recall that the theorem states that if $(M,\omega)$ is symplectically aspherical, $K\subset M$ is a compact heavy set, and there is a sequence of contact-type regions $W_i$ with incompressible index-bounded boundary, all containing $K$ in the interior, such that $K = \bigcap_i W_i$, then for all $i$, $[\Vol] \in \ker \res^M_{\ol{W_i^c}}$, and in particular $[\Vol] \in \tau(K) = \bigcap_i\ker \res^M_{\ol{W_i^c}}$, therefore $K$ is SH-heavy.

For the rest of the proof we let $W$ be one of the $W_i$. Varolg\"une\c s proves in \cite{varolgunes2018mayer} that $SH^*(M) = H^*(M)\otimes\Lambda_{>0}$. We will prove the following
\begin{claim}\label{claim:T_mu_Vol_in_ker_res} There exists $\mu > 0$ such that
$$T^{\mu}[\Vol] \in \ker \big(SH^*(M) \to SH^*(\ol{W^c})\big)\,.$$
\end{claim}
\noindent That $[\Vol] \in \ker \res^M_{\ol{W^c}}$ follows from this upon tensoring with $\Lambda$.

Let us recall what it means for a set to be heavy. To this end we note that, since $M$ is symplectically aspherical, given any ground field $\F$, the usual unweighted Floer complex $CF^*_{\uw}(H;\F)$ of a nondegenerate Hamiltonian $H$ on $M$ can be defined as the $\F$-vector space with basis $\cP^\circ(H)$, and that it can be given a $\Z$-grading. The Floer differential $d_{\uw}$ is given by the usual formula, see \eqref{eqn:Floer_diff_unweighted}. We let $HF_{\uw}^*(H;\F)$ be the cohomology of $(CF^*_{\uw}(H;\F),d_{\uw})$. Correspondingly, we have the $\Z$-graded quantum cohomology $QH^*(M;\F)$, which is isomorphic to the singular cohomology $H^*(M;\F)$, in this case both additively, and as an algebra. We have the PSS isomorphism $\PSS \fc HF^*_{\text{\rm uw}}(H;\F) \to QH^*(M;\F) = H^*(M;\F)$, see\cite{PSS}.

Since the differential increases the action, for each $a \in \R$, the subspace $CF_{\text{\rm uw},\geq a}^*(H;\F)$ generated by $x\in\cP^\circ(H)$ with $\cA_H(x)\geq a$ is a subcomplex. Let $HF^*_{\text{\rm uw},\geq a}(H;\F)$ be the cohomology of $CF^*_{\text{\rm uw},\geq a}(H;\F)$ and let $j_a^* \fc HF^*_{\text{\rm uw},\geq a}(H;\F) \to HF^*_{\text{\rm uw}}(H;\F)$ be the morphism induced by the inclusion. We have the \emph{cohomological spectral invariants}
$$c^\vee(A,H) = \sup\{a \in \R\,|\,A \in \im (\PSS\circ j^*_a)\}\quad\text{for }A\in QH^*(M;\F)\,.$$
The corresponding partial symplectic quasi-state $\zeta \fc C^\infty(M) \to \R$ is given by
$$\zeta(H) = \lim_{k\to\infty}\frac{c^\vee([\Vol],kH)}{k}\,,$$
where $[\Vol] \in H^{2n}(M;\F)$ is the volume class. The original definition in \cite{entov2009rigid} used homological spectral invariants $c([M],\cdot)$ relative to the fundamental class $[M] \in H_{2n}(M;\F)$, but $c([M],\cdot)=c^\vee([\Vol],\cdot)$ thanks to the duality formula, see for instance \cite[Section 4.2]{leclercq2018spectral}. For $K$ to be heavy means that for each $F \in C^\infty(M)$ we have
$$\zeta(F) \geq \min_K F\,.$$

For us, the crucial consequence of the assumption that $K$ is heavy is the following
\begin{lemma}\label{lem:heavy_implies_high_spectral_invts}
  For any $L>0$ there exists $F \in C^\infty(M)$ such that $F|_{\ol{W^c}} < 0$ and such that $c^\vee([\Vol],F) > L$.
\end{lemma}
\begin{proof}
  Let $F_0 \in C^\infty(M)$ satisfy $F_0|_{\ol{W^c}} < 0$ and $F_0|_K > 0$. Since $K$ is heavy, we have
  $$0<\min_K F_0 \leq \zeta(F_0) = \lim_{k\to\infty}\frac{c^\vee([\Vol],kF_0)}{k}\,.$$
  In particular there is $k_0$ such that $c^\vee([\Vol],k_0F_0) > L$. Now put $F=k_0F_0$.
\end{proof}

Let us identify a neighborhood of the boundary $\Sigma = \partial W$ with the piece $\Sigma\times(1-\epsilon,1+\epsilon)_r$ of the symplectization of $\Sigma$, so that $r\partial_r$ is the Liouville vector field.

For real numbers $e,E$ such that $e<0$, $E>0$, let us call a nondegenerate Hamiltonian $H$ on $M$ \tb{$(E,e)$-admissible} if it is a sufficiently small perturbation of a smooth function $\wt H \fc M \to \R$ which satisfies: $\wt H|_{\ol{W^c}}\equiv e$, $\wt H|_{W\setminus (\Sigma\times(1-\epsilon,1])}\equiv E$, $\wt H|_{\Sigma\times(1-\epsilon,1]} = f\circ r$, where $f \fc [1-\epsilon,1] \to \R$ is a smooth function such that $f(1-\epsilon)=E$, $f(1)=e$, such that $f$ is strictly concave on $[1-\epsilon,1-2\epsilon/3]$, strictly convex on $[1-\epsilon/3,1]$, and $f(r)=ar+b$ for some constants $a,b$ for $r\in[1-2\epsilon/3,1-\epsilon/3]$, where the slope $a$ is noncharacteristic, that is $|a|$ is not the length of a contractible Reeb orbit of $\Sigma$. By ``sufficiently small'' here we mean the following:
\begin{enumerate}
  \item on $\ol{W^c}$, $H$ is $e$ plus a $C^2$-small Morse function, so that its only contractible $1$-periodic orbits there are the critical points of the Morse function;
  \item on $W\setminus(\Sigma\times(1-\epsilon,1])$, $H$ is $E$ plus a $C^2$-small Morse function, so that its only contractible $1$-periodic orbits there are the critical points of the Morse function;
  \item on $\Sigma\times[1-2\epsilon/3,1-\epsilon/3]$, $H=\wt H$, so that in particular $H$ has no contractible $1$-periodic orbits there;
  \item on $\Sigma\times[1-\epsilon,1-2\epsilon/3]$, each contractible $1$-periodic orbit $x$ of $H$ of Conley--Zehnder index $2n$ is close enough to an orbit of $\wt H$ of the form $\gamma\times\{r\}$, where $\gamma$ is a suitable reparametrization (in the negative direction!) of a closed Reeb orbit of $\Sigma$; in particular we require the Conley--Zenhder index of this reparametrized Reeb orbit to be universally bounded, which implies that its period is universally bounded, and we require $\cA_H(x)$ to be close to the action $\cA_{\wt H}(\gamma\times\{r\})=f(r)-rf'(r)$ up to an error of $1$;
  \item on $\Sigma\times[1-\epsilon/3,1]$, the same is required of $H$ as in the previous item.
\end{enumerate}
\noindent We invite the reader to revisit the arguments in Section \ref{ss:pf_thm_SH_Liouv_domains} in order to better understand the logic here. We refer to the contractible $1$-periodic orbits of $H$ in $W\setminus(\Sigma\times[1-2\epsilon/3,1])$ as the ``upper orbits,'' while to the rest as the ``lower orbits.''

The point of admissible Hamiltonians is as follows.
\begin{lemma}
  There exists $L>0$ so that whenever $H$ is an $(E,e)$-admissible Hamiltonian, then:
  \begin{enumerate}
    \item the actions of its upper orbits of index $2n$ in $\Sigma\times[1-\epsilon,1-2\epsilon/3]$ are at least $E-1$;
    \item the actions of its lower orbits of index $2n$ in $\Sigma\times[1-2\epsilon/3,1]$ are  at most $L$.
  \end{enumerate}
\end{lemma}
\begin{proof}
  Item (i) is a consequence of assumptions (ii) and (iv) in the definition of admissibility, together with the fact that the action of an orbit of $H$ of the form $\gamma\times\{r\}$ is $f(r)-rf'(r)$, the $y$-intercept of the tangent to the graph of $f$ at $r$. This number can easily be seen to be at least $E$, and our assumption on the $1$-closeness of $H$- and $\wt H$-actions, as in assumption (iv).

  Item (ii) is similarly a consequence of assumptions (i) and (iii). The action bound $L$ comes from the universal bound on the lengths of Reeb orbits which correspond to orbits of $H$ of Conley--Zehnder index $2n$.
\end{proof}

Fix $L$ as in the lemma, and fix $F$ as in Lemma \ref{lem:heavy_implies_high_spectral_invts} for this $L$. Then there exists an acceleration datum $(H_i)_{i\geq 0}$ for $\ol {W^c}$ such that for all $i$ we have:
\begin{itemize}
  \item $H_i \geq F$;
  \item each $H_i$ is $(E_i,e_i)$-admissible, where $e_i\to0$ and $E_i \to \infty$;
  \item all the upper orbits of $H_i$ of index $2n$ have actions $>L$;
  \item There exists $c>0$ such that the minimal action of an upper orbit of $H_{i+1}$ of index $2n$ is at least $c+$ the maximal action of an upper orbit of $H_i$ of index $2n$.
\end{itemize}
\noindent See Figure \ref{fig:HeavySHheavyActions} for an illustration. Note that as a consequence of admissibility, all the lower orbits of $H_i$ of index $2n$ have actions $<L$.

\begin{figure}[h]
	\centering
	\includegraphics[scale=1]{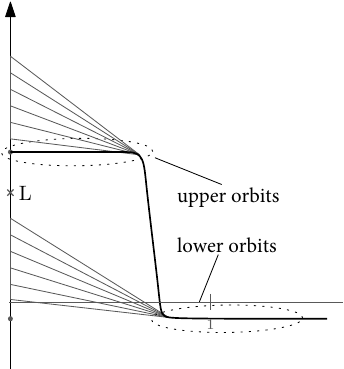}
	\caption{\small{The action separation between the lower and upper orbits is demonstrated on the on the Hamiltonians $H_n$. The value $L$ is marked. The tangents whose $y$ intercepts are the actions are depicted in gray, as well as dots representing the actions of the Morse critical points.
	}}
	\label{fig:HeavySHheavyActions}
\end{figure}

This acceleration datum yields a $1$-ray of Floer complexes and continuation maps
$$CF^*(H_0)\xrightarrow{\Phi_{H_0}^{H_1}} CF^*(H_1)\to\dots$$
\begin{lemma}\label{lem:lin_comb_upper_orbits_dies_completed_limit}
  Let $y \in CF^{2n}(H_0)$ be a linear combination of the upper orbits of $H_0$. Then it is killed by the composition
  $$\textstyle CF^*(H_0) \xrightarrow{\Phi_{H_0}} \varinjlim_i CF^*(H_i) \xrightarrow{\wh\cdot} \wh\varinjlim_i CF^*(H_i)\,,$$
  where $\Phi_{H_0}$ is the natural map into the direct limit.
\end{lemma}
\begin{proof}
  The map $\Phi_{H_0}$ factors as follows:
  $$CF^*(H_0) \xrightarrow{\Phi_{H_{j-1}}^{H_j}\circ\dots\circ\Phi_{H_0}^{H_1}} CF^*(H_j) \xrightarrow{\Phi_{H_j}}\varinjlim_i CF^*(H_i)\,,$$
  where $\Phi_{H_j}$ is again the natural map into the direct limit. By the action separation property of the acceleration datum $(H_i)_i$, $\Phi_{H_i}^{H_{i+1}} \fc CF^{2n}(H_i) \to CF^{2n}(H_{i+1})$ maps any linear combination $z$ of upper orbits of $H_i$ to a linear combination of upper orbits of $H_{i+1}$, since it is action-increasing. Thus $\Phi_{H_i}^{H_{i+1}}(z) \in T^c\cdot CF^{2n}(H_{i+1})$, where $c$ has been chosen in the above list of properties of $(H_i)_i$, and by induction we have
  $$(\Phi_{H_{j-1}}^{H_j}\circ\dots\circ\Phi_{H_0}^{H_1})(y) \in T^{jc}\cdot CF^*(H_j)\,,$$
  whence
  $$\Phi_{H_0}(y) = \Phi_{H_j}\big((\Phi_{H_{j-1}}^{H_j}\circ\dots\circ\Phi_{H_0}^{H_1})(y)\big) \in T^{jc}\varinjlim_i CF^*(H_i)$$
  for all $j$, that is
  $$\Phi_{H_0}(y) \in \bigcap_{\lambda \geq 0}T^\lambda\varinjlim_i CF^*(H_i) \subset \ker\wh\cdot\,.$$
\end{proof}

We now fix a $C^2$-small negative Morse function $G \fc M \to \R$, let $(s_i)_{i\geq0}$ be a strictly decreasing sequence of positive numbers $<1$ converging to zero, and put $G_i=s_iG$. Then $(G_i)_{i\geq0}$ is an acceleration datum for $M$. The weighted Floer complexes $CF^*(G_i)$ are all isomorphic to the Morse complex $CM^*(G)\otimes\Lambda_{\geq0}$ with the differential weighted by suitable powers of $T$, which depend on $i$, while the continuation map $CF^*(G_i)\to CF^*(G_{i+1})$ is given by $\Crit G\ni p \mapsto T^{G_{i+1}(p)-G_i(p)}p$. It follows that $\varinjlim_i CF^*(G_i)$ is canonically isomorphic to $CM^*(G)\otimes \Lambda_{>0}$ with the usual unweighted Morse differential; in particular it is a complete $\Lambda_{\geq0}$-module. The natural map $CM^*(G)\otimes\Lambda_{\geq0}=CF^*(G_j) \to \varinjlim_i CF^*(G_i) = CM^*(G)\otimes\Lambda_{>0}$ is given by $\Crit G\ni p\mapsto T^{-G_j(p)}p$.

In particular $SH^*(M) = H\big(CM^*(G)\otimes\Lambda_{>0}\big)=HM^*(G)\otimes\Lambda_{>0}$ and if $q$ is a maximum of $G$, then $[\Vol]\otimes T^{-G_0(q)} \in SH^*(M)$ is represented by $T^{-G_0(q)}q \in CM^*(G)\otimes\Lambda_{>0}$ and is the image of $q \in CF^*(G_0)$ under the natural map $CF^*(G_0) \to \varinjlim_i CF^*(G_i)$. Note that we can assume that $G_i \leq H_i$ for all $i$, therefore as in Section \ref{ss:Floer_data_rays} there exists a filling between the acceleration data $(G_i)_i,(H_i)_i$ and in particular we have the corresponding continuation maps $\Phi_{G_i}^{H_i} \fc CF^*(G_i) \to CF^*(H_i)$.
\begin{lemma}\label{lem:exists_mu_cohomologous}
  Let $q \in CF^{2n}(G_0)$ be a maximum. Then there is $\mu \geq -G_0(q)$ such that $T^{\mu+G_0(q)}\Phi_{G_0}^{H_0}(q)$ is cohomologous in $CF^*(H_0)$ to a linear combination of upper orbits.
\end{lemma}
\noindent Assuming this for a moment, let us proceed to
\begin{proof}[Proof of Claim \ref{claim:T_mu_Vol_in_ker_res}]
We have the following commutative diagram:
\begin{equation}\label{eqn:diagram_G_0_H_0}
\xymatrix{CF^*(G_0) \ar[r] \ar[d]^{\Phi_{G_0}^{H_0}} & \wh{\tel}_i CF^*(G_i) \ar[r] \ar[d] & \wh{\varinjlim_i} CF^*(G_i) = \varinjlim_i CF^*(G_i) \\
            CF^*(H_0) \ar[r] & \wh{\tel}_i CF^*(H_i)\ar[r] & \wh{\varinjlim_i} CF^*(H_i)}
\end{equation}
The middle vertical arrow comes from the $2$-ray of Floer complexes coming from the aforementioned filling. The top left arrow is the inclusion of $CF^*(G_0)$ into the second direct sum in $\tel_i CF^*(G_i) = \bigoplus_i CF^*(G_i)[1] \oplus \bigoplus_i CF^*(G_i)$, followed by completion, and similarly for the bottom left arrow. It follows that the square commutes. The top right arrow is the completion of the composition of the projection of $\tel_i CF^*(G_i)$ onto $\bigoplus_i CF^*(G_i)$ followed by the natural map into the direct limit. Varolg\"une\c s proves in \cite{varolgunes2018mayer} that this arrow is a quasi-isomorphism. The same considerations hold for the bottom right arrow. Note that the composition of the top arrows is the natural map into the direct limit, and the same is true for the composition of the bottom arrows, except the completion is then tagged on.

Fix $\mu$ as in Lemma \ref{lem:exists_mu_cohomologous}. We will show that $T^{\mu}[\Vol]$ is killed by $SH^*(M) \to SH^*(\ol{W^c})$. The element $T^{\mu}[\Vol] \in SH^{2n}(M) = H^{2n}(M)\otimes \Lambda_{>0}=H^{2n}\big(\wh\tel_i CF^*(G_i)\big)$ is represented by the image of $T^{\mu+G_0(q)}q\in CF^{2n}(G_0)$ under the top left arrow. It follows that the image of $T^{\mu}[\Vol]$ by $SH^*(M) \to SH^*(\ol{W^c})$ is represented by the image of $\Phi_{G_0}^{H_0}(T^{\mu+G_0(q)}q) \in CF^{2n}(H_0)$ in $\wh{\varinjlim_i} CF^*(H_i)$ by the map which is the composition in the top row of \eqref{eqn:diagram_G_0_H_0}. By Lemma \ref{lem:exists_mu_cohomologous}, $\Phi_{G_0}^{H_0}(T^{\mu+G_0(q)}q) = T^{\mu+G_0(q)}\Phi_{G_0}^{H_0}(q)$ is cohomologous to a linear combination of upper orbits of $H_0$, and thus by Lemma \ref{lem:lin_comb_upper_orbits_dies_completed_limit}, this linear combination is killed by the map $CF^*(H_0) \to \wh{\varinjlim_i} CF^*(H_i)$. It follows that the image of $T^{\mu+G_0(q)}\Phi_{G_0}^{H_0}(q)$ in $\wh{\tel}_i CF^*(H_i)$ is cohomologous to zero, and thus we have our claim.
\end{proof}

\begin{proof}[Proof of Lemma \ref{lem:exists_mu_cohomologous}]
  Consider the following commutative diagram:
  $$\xymatrix{ & & CF^*_{\text{\rm uw}}(G_0;\F)\ar[d] \ar[rd] & \\
               CF_{\text{\rm uw},\geq0}^*(G_0;\Lambda)\ar[rr] \ar[dd]\ar[rd] & & CF^*_{\text{\rm uw}}(G_0;\Lambda)\ar[rd]\ar[dd]|\hole & CF^*_{\text{\rm uw}}(H_0;\F) \ar[d]\\
               & CF_{\text{\rm uw},\geq0}^*(H_0;\Lambda) \ar[dd] \ar[rr]& & CF^*_{\text{\rm uw}}(H_0;\Lambda)\ar[dd]\\
               CF^*(G_0) \ar[rr]|\hole \ar[rd] & & CF^*(G_0)\otimes\Lambda \ar[rd]& \\
               & CF^*(H_0)\ar[rr] & & CF^*(H_0)\otimes\Lambda}$$
  Here for a nondegenerate Hamiltonian $E$ we have a valuation on $CF^*_{\text{\rm uw}}(E;\Lambda)$ given by $T^\lambda x\mapsto \lambda+\cA_E(x)$ and $CF^*_{\text{\rm uw},\geq0}(E;\Lambda)$ stands for the $\Lambda_{\geq0}$-submodule of elements with nonnegative valuation. The oblique arrows are continuation maps in respective Floer theories (unweighted or $T$-weighted), the upper vertical arrows are the embeddings induced by the field extension $\F \subset \Lambda$, the horizontal arrows are inclusions, while the rest of the vertical arrows are the isomorphism $\psi$ from Proposition \ref{prop:Weighted_vs_classical_CF} in Section \ref{sss:class_weighted_HF}. For instance $\psi \fc CF^*_{\text{\rm uw}}(G_0;\Lambda) \to CF^*(G_0)\otimes \Lambda$ is defined by $\psi(x) =x\otimes T^{\cA_{G_0}(x)}$, and similarly for $H_0$. It is easy to check that $\psi$ maps $CF^*_{\text{\rm uw},\geq0}(G_0;\Lambda)$ isomorphically onto $CF^*(G_0)\subset CF^*(G_0)\otimes \Lambda$ and same for $H_0$.

  Let now $q \in CF^{2n}_{\text{\rm uw}}(G_0;\F)$ be a maximum, that is a representative of the volume class. Since continuation maps commute on cohomology with the PSS isomorphisms, it follows that $\wt y:=\Phi(q) \in CF^{2n}_{\text{\rm uw}}(H_0;\F)$ also represents the volume class, where $\Phi \fc CF^*_{\text{\rm uw}}(G_0;\F) \to CF^*_{\text{\rm uw}}(H_0;\F)$ is the unweighted continuation map. On the other hand, Lemma \ref{lem:heavy_implies_high_spectral_invts} and the assumption $H_0\geq F$ imply that $c^\vee([\Vol],H_0) > L$, since spectral invariants are monotone with respect to the Hamiltonian, therefore there is a representative of the volume class $y \in CF^{2n}_{\text{\rm uw}}(H_0;\F)$ consisting of orbits of action $> L$, and by our choice of acceleration datum this forces $y$ to be a linear combination of the upper orbits of $H_0$. It follows that there is $b\in CF^{2n-1}_{\text{\rm uw}}(H_0;\F)$ such that $\wt y - y = d_{\text{\rm uw}}b$. We can now look at this equation in $CF^*_{\text{\rm uw}}(H_0;\Lambda)$. The elements $\wt y,b$ do not necessarily lie in $CF^*_{\text{\rm uw},\geq 0}(H_0;\Lambda)$. Let $\mu\geq - G_0(q)>0$ be such that $T^\mu\wt y,T^\mu b \in CF^*_{\text{\rm uw},\geq 0}(H_0;\Lambda)$. Applying $\psi$ and noting that it is $\Lambda_{\geq0}$-linear and commutes with continuation maps, we obtain the following equation in $CF^*(H_0)$:
  $$d(T^\mu\psi(b)) = \psi(T^\mu\Phi(q)) - \psi(T^\mu y) = T^\mu\Phi_{G_0}^{H_0}(\psi(q)) - \psi(T^\mu y)=T^{\mu+G_0(q)}\Phi_{G_0}^{H_0}(q) - \psi(T^\mu y)\,.$$
  This means that $T^{\mu+G_0(q)}\Phi_{G_0}^{H_0}(q)$ is cohomologous in $CF^*(H_0)$ to $\psi(T^\mu y)$, which is a linear combination of the upper orbits of $H_0$, as claimed.
\end{proof}

\section{Centerpoint theorems for IVMs}\label{s:abstr_centerpt_thm_IVMs}

In this section, we first formulate an abstract centerpoint theorem for IVMs, from which we will first deduce Karasev's result, Theorem \ref{thm:top_Tverberg}, as well as our Theorem \ref{thm:symplectic_Tverberg}, which is used in Section \ref{ss:pf_thm_torus_cross} in the proof of the symplectic rigidity result, Theorem \ref{thm:torus_cross}, and eventually Gromov's result, Theorem \ref{thm:Gromov_torus}.
\begin{thm}\label{thm:central_point} Let $Y$ be a compact metric space of covering dimension $d$, let $A$ be an algebra, and let $I \in \cI(A)$ be a graded ideal with $I^{*(d+1)} \neq 0$. For an $A$-IVM $\nu$ on $Y$ put
$$\cX_{I,\nu} = \{Z\subset Y\,|\,Z\text{ compact with }I\subset \nu(Z)\}\,.$$
Then $\bigcap_{Z \in \cX_{I,\nu}}Z \neq \varnothing$.
\end{thm}
\noindent We refer to a point in the above intersection as a \tb{centerpoint} of $\nu$ with respect to $I$.
\begin{proof}
  Assume on the contrary that $\bigcap_{Z\in\cX_{I,\nu}}Z = \varnothing$. Then $(Z^c)_{Z\in\cX_{I,\nu}}$ is an open covering of $Y$, therefore by the assumption on the covering dimension on $Y$ and by Milnor's lemma (see \cite[Lemma 2.4]{Palais_1966_Homotopy_theory_of_infinite_dimensional_manifolds}), this covering admits a finite refinement $\{V_{ij}\}_{i,j}$ where $i=0,\dots,d$ and $V_{ij} \cap V_{ij'} = \varnothing$ if $j\neq j'$. Let $Y_i = \bigcup_j V_{ij}$. For each $i,j$ there is $Z\in\cX_{I,\nu}$ with $V_{ij}\subset Z^c$, therefore by monotonicity $I\subset \nu(Z)\subset \nu(V_{ij}^c)$. Since $Y=V_{ij}^c\cup V_{ij'}^c$ for any $j\neq j'$, by the intersection property we have
  $$\textstyle\nu(Y_i^c) = \nu\Big(\bigcap_j V_{ij}^c\Big) = \bigcap_j \nu(V_{ij}^c)\supset I\,,$$
  and by multiplicativity and normalization we obtain
  $$\textstyle 0\neq I^{*(d+1)} \subset \prod_{i=0}^{d}\nu(Y_i^c) \subset \nu\left(\bigcap_{i=0}^dY_i^c\right) = \nu(\varnothing) = 0\,,$$
  which is a contradiction.
\end{proof}
\noindent This abstract theorem has the following consequence, which will be used for the proof of Gromov's version of the centerpoint theorem \ref{thm:Gromov_central_point}.
\begin{coroll}\label{thm-karas}Under the assumptions of Theorem \ref{thm:central_point}, and under the additional assumption that $A$ is graded Noetherian, a centerpoint $y_0 \in Y$ satisfies $I\not\subset\nu(Y\setminus\{y_0\})$.
\end{coroll}
\begin{proof}
  Otherwise $I\subset \nu(Y\setminus\{y_0\})$, therefore by continuity
  $$I\subset \nu(Y\setminus\{y_0\}) = \nu\bigg(\bigcup_{i\in \N}Y\setminus \ol B_{y_0}(\tfrac 1 i)\bigg) = \bigcup_{i\in \N}\nu(Y\setminus \ol B_{y_0}(\tfrac 1 i))\,,$$
  and by the graded Noetherian property the ascending chain of graded ideals in the union stabilizes, which means that there is $i_0 \in \N$ such that $I \subset \nu(Y\setminus \ol B_{y_0}(\tfrac 1 {i_0}))$. By monotonicity it follows that $I\subset \nu(Y\setminus B_{y_0}(\tfrac 1 {i_0}))$, which, thanks to the fact that $y_0$ is a centerpoint of $\nu$ with respect to $I$, means that $y_0 \in Y\setminus B_{y_0}(\tfrac 1 {i_0})$, which is absurd.
\end{proof}

\noindent We will derive the following from Theorem \ref{thm:central_point}:
\begin{coroll}\label{cor:big_fiber_cont_maps_into_dim_d_metric}
Let $X$ be a compact Hausdorff space, let $A$ be an algebra, and let $\mu$ be an $A$-IVM. Let $I \in \cI(A)$ be a graded ideal such that $I^{*(d+1)} \neq 0$ for some $d \geq 1$. Let $Y$ be a metric space of covering dimension $d$. Then any continuous map $f \fc X \to Y$ has a fiber which intersects all the members of $\cX_{I,\mu}$.
\end{coroll}
\noindent We call such a fiber a \tb{central fiber of $f$ with respect to $I$}. To deduce Corollary \ref{cor:big_fiber_cont_maps_into_dim_d_metric} from Theorem \ref{thm:central_point}, we need the notion of pushforward for IVMs.
\begin{defin}\label{def:pushforward}
  Let $A$ be an algebra, let $f \fc X \to Y$ be a continuous map, and let $\mu$ be an $A$-IVM on $X$. The \tb{pushforward of $\mu$ by $f$} is the set function $f_*\mu$ defined on the open sets of $Y$ by $f_*\mu(V) = \mu(f^{-1}(V))$.
\end{defin}
\begin{rem}\label{rem:pushforward_defined_more_generally}
  \begin{itemize}
    \item The pushforward construction makes sense for any $\cI(A)$-valued function defined on open or closed subsets of $X$, not necessarily an $A$-IVM. We will use this below.
    \item It is easy to see that $f_*\mu$ is an $A$-IVM on $Y$. If $X,Y$ are Hausdorff and $X$ is in addition compact, then the extension of $f_*\mu$ to the compact subsets of $Y$, as described in Remark \ref{rem:regularization_IVM}, coincides with the pushforward of the extension of $\mu$ to the compact subsets of $X$.
  \end{itemize}
\end{rem}
\begin{proof}[Proof of Corollary \ref{cor:big_fiber_cont_maps_into_dim_d_metric}] Without loss of generality assume that $Y$ is compact and that $f$ is onto. Applying Theorem \ref{thm:central_point} to the $A$-IVM $\nu=f_*\mu$, we obtain a point $y_0$ contained in any compact $Y'\subset Y$ with $I \subset \nu(Y')$. If $Z \in \cX_{I,\mu}$, then, since $Z \subset f^{-1}(f(Z))$, we have
$$I \subset \mu(Z) \subset \mu(f^{-1}(f(Z)) = \nu(f(Z))\,,$$
therefore $y_0 \in f(Z)$, as claimed.
\end{proof}

Karasev's topological centerpoint theorem \ref{thm:top_Tverberg} follows from a trick from toric geometry appearing in Karasev's original proof \cite{karasev2014covering}, in combination with the following result, which itself easily follows from Corollary \ref{cor:big_fiber_cont_maps_into_dim_d_metric}.
\begin{thm}\label{thm:cont_maps_CP_n_Y_d_fiber_proj_subspaces}Let $n=p(d+1)$, where $p,d$ are positive integers. Then for any continuous map $g \fc \C P^n \to Y$, where $Y$ is a metric space of covering dimension $d$, there exists $y_0 \in Y$ such that $g^{-1}(y_0)$ intersects all the $pd$-dimensional projective subspaces of $\C P^n$.
\end{thm}
\begin{proof}
  Let $\mu$ be the cohomology IVM on $\C P^n$, let $h \in H^2(\C P^n)$ be a generator and consider the graded ideal $I = \langle h^p\rangle$. Then $I^{\smile(d+1)}\neq 0$. Moreover, if $C\subset \C P^n$ is a $pd$-dimensional complex projective subspace, then, since $h^p$ is Poincar\'e dual to $C$, it follows that for every open neighborhood $U\supset C$ we have $h^p|_{\C P^n \setminus U}=0$, which means that $I\subset \mu(C)$, in particular $C \in \cX_{I,\mu}$, using the notation of Theorem \ref{thm:central_point}. Corollary \ref{cor:big_fiber_cont_maps_into_dim_d_metric} then implies that $g$ has a fiber intersecting all the members of $\cX_{I,\mu}$, and in particular each $pd$-dimensional projective subspace.
\end{proof}
We can now present
\begin{proof}[Proof of Theorem \ref{thm:top_Tverberg}]
  Consider a continuous map $f \fc \Delta^n \to Y$, where $Y$ is a metric space of covering dimension $d$ and $n=p(d+1)$. Consider the standard toric moment map $\Phi \fc \C P^n \to \Delta^n$, and let $g=f\circ\Phi$. Let $y_0 \in Y$ be the point whose existence is guaranteed by Theorem \ref{thm:cont_maps_CP_n_Y_d_fiber_proj_subspaces}. If $Z \subset \Delta^n$ is a $pd$-dimensional face, then $\Phi^{-1}(Z)$ is a $pd$-dimensional complex projective subspace, therefore
  $$\varnothing \neq g^{-1}(y_0)\cap \Phi^{-1}(Z) = \Phi^{-1}(f^{-1}(y_0))\cap\Phi^{-1}(Z) = \Phi^{-1}(f^{-1}(y_0)\cap Z)\,,$$
  whence $f^{-1}(y_0) \cap Z \neq \varnothing$, as claimed.
\end{proof}

\medskip

Next we prove Gromov's centerpoint theorem. We invite the reader to review the notion of rank, Definition \ref{def:ranks_algebras}.
\begin{thm} [Gromov, \cite{gromov2010singularities}]\label{thm:Gromov_central_point}
  Let $Y$ be a compact metric space of covering dimension $d$, let $A$ be a finite-dimensional algebra, and let $\mu$ be an $A$-IVM on $Y$. Then there is $y_0 \in Y$ such that
  $$\dim A/\mu\big(Y \setminus  \{y_0\}\big) \geq \rk_{d+1}A\,.$$
\end{thm}
\begin{proof} Put $r = \rk_{d+1}(A)$. By definition, $(A^{/r})^{*(d+1)} \neq 0$, therefore by Corollary \ref{thm-karas}, a centerpoint $y_0 \in Y$ of $\mu$ with respect to $A^{/r}$ satisfies $A^{/r} \not\subset \mu(Y \setminus \{y_0\})$. Therefore
$\dim A/\mu(Y \setminus \{y_0\}) \geq r$ by the definition of $A^{/r}$.\end{proof}

This result implies Gromov's Theorem \ref{thm:Gromov_torus}. For the proof we will need the following explicit calculation of ranks.

\begin{exam} \label{ex:rank_cohomology}(\cite[Section 4.1]{gromov2010singularities}) If $X$ is a closed oriented manifold, then Poincar\'e duality implies that any nonzero graded ideal in $H^*(X)$ must contain the orientation class $[X]$. If $X_1,\dots,X_d$ are closed oriented manifolds, define $m=\min_{1\leq i\leq d} \dim H^*(X_i)$, and $X = \prod_{i=1}^{d}X_i$. K\"unneth's formula implies that
$$H^*(X) = \bigotimes_{i=1}^d H^*(X_i)$$
as graded skew-commutative algebras. The natural inclusion map
$$\iota_i\fc H^*(X_i) \to H^*(X)\,, \qquad a \mapsto 1^{\otimes(i-1)} \otimes a \otimes 1^{\otimes(d-i)}$$
is a graded algebra morphism. If $K \subset H^*(X)$ is a graded ideal of codimension $<m$, then $\iota_i^{-1}(K)$ is a nonzero graded ideal in $H^*(X_i)$, which then contains $[X_i]$ by the above. It follows that $K$ contains $\iota_i([X_i]) = 1^{\otimes(i-1)} \otimes [X_i] \otimes 1^{\otimes(d-i)}$, and therefore so does $H^*(X)^{/m}$. Since
$$\prod_{i=1}^d\iota_i([X_i]) = [X_1]\otimes\dots\otimes[X_d] = [X] \neq 0\,,$$
it follows that
$$\rk_dH^*(X) \geq m=\min_{1\leq i\leq d} \dim H^*(X_i)\,.$$
If $n \geq p(d+1)$, then for the torus $\T^n = (\T^p)^d\times \T^{n-pd}$, we obtain
$$\rk_{d+1} H^*(\T^n) \geq \dim H^*(\T^p)=2^p\,.$$
\end{exam}

\begin{proof}[Proof of Theorem \ref{thm:Gromov_torus}]
  Recall that we have a map $f \fc \T^n \to Y$, where $Y$ is a metric space of covering dimension $d$ and $n \geq p(d+1)$. Without loss of generality assume that $f$ is onto and that $Y$ is compact. Let $\mu$ be the cohomology IVM on $\T^n$ and let $\nu = f_*\mu$. Theorem \ref{thm:Gromov_central_point} implies that there is $y_0 \in Y$ such that
  $$\dim \check H^*(\T^n)/\nu(Y \setminus  \{y_0\}) \geq \rk_{d+1} \check H^*(\T^n)\,.$$
  Example \ref{ex:rank_cohomology} implies that $\rk_{d+1}\check H^*(\T^n) \geq 2^p$, therefore
  $$\dim \check H^*(\T^n)/\mu(\T^n \setminus  f^{-1}(y_0)) \geq 2^p\,.$$
  By the definition in Example \ref{ex:coh_IVM} we have
  $$\mu(\T^n \setminus  f^{-1}(y_0)) = \ker\big(\check H^*(\T^n) \to \check H^*(f^{-1}(y_0))\big)\,,$$
  therefore
  $$\rk \big(\check H^*(\T^n) \to \check H^*(f^{-1}(y_0))\big) =\dim \check H^*(\T^n) - \dim \ker\big(\check H^*(\T^n) \to\check  H^*(f^{-1}(y_0))\big) \geq 2^p\,.$$
\end{proof}

We close this section with the proofs of Theorems \ref{thm:quantitative_nondisplaceable_fiber} and \ref{thm:symplectic_Tverberg}, for which we need the following observation on the relation between the pushforward construction, see Remark \ref{rem:pushforward_defined_more_generally}, and involutive maps.
\begin{prop}\label{prop:involutive_maps_push_IVQM_to_IVM}
  Let $(M,\omega)$ be a closed symplectic manifold and let $\pi \fc M \to B$ be a smooth involutive map. If $A$ is an algebra and $\tau$ is a $A$-IVQM on $M$, then $\pi_*\tau$ is an $A$-IVM on $B$.
\end{prop}
\begin{proof}
  All the axioms of an $A$-IVM are satisfied automatically, except multiplicativity. If $U,U'\subset B$ are open, then by Remark \ref{rem:preimages_involutive_maps} their preimages Poisson commute, therefore quasi-multiplicativity implies
  $$\pi_*\tau(U)*\pi_*\tau(U') = \tau(\pi^{-1}(U))*\tau(\pi^{-1}(U')) \subset \tau(\pi^{-1}(U)\cap\pi^{-1}(U')) = \pi_*\tau(U\cap U')\,,$$
  as required.
\end{proof}

\begin{rem}
  Remark \ref{rem:pushforward_defined_more_generally} applies here as well, that is the pushforward of the extension of $\tau$ to compact subsets of $M$ coincides with the extension to compact subsets of $B$ of the pushforward of $\tau$, because $M$ and $B$ are Hausdorff and $M$ is in addition compact.
\end{rem}

\begin{coroll}\label{cor:cont_involutive_maps_push_IVQM_to_IVM}
  Let $(M,\omega)$ be a closed symplectic manifold.  Then for any continuous involutive map $f \fc M \to Y$ the pushforward $f_*\tau$ of an $A$-IVQM $\tau$ is an $A$-IVM.
\end{coroll}

\begin{proof}
  Factor $f$ as $M \xrightarrow{\pi} B \xrightarrow{\ol f} Y$, where $\pi$ is a smooth involutive map. Then $\pi_*\tau$ is an $A$-IVM on $B$ according to Proposition \ref{prop:involutive_maps_push_IVQM_to_IVM}. Therefore $f_*\tau = \ol f_*(\pi_*\tau)$ is an $A$-IVM on $Y$.
\end{proof}

\begin{proof}[Proof of Theorem \ref{thm:quantitative_nondisplaceable_fiber}]Apply Theorem \ref{thm:Gromov_central_point} to the $A$-IVM $f_*\tau$.
\end{proof}

\begin{proof}[Proof of Theorem \ref{thm:symplectic_Tverberg}]
  Without loss of generality assume that $Y$ is compact and that $f$ is onto. Let $\mu = f_*\tau$ and note that this is an $A$-IVM on $Y$ thanks to Corollary \ref{cor:cont_involutive_maps_push_IVQM_to_IVM}. Theorem \ref{thm:central_point} implies that there is $y_0 \in Y$ contained in every compact $Y' \subset Y$ such that $I \subset \mu(Y')$. If $Z \in \cX_{I,\tau}$, then
  $$I \subset \tau(Z) \subset \tau(f^{-1}(f(Z))) = \mu(f(Z))\,,$$
  therefore $y_0 \in f(Z)$, as claimed.
\end{proof}

\medskip
\noindent
{\bf Acknowledgments:} We thank Shachar Carmeli, Liran Shaul and Moshe White for useful discussions about higher category theory, completions of modules, and
combinatorial geometry, respectively. We thank Pierre-Alexandre Mailhot, Jordan Payette, and Felix Schlenk for very useful remarks on the manuscript, Gleb Smirnov for pointing out that Theorem \ref{thm:cont_maps_CP_n_Y_d_fiber_proj_subspaces}, which subsumes Theorem \ref{thm:top_Tverberg}, is an interesting result in its own right, and Umut Varolg\"une\c s for helpful comments on Sections \ref{subsubsec-multi} and \ref{subsubsec-LagIVQM}. The second author was a postdoc at the Technion when work on this project began, and he would like to thank Michael Entov for the hosting and support. Finally, we wish to thank the anonymous referee for many useful remarks and suggestions, which helped us to improve the exposition.

\bibliographystyle{abbrv}
\bibliography{SH_IVM1019-submit}

\noindent
\begin{tabular}{l}
{\bf Adi Dickstein} \\
Tel Aviv University \\
School of Mathematical Sciences\\
69978, Tel Aviv, Israel\\
{\em E-mail:}  \texttt{adi.dickstein@gmail.com}
\end{tabular}

\medskip

\noindent
\begin{tabular}{l}
{\bf Yaniv Ganor} \\
Holon Institute of Technology \\
School of Mathematical Sciences\\
52 Golomb Street\\
POB 305, Holon, 5810201\\
Israel\\
{\em E-mail:}  \texttt{ganory@gmail.com}
\end{tabular}

\medskip

\noindent
\begin{tabular}{l}
{\bf Leonid Polterovich} \\
Tel Aviv University \\
School of Mathematical Sciences\\
69978, Tel Aviv, Israel\\
{\em E-mail:}  \texttt{polterov@tauex.tau.ac.il}
\end{tabular}

\medskip

\noindent
\begin{tabular}{l}
{\bf Frol Zapolsky} \\
University of Haifa \\
Department of Mathematics \\
Faculty of Natural Sciences \\
3498838, Haifa, Israel \\
\&\\
MISANU \\
Kneza Mihaila 36 \\
Belgrade 11001\\
Serbia\\
{\em E-mail:}  \texttt{frol.zapolsky@gmail.com}
\end{tabular}

\end{document}